\documentclass{amsart}
\usepackage[utf8]{inputenc}

\title{ Markov chains on hyperbolic-like groups and quasi-isometries}

\usepackage[english]{babel}

\usepackage{amsthm}
\usepackage{amsmath}
\usepackage{amssymb}

\usepackage{graphicx}
\usepackage{commath}
\usepackage{breqn}
\usepackage{xcolor}
\usepackage{comment}
\usepackage{hyperref}
\usepackage{tikz}
\usetikzlibrary{decorations.pathmorphing}
\usetikzlibrary{decorations.pathreplacing}
\usetikzlibrary{arrows}
\usetikzlibrary{arrows, arrows.meta}
\usepackage{caption}
\tikzset{snake it/.style={decorate, decoration=snake}}

\usepackage{lmodern}

\usepackage[left=2cm, right=2cm]{geometry}

\DeclareMathOperator{\diam}{diam}
\DeclareMathOperator{\Cay}{Cay}

\DeclareMathOperator{\Stab}{Stab}

\usepackage{bbold}

\usepackage{amssymb}
\usepackage{mathtools} 
\usepackage{float}

\theoremstyle{plain}
\newtheorem{lemma}{Lemma}[section]
\newtheorem{proposition}[lemma]{Proposition}
\newtheorem{corollary}[lemma]{Corollary}
\newtheorem{theorem}[lemma]{Theorem}
\newtheorem{example}[lemma]{Example}
\newtheorem{problem}[lemma]{Question}
\newtheorem{notation}[lemma]{Notation}
\newtheorem{assumptions}[lemma]{Assumption}
\newtheorem{claim}{Claim}

\newtheorem{remark}[lemma]{Remark}
\theoremstyle{definition}
\newtheorem{definition}[lemma]{Definition}

	\author[Antoine Goldsborough]{Antoine Goldsborough}
	\address{Department of Mathematics, Heriot-Watt University, Edinburgh, UK}
	\email{ag2017@hw.ac.uk}

\author[Alessandro Sisto]{Alessandro Sisto}
	\address{Department of Mathematics, Heriot-Watt University, Edinburgh, UK}
	\email{a.sisto@hw.ac.uk}

\begin{document}

\maketitle
\begin{abstract}
 We propose the study of Markov chains on groups as a ``quasi-isometry invariant" theory that encompasses random walks. In particular, we focus on certain classes of groups acting on hyperbolic spaces including (non-elementary) hyperbolic and relatively hyperbolic groups, acylindrically hyperbolic 3-manifold groups, as well as fundamental groups of certain graphs of groups with edge groups of subexponential growth. For those, we prove a linear progress result and various applications, and these lead to a Central Limit Theorem for random walks on groups quasi-isometric to the ones we consider.
\end{abstract}

%\tableofcontents

\section{Introduction}
Random walks on groups have been studied extensively since, at least, work of Kesten \cite{Kesten} in the fifties. In particular, there are various ways in which random walks have been examined in the context of geometric group theory, and most relevantly for us a substantial amount of work has been devoted to understanding random walks on hyperbolic groups and various generalisations (see below for some relevant references).

Given the strength of the connections between random walks and geometric group theory, it is natural to wonder how random walks interact with a crucial notion in geometric group theory, namely that of quasi-isometry. Unfortunately, random walks and quasi-isometries are not compatible, as given a random walk on a group $G$ and another group $H$ quasi-isometric to $G$, there is no "corresponding" random walk on $H$ in any meaningful sense.

The main goal of this paper is to suggest that this conflict can be resolved if one studies more general Markov chains rather than random walks. Indeed, more general Markov chains can be ``pushed forward" via quasi-isometries in a natural way. (We note that we cannot restrict to Markov chains that are finite-state in any meaningful way if we want this push-forward property; we explain this in Remark \ref{rmk:no_finite_state}.)

To illustrate the usefulness of this approach, we now state a theorem about random walks (not Markov chains) that can be deduced from our initial results on Markov chains on groups together with known results about random walks. The groups that we can consider are those satisfying either of two assumptions (Assumptions \ref{assump:superdiv} and \ref{assump:graph} below) which we do not state in this introduction, but rather we list some of the groups to which they apply (we note that we only consider finitely generated groups in this paper):

\begin{itemize}
    \item non-elementary hyperbolic and relatively hyperbolic groups,
    \item acylindrically hyperbolic 3-manifold groups,
    \item right-angled Artin groups whose defining graph is a tree of diameter at least 3,
    \item groups acting acylindrically and non-elementarily on a tree with nilpotent and undistorted edge stabilisers.
\end{itemize}

The theorem is a central limit theorem for the distance from the identity. More general random walks can be included, but we state it only for simple random walks only for now.

\begin{theorem}
\label{thm:clt_intro}
Let $G$ be a group quasi-isometric to some group satisfying either of Assumption \ref{assump:superdiv} or Assumption \ref{assump:graph} below (e.g. a group from the list above), fix a word metric $d_G$ on $G$, and let $(Z_n)$ be a simple-random walk on $G$.  Then there exist constants $L, \sigma >0$ such that we have the following convergence in distribution:  $$\frac{d_{G}(1,Z_n)-Ln}{\sigma^2\sqrt{n}} \to \mathcal{N}(0,1),$$
where $\mathcal N(0,1)$ denotes the normal distribution with mean $0$ and variance $1$.
\end{theorem}

The interesting feature of the theorem is that it is a theorem about random walks on a group for which the only assumption is that it is quasi-isometric to a group of a certain kind. As we explain below, the proof uses more general Markov chains than random walks, and this should give an idea of the potential of the Markov chain approach.

We should note that the groups satisfying Assumptions \ref{assump:superdiv} or Assumption \ref{assump:graph} are acylindrically hyperbolic, and there is a central limit theorem for random walks on acylindrically hyperbolic groups \cite{MathieuSisto}. However, it is not known whether being acylindrically hyperbolic is a quasi-isometry invariant, and in fact it is not even known whether it is a commensurability invariant, see \cite[Question 2]{MinasyanOsinErratum} for a discussion of this.

We also note that the theorem does not hold for Markov chains, as those might not even have well-defined drift; an example of this is described in work in progress of the first author. This should also be taken as a warning that Markov chains can have more exotic behaviour than one might expect based on random walks.

\subsection*{Literature on Markov chains and quasi-isometries} Before discussing our results, we should mention that Markov chains on quasi-isometric graphs have been considered in the literature before, see e.g. \cite{benjaminiLiouville,soardiyamasakiParabolicindex,hermonperes,DingPeres,BerryMixingtimeRobust} (the term ``rough isometry" is used in this context, rather than ``quasi-isometry"). It is shown in \cite{kanai, SaloffCosteCoulhon} that transience (resp. recurrence) of a graph is a quasi-isometric invariant. We are however not aware of any paper that aims to study quasi-isometric \emph{groups} from this perspective.

While, for example, \cite{soardiyamasakiParabolicindex} studies a property of Markov chains which turns out to be quasi-isometry invariant, most of the papers mentioned above actually show that certain properties of Markov chains are not quasi-isometry invariants in a suitable sense. For example \cite{hermonperes, DingPeres} shows that the mixing time can vary drastically for (finite) quasi-isometric graphs (with bounded quasi-isometry constants).

In a different direction, we should also mention that there are results in the literature on random walks on quasi-isometric groups. For example, the fact that (non-)amenability is a quasi-isometry invariant implies that having spectral radius for random walks less than 1 or equal to 1 are quasi-isometry invariants; a stronger result for amenable groups is proved in \cite{SaloffCostePittet}.

\subsection*{Linear progress}
The groups we will be dealing with all come with an action on a hyperbolic space. In the case of random walks, it turns out that there is one result in particular that is a very useful starting point for other applications, namely linear progress. This is a result of Maher-Tiozzo \cite{maher2015random} who proved that a (non-elementary) random walk on a group acting on a hyperbolic space $X$ makes linear progress in $X$. This result feeds into the proof of several others, for example on the translation length \cite{maher2015random}, random subgroups \cite{MS-rndsubgp,Abbott2021RandomWA}, various kinds of projections \cite{SistoTaylor}, and deviation from quasi-geodesics \cite{MathieuSisto}, as well as the already mentioned central limit theorem for acylindrically hyperbolic groups, and results on counting measures, see e.g. \cite{GTT:counting_lox}.

Therefore, since we want to initiate the study of more general Markov chains on groups acting on hyperbolic spaces, it is natural to start by generalising the linear progress result. Unfortunately, one cannot just use the proof from \cite{maher2015random}, and not even the proof from \cite{MathieuSisto} for the case where the action on the hyperbolic space is acylindrical, or the related arguments from, e.g. \cite{sunderlandlinearprogress,  sisto2013contracting, gouezel2021exponential, boulanger2020large}. The reasons for this are, for the arguments that use boundary theory, the fact that it is not even clear what should play the role of stationary measures, and for the arguments that do not, the fact that, when decomposing a Markov path in 2 or more parts, subsequent parts are not independent from the previous ones. We note moreover that since the Markov chains we consider are not finite-state in any meaningful way, they are not the same kind of Markov chains considered in e.g. \cite{CalegariFujiwara,GTT:counting_lox}, and the arguments from those papers do not apply here either.

Still, the theorem below gives linear progress in $X$ in certain cases. The notion of tameness is defined and discussed in Subsection \ref{tameMarkov}, but it is not so important for this general discussion so we omit it here.

\begin{theorem}
\label{linear progress}
Let $G$ be a group acting on a corresponding hyperbolic space $X$ as in either of Assumption \ref{assump:superdiv} or Assumption \ref{assump:graph}.

Let $(w^o_n)_{n\geq 0}$ be a tame Markov chain in $G$, starting at an element $o\in G$ and consider any basepoint $x_0 \in X$. Then there exist constants $L, C>0$ such that

$$\mathbb{P}\Big[ od_{X}(x_0,w^{o}_nx_0) \geq L n\Big] \geq  1 - C e^{-n/C}.$$
\end{theorem}

Our proof, while inspired by that in \cite{MathieuSisto}, is different from either of the aforementioned proofs, and in fact we even need additional assumptions of a geometric nature. It is possible that our techniques can be pushed to deal with more cases, and we elaborate on this in the outline of the paper below.

In particular, it would be very interesting to deal with general acylindrically hyperbolic groups, or at least with some other subfamilies. We record this in the following question.

\begin{problem}
    Does the conclusion of Theorem \ref{linear progress} hold when $G$ is only assumed to act acylindrically and non-elementarily on $X$? What if in addition $G$ is a hierarchically hyperbolic group?
\end{problem}

Just as in the case of random walks, we can deduce various other results from Theorem \ref{linear progress}. While, as mentioned above, we had to use different techniques than for random walks to prove Theorem \ref{linear progress}, for the following applications we can use known arguments. The first item says in particular that (under the assumptions of the theorem) loxodromic elements are generic with respect to Markov chains.

\begin{theorem}
\label{thm:qgeod_intro}
Let $G$ be a group acting on a hyperbolic space $X$, with a fixed word metric $d_G$.
Let $(w^o_n)_{n\geq 0}$ be a tame Markov chain satisfying the conclusion of Theorem \ref{linear progress}. 
Then:
\begin{itemize}
\item (Translation length) There exist constants $L_1$, $C_1>0$ such that $$\mathbb{P}\big[ \tau(w^{o}_n) >L_1 n \big] \geq 1-C_1e^{-n / C_1} $$ where $\tau(h)$ is the translation length in $X$ of the element $h \in G$. 

 \item (Deviation from quasi-geodesics) For every $D>0$, there exists $C_2>0$ such that for all $l>0$ and $0 \leq k \leq n$ we have :
$$
\mathbb{P}\left[ \sup_{\alpha \in QG_{D}(o,w^{o}_n)} d_{G}(w^{o}_k, \alpha) \geq l \right] \leq C_2 e^{-l/C_1}
$$ where $QG_{D}(a,b)$ is the set of all $(D,D)$- quasi-geodesics from $a$ to $b$ (with respect to $d_{G}$).

 \end{itemize}
\end{theorem}

Note that the second item makes no reference to $X$, and actually this will be the crucial statement about Markov chains that we need to prove the central limit theorem, Theorem \ref{thm:clt_intro}. The sketch of proof is as follows. Consider a group $G$ and a random walk on it as in the statement. Then the random walk can be pushed-forward to a Markov chain (probably not a random walk) on another group $H$ quasi-isometric to $G$. In $H$ we know that, in the sense of the second item above, Markov paths stay close to quasi-geodesics. We deduce that sample paths of the random walk in $G$ stay close (in the same sense) to geodesics, and at this point we can simply apply results from  \cite{MathieuSisto} to conclude.

\subsection{Outline of the paper}

In Section \ref{preliminaries}, we look at some basic definitions as well as discussing the setup for Markov chains and random walks. This allows us to consider a Markov chain under some assumptions, which we shall call a tame Markov chain. We show that a random walk (satisfying standard conditions) 'pushed-forward' via a quasi-isometry gives rise to a tame Markov chain, see Lemma \ref{pushforwardistame}. This, as explained above, is the key fact about Markov chains that we are interested in.

Section \ref{sec:WPD} has two purposes. First, we introduce the notion of super-divergent element, which is the key notion for one of the assumptions under which we can prove linear progress. Second, we review and set up notation regarding loxodromic WPD element, as well as showing that super-divergent elements are WPD.

In Section \ref{sec:assump} we state the two sets of assumptions under which we can prove linear progress, and discuss various examples.

In Section \ref{geometricarguments} we state Proposition \ref{axiom}, which we then prove under either of the assumptions for linear progress. This proposition is a key intermediate step in the proof of linear progress, and we highlight it because of the following two reasons. First, its conclusion can be stated whenever one has a group acting on a hyperbolic space with a loxodromic WPD element. So, potentially the proposition could be true for all acylindrically hyperbolic groups. Second, in Section \ref{probabilisticarguments} we show that whenever the conclusion of the proposition holds, then linear progress also holds. Hopefully Proposition \ref{axiom} will be used in the future to cover more groups and group actions.

We note that the proof of Proposition \ref{axiom} in both cases uses a key idea  of the form "if you have a path that, in the Cayley graph, moves away from certain specified parts, then a certain projection moves slowly". See for example Lemma \ref{20francs} and Lemma \ref{o(log)}.

Finally, in Section \ref{applications} we prove all the applications of linear progress mentioned above.

\subsection*{Acknowledgements}

Part of the content of this paper originated from discussions with Dominik Gruber, without whom this paper would not exist, and the authors are very thankful for his contributions.

The authors would also like to thank Adrien Boulanger for pointing out useful references and Cagri Sert for useful remarks that led to clarifications in the introduction. The authors would also like to thank an anonymous referee for useful comments.

The work of the first author was supported by the EPSRC-UKRI studentship EP/V520044/1.

\section{Tame Markov chains}
\label{preliminaries}

 \subsection{Basic geometric group theory definitions}

  We recall some of the main notions from geometric group theory that will be used throughout the paper.

 	A geodesic metric space $(X,d_X)$ is \textit{$\delta$-hyperbolic} for some $\delta>0$ (or simply \textit{hyperbolic}) if for any geodesic triangle in $X$ the following holds: any side of the triangle is contained in the closed $\delta$-neighborhood of the union of the other two sides. (We recall that the $M$-neighborhood of a subset $A$ of the metric space $(X,d)$ is $\mathcal N_M(A) = \{ x \in X : d(x,A) \leq M \}$, where $ d(x,A)=\inf \{ d(x,a) : a \in A  \}$.)

 A finitely generated group is called \textit{hyperbolic} if its Cayley graph with respect to some (equivalently, any) finite generating set is a hyperbolic metric space, when endowed with the graph metric.
 
%  	We will write $(x,y)_z$ for the Gromov product of $x$ and $y$ with respect to $z$.
% 	This product is defined as
%  	 \[
% 	(x,y)_z= \frac{1}{2} \left(d_X(x,z)+d_X(y,z)-d_X(x,y)\right).
% 	\]

For two metric spaces $(X, d_X)$ and $(Y, d_Y)$, we say that a function $f:X \to Y$ is a \textit{quasi-isometry} if there exist constants $A\geq 1, B \geq 0$ and $C \geq 0$ such that:
\begin{itemize}
    \item For all $x,y \in X$ : $d_{X}(x,y)/A -B \leq d_Y(f(x),f(y)) \leq A d_{X}(x,y)+B$,
   \item $\forall z \in Y$  $ \exists x \in X : d_{Y}(z,f(x))\leq C$.
 \end{itemize}

Relatedly, a function $f$ as above is $L$-coarsely Lipschitz if $d_Y(f(x),f(y))\leq Ld_X(x,y)$ for all $x,y\in X$.

% Fix a word metric $d_G$ on a group $G$, with respect to a finite generating set $S$. Let $p$ be a path in the Cayley graph $\Cay(G,S)$. 
% The \textit{length} of $p$ is the number of edges of $p$ which we write as $\ell_S(p)$. If the generating set has been fixed, we will drop the subscript. \\

% For a fixed word metric $d_G$, we can define the $M$-\textit{neighborhood} of a subset $A \subseteq G$ as \[
% \mathcal N_M(A) = \{ g \in G : d_G(g,A) \leq M \}
% \]
%where $ d_G(g,A)=\inf \{ d_G(g,a) : a \in A  \} $. \\

The following is a well-known exercise in hyperbolic-geometry; we omit the proof. First, we recall the relevant notions. A subset $A$ of a geodesic metric space $X$ is \emph{quasi-convex} if there exists $C\geq 0$ such that all geodesics with endpoints in $A$ are contained in the $C$-neighborhood of $A$. Given a subset $A$ of a metric space $X$, we call a map $\pi:X\to A$ a closest-point projection if $d_X(x,\pi(x))=d_X(x,A)$ for all $x\in X$.

\begin{lemma}
\label{lem:exo_hyperbolicspace}

Let $X$ be a $\delta$-hyperbolic space. Let $Q$ be a quasi-convex set and $\pi_Q:X \rightarrow Q$ a closest-point projection.
%map, that is to say a map such that $\pi_{Q}\big|_{Q}$ is the identity.
Then there exists a constant $R>0$ only depending on $\delta$ and the quasi-convexity constant such that the following hold.
\begin{enumerate}
    \item $\pi_Q$ is $R$-coarsely Lipschitz.
    \item For all $x,y \in X$, if $d_{X}(\pi_Q(x), \pi_Q(y)) \geq R$ then there are points $m_1, m_2 \in [x,y]$ such that $d_X(m_1,\pi_Q(x)) \leq R$ and $d_X(m_2,\pi_Q(y)) \leq R$   where $[x,y]$ is a geodesic between $x$ and $y$. Further, the subgeodesic of $[x,y]$ between $m_1$ and $m_2$ lies in the $R$-neighborhood of $Q$.
\end{enumerate}

\end{lemma}

\subsection{General Markov chains}

We explain the formal setup of general Markov chains and then the specific case of random walks. This setup is based on Chapter 1 from \cite{woess_2000}. \\
 
Given a finitely generated group $G$, we define a \textit{Markov chain on } $G$.  This Markov chain has \textit{state space} $G$ and \textit{transition operator} $P=(p_G(g,h))_{g,h \in G}$, with the requirement that each $p_G(g,h)$ is non-negative and for each $g\in G$ we have $\sum_{h\in G} p_G(g,h)=1$. Each $p_G(g,h)$ represents the probability of moving from $g$ to $h$ in one step. When the group being considered is clear, we will drop the subscript. \\

When a basepoint $o\in G$ is fixed, a Markov chain gives a sequence of random variables $(w^{o}_n)_{n \geq 0}$ which describe the position of the Markov chain starting at $o$ and after $n$ steps. To model this we can choose as probability space the \textit{path space} $\Omega = G^{\mathbb{N}}$ equipped with the product $\sigma$-algebra arising from the discrete $\sigma$-algebra on $G$. We equip  $\Omega$ with the probability measure $\mathbb P_{o}$ defined on cylinder sets (and extended via the Kolmogorov extension theorem) by$$ \mathbb P_o \Big[\ \{(h_i)\in G^{\mathbb{N}} : h_j = g_j\ \forall j=0,\dots n\}\ \Big]  =\delta_{o}(g_0)p(g_0,g_1)\cdots p(g_{n-1}, g_n)$$
for all given $g_0,\dots, g_n\in G$.

Then $w^o_n$ is just the $n$-th projection $\Omega \rightarrow G$.

Note that a reformulation of the above is that for all given $o,g_i \in G$ we have
$$\mathbb P \Big[ w^{o}_0=g_0,w^{o}_1=g_1 \cdots ,w^{o}_n=g_n \Big]=\delta_{o}(g_0)p(g_0,g_1)\cdots p(g_{n-1}, g_n).$$

 We will often consider $\mathbb{P}[w^g_n=h]$ which is the probability that a Markov chain starting at the element $g$ ends up on the element $h$ after $n$ steps. Note that this can be written as a certain sum over all sequences of points in $G$ of length $n$ that start at $o$ and end in $g$, but we will never use this sum explicitly. Instead, similarly to the above, we note that the following holds for all $o,g,g_i\in G$ and $k\leq n$ (using a similar argument summing over suitable sequences):
 
 \begin{equation}
 \label{eqn:change_of_basepoint}
      \mathbb P_o \Big[ w^{o}_k=g_0,w^{o}_{k+1}=g_1 \cdots ,w^{o}_n=g_{n-k} | w^o_k=h\Big]= \mathbb P_o \Big[ w^{h}_0=g_0,w^{h}_{1}=g_1 \cdots ,w^{h}_{n-k}=g_{n-k} \Big].
 \end{equation}

%We write $p^{(n)}(g,h)= \mathbb{P}(w^g_n=h)$ the probability that a Markov chain starting at $g$ is at the element $h$ after $n$ steps.

\subsubsection{A note on a Markov property}
\label{note_Markov_property}

Usually, Markov chains are defined as stochastic processes satisfying the \textit{Markov property}: For all $h,g_i \in G$ and $n$: 
$$\mathbb{P}\big[ w^h_n=g_n \quad \vert w^h_{n-1}=g_{n-1}, \dots, w^h_0=g_0 \big]=\mathbb{P}\big[ w^h_n=g_n \quad \vert w^h_{n-1}=g_{n-1}\big].$$

This is well-known to be equivalent to the above set up of Markov chains, see e.g. \cite[Theorem 1.1.1]{markovnorris}. We will also need a similar property, which we record in the following lemma.

% below as Lemma \ref{lem:Markov_property_set}. Before proving this lemma, we note the following: for all $k \leq n$ and $g,h,h' \in G$ we have:
% \[
% \mathbb{P} \big[ w^g_n=h \vert w^g_{n-k}=h' \big] = \mathbb {P}\big[ w^{h'}_k=h\big]
% \]

% We fix $n$ and show this equality by induction on $k$. The base case $k=1$ is clear: $$\mathbb P \big [ w^g_n=h \vert w^g_{n-1}=h' \big ] =p(h',h)=\mathbb P \big[ w^{h'}_1=h \big]$$ by definition. \\

% Assuming the inductive hypothesis is true, we now show that the statement is true for $k+1$: 

% \begin{align*}
%     \begin{split}
%         \mathbb P \big[ w^g_n=h \vert w^g_{n-(k+1)}=h'\big] &= \sum_{t \in G } \mathbb P \big[ w^g_n=h \vert w^g_{n-k}=t\big] \mathbb P \big[  w^g_{n-k}=t \vert w^g_{n-(k+1)}=h'\big] \\
%         &= \sum_{t \in G} \mathbb P \big[ w^t_k=h\big]p(h',t) \\
%         &= \mathbb P \big[ w^{h'}_{k+1}=h \big]
%     \end{split}
% \end{align*}

% where we use the inductive hypothesis to get from line 1 to line 2, obtaining the desired result.

% This observation now allows us to prove the following property of general Markov chains, which will be crucial throughout this paper.

\begin{lemma}
\label{lem:Markov_property_set}
	For all $k \leq n$, all $A \subseteq G^{n-k+1}$,  and $o,h \in G$ we have 
	
	$$\mathbb P \Big[ (w^{o}_k,w^{o}_{k+1},\cdots, w^{o}_n) \in A \quad \Big\vert \quad w^{o}_{k}=h\Big]= \mathbb P \Big[ (w^{h}_0, w^{h}_2, \cdots, w^{h}_{n-k}) \in A \Big].$$

Moreover, for a subset $A \subseteq G$ and for all $o,h \in G$ and $k \leq n$: $$ \mathbb P \Big[  w^{o}_n \in A  \Big\vert w^{o}_{n-k} =h\Big] =\mathbb P \Big[ w^{h}_k \in A\Big] .$$
\end{lemma}

\begin{proof}
%Fix $n$ and $k \leq n$. 
This just follows from summing Equation \eqref{eqn:change_of_basepoint} over all elements of $A$.

For the ``moreover" part, we can deduce it from the first statement applied to the set of all sequences of appropriate length where the last entry is in $A$.
\end{proof}

With a slight abuse of notation, we will often write ``Let $(w^o_n)$ be a Markov chain", and similar, to mean that we fix transition probabilities and denote $w^o_n$ the corresponding random variables.\\

%Note that, given a Markov chain on $G$, for all $g,h,k\in G$ and $n_1,n_2\geq 0$ we have $$\mathbb{P}(w^g_{n_1+n_2}= k | w^g_{n_1}=h ) = \mathbb{P}(w^g_{n_1}= h) \mathbb{P}(w^h_{n_2}= k).$$

%$\forall h_i,g \in G : \mathbb{P}\big(w^g_{n+1}=h \vert w^g_{n}=h_n, \dots, w^g_0=h_0\big) =\mathbb{P}(w^g_{n+1}=h\vert w^g_n=h_n)$

%This describes the Markov chain starting at $o$, when $\Omega$ is equipped with the probability measure given by the Kolmogorov extension theorem by \[ \mathbb{P} \big[ w^{o}_0=g_0, w^{o}_1=g_1, \cdots, w^{o}_n=g_n\big] =\delta_{o}(g_0)p(g_0,g_1)\cdots p(g_{n-1}, g_n)
%\]

% If $G$ acts on a hyperbolic space $X$, we will be considering a Markov chain in $G$ and pushing it via the orbit map $G \rightarrow X, g \rightarrow gx_0$, for some basepoint $x_0 \in X $. We can therefore consider the sequence $(g_nx_0)_{n \in \mathbb{N}} $ where $(g_n)_{n\in \mathbb{N}} \in G^{\mathbb{N}}$ is given by the Markov chain defined in $G$ and the \textit{location at time $n$} is $w_nx_0=g_1g_2\cdots g_n x_0$. When it is clear that we are working in $X$, we will drop the $x_0$.\\

\subsubsection{Random walks}

A random walk is an equivariant Markov chain, meaning that for all $x,g,y \in G$ we have
$$p_G(x,y)=p_G(gx,gy).$$

Equivalently, for all $x,y\in G$ we have

$$
p_G(x,y) = \mu(x^{-1}y)
$$ 

for some fixed probability measure $\mu$ on $G$, which is called the \textit{driving measure of the random walk}. In the case of a random walk, the $n$-step transition probabilities are obtained by $$
p^{(n)}(x,y)=\mu^{(n)}(x^{-1}y)
$$
where $\mu^{(n)}$ is the $n$-fold convolution of $\mu$ with itself. 

Due to equivariance, the starting point of a random walk is often not as important, while for Markov chains we will have to ensure that all our statements are ``uniform" over all choices of starting point. In proofs, we will sometimes write $w_n$ instead of $w^{o}_n$ Markov path starting at a fixed basepoint $o$, but only when it is safe to do so.  \\

% Following Woess, we can define a random walk on a group $G$. Let $\mu$ be a probability measure on $G$. The \textit{(right) random walk} on $G$, with law $\mu$ is the Markov chain with state space $G$ and transition probability \[
% p_G(x,y) = \mu(x^{-1}y)
% \] 
% We note that this is just the Markov chain with identical distribution $\mu$ for all $g,h \in G$.
% Again, we may also use the product space $(G, \mu)^{\mathbb{N}}$ to obtain an equivalent model of $(Z_n)_n$ : the $n$th projection $X_n$ of $G^{\mathbb{N}}$ onto $G$ constitutes a sequence of independent $G$-valued random variables with common distribution $\mu$. The right random walk starting at $x \in G$ is \[
% Z_n= xX_1X_2\cdots X_n
% \]

\subsubsection{Push-forwards}

Random walks can be pushed-forward via group homomorphisms, while we are interested in the fact that Markov chains can be pushed forward via bijections (or more general maps, but we will not need it in this paper). We now make this precise in the following definition.

\begin{definition}
\label{defn:push_forward}
Consider now two finitely generated groups $G,H$ and $ f: H \rightarrow G$ a bijective map (which later on will be a quasi-isometry). Given a Markov chain on $H$, we define a Markov chain on $G$ given by  $$p_{G}(g,h) = p_{H}(f^{-1}(g), f^{-1}(h)) $$ for all $g,h \in G$.

We will call this new Markov chain the \textit{push-forward} via $f$. We note that in general this is not a random walk even when the Markov chain on $H$ is a random walk.
\end{definition}

This can be rephrased in terms of the corresponding random variables as follows.

\begin{lemma}
\label{lem:push_variable}
Let $(w^o_n)$ be a Markov chain on a group $H$ and let $(z^p_n)$ be its push-forward via some bijection $f:H \to G$. Then for all $g,h \in H$ we have
$$ \mathbb{P} \big[ w_{n}^{g} =h\big] =\mathbb{P} \big[ z_{n}^{f(g)} =f(h)\big].$$

Similarly, for all $A \subseteq H^{n}$ and $g \in H$ we have
$$ \mathbb P \big[ (w^g_0,\dots,w^g_{n-1}) \in A \big] =\mathbb P \big[ (z^{f(g)}_0,\dots,z^{f(g)}_{n-1})\in f^{(n)}(A)\big],$$
where $f^{(n)}:H^n\to G^n$ is the bijection obtained applying $f$ to each coordinate.
\end{lemma}

\begin{proof}
We proceed by induction on $n$, and in fact we show by induction that, given $\bar h=(h_0,\dots,h_{n-1})\in H^n$, $h\in H$, and $n$ we have
$$ \mathbb P \big[ (w^h_0,\dots,w^h_{n-1}) =\bar h \big] =\mathbb P \big[ (z^{f(h)}_0,\dots,z^{f(h)}_{n-1}) = f^{(n)}(\bar h)\big].$$

Given this, the final statement can be obtained simply by summing over all elements of $A$, and the first statement is a consequence of the final one.

For $n=0$, the equality is clear since the left-hand side is $1$ if $h=h_0$ and $0$ otherwise, and similarly for the right-hand side.

For $n=1$, the equality follows directly from the definition of the push-forward Markov chain, since the probability on both sides boils down to a transition probability.

Assume that the equality is true for $n\geq 1$, we will show it for $n+1$. Fix $\bar h=(h_0,\dots,h_{n})\in H^{n+1}$ and let $\bar h'=(h_0,\dots,h_{n-1})\in H^n$.

We compute

\begin{align*}
    \begin{split}
       \mathbb{P} \big[ (w^h_0,\dots,w_{n}^{h}) =\bar h\big] &= \mathbb P \big[w^{h}_{n} = h_{n} \vert (w^h_0,\dots,w_{n-1}^{h}) =\bar h' \big]\ \mathbb P \big[(w^h_0,\dots,w_{n-1}^{h}) =\bar h'\big] \\
       & = \mathbb P \big[ w^{h_{n-1}}_1=h_{n}\big]\ \mathbb P \big[ (w^h_0,\dots,w_{n-1}^{h}) =\bar h' \big] \\
       &=  \mathbb{P} \big[ z_{1}^{f(h_{n-1})} =f(h_{n})\big] \ \mathbb P \big[ (z^{f(h)}_0,\dots,z_{n-1}^{f(h)}) =f^{(n)}(\bar h') \big] \\
       & = \mathbb P \big[z^{f(h)}_{n} = f(h_{n}) \vert (z^{f(h)}_0,\dots,z_{n-1}^{f(h)}) =f^{(n)}(\bar h') \big]\ \mathbb P \big[ (z^{f(h)}_0,\dots,z_{n-1}^{f(h)}) =f^{(n)}(\bar h') \big]\\
       & =\mathbb P \big[ (z^{f(h)}_0,\dots,z_{n}^{f(h)}) =f^{(n+1)}(\bar h)],
    \end{split}
\end{align*}

% We compute

% \begin{align*}
%     \begin{split}
%       \mathbb{P} \big[ w_{n+1}^{g} =h\big] &= \sum_{t \in G} \mathbb P \big[w^{g}_{n+1} = h \vert w^{g}_n = t \big]\ \mathbb P \big[ w^{g}_n=t\big] \\
%       & = \sum_{t \in G} \mathbb P \big[ w^{t}_1=h\big]\ \mathbb P \big[ w^{g}_n=t \big] \\
%       &= \sum_{t \in G} \mathbb{P} \big[ z_{1}^{f(t)} =f(h)\big] \ \mathbb P \big[ z^{f(g)}_n=f(t)\big] \\
%       & =\sum_{s \in H} \mathbb{P} \big[ z_{n+1}^{f(g)} =f(h) \vert z^{f(g)}_n=s\big]\ \mathbb P \big[ z_n^{f(g)}=s\big] \\
%       & =\mathbb P_H \big[ z_{n+1}^{f(g)}=f(h) \big],
%     \end{split}
% \end{align*}

where we use the Markov property in the second and fourth equality, as well as the inductive hypothesis in the third equality.

%We also use the fact that $f$ is a bijection to change the sum from all elements of $G$ to all elements of $H$.
\end{proof}

% Throughout, we will be considering both \textit{Markov random elements} and \textit{Markov random path}. 

\subsection{Tame Markov chains}
\label{tameMarkov}

% 	We say that a Markov chain (or random walk) makes \textit{linear progress} in $X$ if it satisfies the conclusion Theorem \ref{linear progress}.  

% We recall the definition of the spectral radius for a Markov chain (see e.g. \cite[Equation (1.8)]{woess_2000}):
% \[
% \rho := \limsup_{n \to + \infty} (p^{(2n)}(x, x))^{1/2n}
% \]

% for any element $x$ of the group. This is independent of the choice of $x$. 

% We recall the definition of a strongly reversible Markov chain (see e.g. \cite[Definition 3.4]{woess_2000}): 

% \begin{definition}
% \label{stronglyreversible}
% 	A Markov chain on a group $G$ is $\textit{ strongly reversible}$ if there is a constant $A>0$ such that
% $$p_G(x,y) \leq A\ p_G(y,x)$$ for all $x,y \in G$.
% \end{definition} 

\begin{definition}[Tame Markov chain]
\label{defn:tame}

We shall say that a Markov chain on $G$ is \textit{tame} if it satisfies the following conditions:\\
	\begin{enumerate}
	
    \item\label{item:bounded_jumps} {\bf (Bounded jumps)} There exists a finite set $S\subseteq G$ such that $p(g,h)=\mathbb P[w^g_1=h]=0$ if $h\notin gS$.
    \item\label{item:non-amen} {\bf (Non-amenability)} There exist $A>0$ and $\rho<1$ such that for all $x,y\in G$ and $n\geq 0$ we have
    $$\mathbb P[w^x_n=y]\leq A\rho^n.$$
  \item\label{item:irred} {\bf (Irreducibility)} For all $s \in G$ there exist constants $\epsilon_s, K_s>0$ such that for all $g \in G$ we have
  $$\mathbb P[w^g_k=gs] \geq \epsilon_s$$
  for some $k \leq K_s$.
%   \item\label{item:reverse}  {\bf (Strong reversibility)} There is a constant $A>0$ such that $p_G(x,y) \leq A p_G(y,x)$ for all $x,y \in G$.
	\end{enumerate}
	
%	where we recall that $p^{(n)}(x,y)=\mathbb P \big[ w^x_{n}=y\big]$.
\end{definition}
%   We note that for a finitely generated group $G= \langle S \rangle$ if a Markov chain satisfies point (3) then it is  \textit{uniformly irreducible} in the sense of  Woess (\cite{woess_2000} page 8) : there exist $\epsilon_0>0$ and $U< \infty$ such that \[
%   x \sim y \quad \text{implies} \quad  p^{(k)}(x,y) \geq \epsilon_0 \quad \text{for some} \quad  k \leq U
%   \]

 %Indeed, we can take $\epsilon_0=\min \{\epsilon_s : s\in S\}$ and $U=\max \{K_s : s\in S\}$.\\
 
 \begin{remark}
 \label{rmk:bounded_jumps}
 	Once we have fixed a word metric $d_G$ on $G$ then the assumption of Bounded jumps \ref{defn:tame}-\eqref{item:bounded_jumps} is equivalent to the following: There exists a constant $K>0$ such that for all $n \in \mathbb N$ and starting point $o \in G$ we have $d_G(w^o_n, w^o_{n+1}) \leq K$.
\end{remark}

 \begin{remark}\label{rmk:no_finite_state}
 (We cannot reduce to the finite-state case.) One can show that there are uncountably many tame Markov chains on a free group with all transition probabilities of the form $i/10$. This boils down to constructing uncountably many labellings on the edges of a standard Cayley graph with the property that for each vertex the sum of the labels emanating from it is 1, and all edges have weight at least 1/10. Note that, instead, there are only countably many random walks satisfying the same requirement on transition probabilities. What is more, the same is true for finite-state Markov chains, again with the same requirement, and this is one way to see that the study of tame Markov chains cannot be reduced to that of finite-state Markov chains.
 
 In fact, one can even construct uncountably many tame Markov chains on the free group that are all push-forwards of the same random walk via \emph{isometries} of the standard Cayley graph, using a similar strategy. Therefore, for our purposes of establishing a quasi-isometry invariant theory, we cannot reduce to finite-state Markov chains, even if we further restricted the notion of tameness in some way.
 \end{remark}

% The following lemma is proven in Woess only using item \eqref{item:reverse} of Definition \ref{defn:tame}.

% \begin{lemma}\cite[Lemma 8.1-(b)]{woess_2000}
% 	For a tame Markov chain we have $p^{(n)}(x,y) < A \rho^{n} $ for all $x,y \in X$ where $A$ is as in \ref{defn:tame}-\eqref{item:reverse}.
% \end{lemma}

% We recall the following theorem from Whyte \cite{whyte}. As all the groups we will be interested in are non-amenable, this allows us to consider bijective quasi-isometries instead of all quasi-isometries. 

% \begin{theorem}[\cite{whyte}]
% 	Every quasi-isometry between two non-amenable groups lies at finite distance from a bijection.
% 	\end{theorem}

% For a fixed probability measure $\mu$, we say that the support of $\mu$ \textit{generates} $G$ \textit{as a semigroup} if 
% \[ \bigcup_{n\geq 1} (\supp(\mu))^n =G.
% \]

In the following lemma we consider a bijective quasi-isometry. Due to a result of Whyte \cite[Theorem 2]{whyte} every quasi-isometry between two non-amenable groups lies at finite distance from a bijection, so bijectivity will not be a real constraint for us.

Note also that the lemma applies with $G=H$ and $f$ the identity map, and in this case it says that random walks with suitable driving measures on non-amenable groups are tame Markov chains.

\begin{lemma}
\label{pushforwardistame}
% Let $H$ be a finitely generated non-amenable group. Let $g: H \rightarrow G$ be a quasi-isometry at finite distance from
Let $G,H$ be finitely generated non-amenable groups and let $f:H \to G$ be a bijective quasi-isometry. Let $\mu$ be the driving measure for a random walk on $H$ such that $\mu$ has finite support and generates $H$ as a semigroup. Then the push-forward Markov chain on $G$ (see Definition \ref{defn:push_forward})
%with transition probability defined by
%$$p_G(g,h):=p_H(f^{-1}(g),f^{-1}(h))= \mu \Big( (f^{-1}(g))^{-1} f^{-1}(h) \Big)$$
is a tame Markov chain. 
\end{lemma}
	
\begin{proof}
For convenience, we recall that the transition probabilities  defining the Markov chains under consideration are
$$p_G(g,h)=p_H(f^{-1}(g),f^{-1}(h))= \mu \Big( (f^{-1}(g))^{-1} f^{-1}(h) \Big).$$

\eqref{item:bounded_jumps} We first show that the Markov chain on $G$ has bounded jumps. Let $S'$ be the support of $\mu$, and note that there exists a finite subset $S$ of $G$ such that for all $h\in H$ we have $f(hS')\subseteq f(h)S$. Indeed, since $f$ is a quasi-isometry and $S'$ is finite, $f(hS')$ lies in a ball of uniformly bounded radius around $f(h)$ in any fixed word metric on $G$.

Thus, for any $g,h\in G$ we have $p_G(g,h)=0$ unless $h\in gS$, as required.

%We first show that this Markov chain has bounded jumps. We note that for any $k$ $w_{k+1}=f(f^{-1}(w_k)s)$ for some $s$ in the generating set of $H$, by definition of the simple random walk on $H$. Hence 
%\begin{align*}
% \begin{split}
% d_{G}(w_k,w_{k+1}) &\leq d_G(w_k, f(f^{-1}(w_k)s)) \\
% &\leq R d_H(f^{-1}(w_k),f^{-1}(w_k)s) \\
% &\leq R
% \end{split}
% \end{align*}
%as $f$ is a bijective $R$-quasi-isometry.  

\par\medskip

\eqref{item:non-amen} The analogous statement for the random walks on $H$, which is a well-known consequence of non-amenability \cite{Kesten,Day}, readily implies this condition \ref{item:non-amen} for the Markov chain on $G$.

\par\medskip

\eqref{item:irred} We now verify the third point in the definition of tameness. First note that for any $s\in G$ there exists a finite set $S\subseteq H$ such that for any $g\in G$ we have $f^{-1}(gs)\in f^{-1}(g) S$; again this is a consequence of $f$ being a quasi-isometry.

Fix now $s\in G$. Given $g\in G$ there exists $s'\in S$ such that for all $k$ we have
$$p^{(k)}(g,gs)= p^{(k)}(f^{-1}(g), f^{-1}(g)s')=\mu^{(k)}(s').$$
Since the support of $\mu$ generates $G$ as a semigroup, for any $s'\in S$ there exists $K_{s'}$ such that $\mu^{(K_{s'})}(s')>\epsilon_{s'}$ for some $\epsilon_{s'}>0$. We can then take $\epsilon_s=\min_{s'\in S}(\epsilon_{s'})$ and $K_s=\max_{s'\in S}K_s$.

%We first note that the random walk on $H$ being irreducible implies it is uniformly irreducible for the Cayley graph structure, for some constants $\epsilon_0, U >0$ (See above and \cite{woess_2000}, p.11). 
%We show that the third point of tameness holds for this Markov chain. For $s \in G$ let $\epsilon_s:= \epsilon_0^{R l_G(s)}$ and $K_s:=Rl_G(s)U$. For all $g \in G$ let $y=(f^{-1}(g))^{-1}f^{-1}(gs) \in H $, then by Lemma \ref{epsilonlength} there exist $k \leq U l_H(y)$ such that for all $h \in H$ we have $ p_H^{(k)}(h,hy)\geq \epsilon_0^{l_H(y)}$. Further, we note that $l_H(y) =d_H(f^{-1}(g), f^{-1}(gs)) \leq R l_G(s)$. Therefore, \[
%p^{(k)}_G(g,gs)=p_H(f^{-1}(g), f^{-1}(gs)) \geq \epsilon_0^{l_H(y)} \geq \epsilon_0^{R l_G(s)}=\epsilon_s
%\] 

%and $k \leq U l_H(y) \leq URl_G(s)=K_s$ and the third bullet point is satisfied. 

% (4) We show that $P$ is strongly reversible. We know (Woess \cite{woess_2000}, p.27) that the simple random walk on a space $X$ is strongly reversible if and only if $X$ has bounded geometry. This is the case for a simple random walk on $H$. Let $m_{H}:H \to (0, +\infty)$ be the strongly reversible function from Definition \ref{stronglyreversible}. 
% Define $m_G:G \to (0, +\infty)$ as $m_G(g)=m_H(f^{-1}(g))$. Strong reversibility  for the Markov chain defined on $G$ follows.

Hence, the Markov chain defined on $G$ is tame. 
\end{proof}
Throughout this paper, we will be working with tame Markov chains.
	
\subsection{Logarithmic subpaths }

In this subsection we prove two lemmas regarding what kinds of paths Markov chains could or are likely to follow.

The first lemma roughly speaking gives a lower bound on the probability that the Markov chain ends up at a specified element. 

\begin{lemma}
\label{epsilonlength}
	Let $(w^o_n)_n$ be a tame Markov chain on a finitely generated group $G=\langle S \rangle$. Then there exist constants $\epsilon_0,U>0$ such that the following holds. For all $y \in G$ there exists $k \leq \ell_S(y) U$ such that for all $h \in G$ we have
	$$\mathbb{P}\big[ w_k^h=hy\big] \geq \epsilon_0^{\ell_S(y)}$$ where $\ell_S(y)$ is the word length of the element $y$ with respect to the finite generating set $S$.
\end{lemma}

\begin{proof}
We can take $\epsilon_0=\min \{\epsilon_s : s\in S\}$ and $U=\max \{K_s : s\in S\}$, where $\epsilon_s,K_s$ are as in Definition \ref{defn:tame}-\eqref{item:irred}.

We will work by induction on the length $l=\ell_S(y)$ of the element $y$. For $l=1$, this is clear as we can choose $k =k_s \leq U$ where $y=s\in S$ and $k_s$ is from Definition \ref{defn:tame}-\eqref{item:irred}.

We assume the lemma holds for all elements of length $l$, and we shall show it for elements of length $l+1$. Let $y=s_1\cdots s_ls_{l+1}$. By the inductive hypothesis, there exists some $k \leq l U$ such that for all $h \in G$ we have $\mathbb{P}\big[ w_k^h=hs_1\cdots s_l\big] \geq \epsilon_0^{l}$. \\

We let $k_{s_{l+1}} \leq U$, from the Definition \ref{defn:tame}-\eqref{item:irred}, be such that $ \mathbb{P}\big[ w_{k_{s_{l+1}}}^{hs_1\cdots s_l}=hs_1\cdots s_ls_{l+1}\big] \geq \epsilon_0 $. Let $k'=k+k_{s_{l+1}}$. 

Hence, for all $h \in G$: 
 \begin{align*}
\begin{split}
	\mathbb P \big[ w_{k'}^h=hy\big]& \geq \mathbb P \big[ (w_{k'}^h=hy) \cap (w^h_k=hs_1\cdots s_l)\big] \\
	& = \mathbb P \big[ w_{k'}^h=hy \quad \vert \quad w^h_k=hs_1\cdots s_l\big] \mathbb P \big[w^h_k=hs_1\cdots s_l\big] \\
	& = \mathbb P \big[  w^{hs_1\cdots s_l}_{k_s}=hy\big] \mathbb P \big[w^h_k=hs_1\cdots s_l\big] \\
	& \geq \epsilon_0 \epsilon^l_0
\end{split}
 \end{align*}

where we used Lemma \ref{lem:Markov_property_set} for the ``change of basepoint'' from the second to the third line. Further $k' \leq (l+1)U$ as desired. This proves the lemma.
\end{proof}

The following lemma will be essential when proving Lemma \ref{etalogproj} later, and roughly speaking it says that it is very likely that a Markov path of length $n$ contains ``a copy'' of any given path of length about $\log(n)$ that it can possibly contain.

\begin{lemma} 
\label{markovword}
Let $G$ be a finitely generated group with a fixed generating set $S$. Let $(w^{o}_{n})$ be a tame Markov chain.  Then there exist constants $\eta, U>0$ such that the following holds for all basepoints $o \in G$. For all $n\geq 1$ and all elements $y \in G$ with $\ell_S(y) \leq \eta \log(n)$ we have: $$\mathbb{P}\Big[\exists i,j \leq n \hspace{2mm} : \hspace{2mm} (w^{o}_{i})^{-1}w^{o}_{j}=y\Big] \geq 1-e^{-\sqrt{n}/U}.
$$

\end{lemma}

\begin{proof}

We fix $\epsilon_0, U,k>0$ as in Lemma \ref{epsilonlength} (note that we will increase $U$ in the last step of the proof). For any basepoint $o$ and $i\geq 1$, we set $u^o_i=(w^o_{(i-1)k})^{-1} w^o_{ik}$. Our goal, roughly, is to show that the probability that no $u^o_i$ equals $y$ is small (uniformly over the choice of basepoint $o$).

By induction on $r$, we show:

\begin{claim}
	 For all $r \geq 1$, and $o\in G$ we have
	 $\mathbb{P}\Big[\forall  1\leq i \leq r \hspace{2mm}  : u^o_{i} \neq y\Big]\leq (1-\epsilon_0^l)^{r}.$
\end{claim} 

\begin{proof}[Proof of Claim]
	
For the base case $r=1$, since $w_0^o=o$ (with probability 1) we have $$ \mathbb P \big[u^{o}_1 \neq y\big]=\mathbb P \big[ o^{-1}w_k^{o}\neq y \big] \leq 1-\epsilon_0^{l}$$
by Lemma \ref{epsilonlength}.
% \begin{align*}
% \begin{split}
%  \mathbb P \big[ u_2^{o} \neq y \big] &=\mathbb P \big[ (w_{k}^{o})^{-1}w_{2k}^{o} \neq y \big]\\& = \sum_{g \in G}\mathbb P \big[ (w_{k}^{o})^{-1}w_{2k}^{o} \neq y \vert w^{o}_k=g\big]\mathbb P \big[ w^{o}_k=g\big]   \\
%  &= \sum_{g \in G}\mathbb P \big[ w_{k}^{g} \neq gy \big]\mathbb P \big[ w^{o}_k=g\big] \\
%  &\leq 1-\epsilon_0^l
% \end{split}
% \end{align*}
% where we use the strong Markov property (Lemma \ref{lem:Markov_property_set}) to go from line 2 to 3 and Lemma \ref{epsilonlength} to conclude.

Now, assume that the claim holds for $r$, we will show that it holds for $r+1$.

\begin{align*}
	\begin{split}
		\mathbb{P}\Big[\forall  1\leq i \leq r+1 \hspace{2mm}  : u^o_{i} \neq y\Big] &= \mathbb P \Big[(\forall  1\leq i \leq r \hspace{2mm}  : u^o_{i} \neq y) \cap (u^{o}_{r+1}\neq y)  \Big] \\
		&=\mathbb P \Big[ u_{r+1}^{o} \neq y \Big\vert \quad \forall  1\leq i \leq r\hspace{2mm}  : u^o_{i} \neq y\Big] \mathbb P \Big[\forall  1\leq i \leq r \hspace{2mm}  : u^o_{i} \neq y \Big] \\
		& \leq  \mathbb P \Big[ u_{r+1}^{o} \neq y \Big\vert \quad \forall  1\leq i \leq r\hspace{2mm}  : u^o_{i} \neq y\Big](1-\epsilon_0^l)^{r}
	\end{split}
\end{align*}

by the inductive hypothesis. Now, 

\begin{align*}
	\begin{split}
		\mathbb P \Big[ u_{r+1}^{o} \neq y \Big\vert \quad \forall  1\leq i \leq r\hspace{2mm}  : u^o_{i} \neq y\Big] &= \sum_{g \in G} \mathbb P \Big[ u_{r+1}^{o} \neq y \Big\vert \quad (\forall  1\leq i \leq r\hspace{2mm}  : u^o_{i} \neq y ) \cap (w_{rk}=g)\Big]\mathbb P \Big[w^{o}_{rk}=g \Big\vert \forall  1\leq i \leq r\hspace{2mm}  : u^o_{i} \neq y  \Big] \\
		& = \sum_{g \in G} \mathbb P \Big[ (w_{rk}^{o})^{-1}w_{(r+1)k}^{o}  \neq y \Big\vert w_{rk}=g\Big]\mathbb P \Big[w^{o}_{rk}=g \Big\vert \forall  1\leq i \leq r\hspace{2mm}  : u^o_{i} \neq y  \Big] \\
		&= \sum_{g \in G} \mathbb P \Big[ w_k^{g} \neq gy \Big]\mathbb P \Big[w^{o}_{rk}=g \Big\vert \forall  1\leq i \leq r\hspace{2mm}  : u^o_{i} \neq y  \Big] \\
		& \leq (1-\epsilon_0^l),
	\end{split}
\end{align*}
where the first line is just the law of total probabilities on a conditional probability, to go to the second we use the (usual) Markov property, and to go to the third we use the strong Markov property and then Lemma \ref{epsilonlength} (or the base case of the induction). 

This leads to $$\mathbb{P}\Big[\forall  1\leq i \leq r+1 \hspace{2mm}  : u^o_{i} \neq y\Big] \leq \mathbb P \Big[ u_{r+1}^{o} \neq y \Big\vert \quad \forall  1\leq i \leq r\hspace{2mm}  : u^o_{i} \neq y\Big](1-\epsilon_0^l)^{r-1}\leq (1-\epsilon_0^l)^{r+1}
 $$ and this proves the inductive step, the claim follows. 
\end{proof}

Using the claim we get (for any $y$, and the associated $k$ and $l=\ell_S(y)$)

\begin{equation*}
 \begin{split}
  \mathbb{P}\Big[\forall i,j \leq n \hspace{2mm} \vert \hspace{2mm} (w_{i})^{-1}w_{j}\neq y\Big] & \leq  \mathbb{P}\Big[\forall 1\leq i \leq \lfloor n/k \rfloor \quad u_i \neq y \Big] \\
  &\leq (1-\epsilon_0^l)^{\lfloor n/k \rfloor-1}\\
  &\leq (1-\epsilon_0^l)^{n/(lU)-1}, 
 \end{split}
\end{equation*}
where for the last line recall that $k\leq U l$.

We now assume that $y$ satisfies $l=\ell_S(y)\leq \eta\log(n)$ for $\eta = \frac{-1}{3\log(\epsilon_0)}$ and we show that for all but finitely many $n$ we have $(1-\epsilon_0^l)^{n/U-1}\leq e^{-\sqrt{n}}$. This suffices to prove the lemma for a large enough $U$, and specifically we will now consider $n\geq 1$ large enough that $n/(lU)-1\geq n^{5/6}$. In this setup we have $\epsilon_0^l \geq \epsilon_0^{\eta\log(n)}= n^{-1/3}$ and

$$(1-\epsilon_0^l)^{n/(lU)-1}\leq ((1-n^{-1/3})^{n^{1/3}})^{\sqrt{n}}\leq e^{-\sqrt{n}},$$

where we used that $(1-a^{-1})^a\leq e^{-1}$ for all $a\geq 1$. This concludes the proof.
\end{proof}

\section{WPD and super-divergent elements}
\label{sec:WPD}

\subsection{Super-divergent elements}
We define super-divergent elements, study some  basic properties of these elements, and give some examples.

For a subset $A \subseteq Z$ of a metric space $Z$ and a fixed map $\pi:Z \rightarrow A$ we write $d_{A}(x,y):= d(\pi(x), \pi(y)).$

\begin{definition}[Super-divergence]
Let $A$ be a subset of a metric space $Z$ and let $\pi:Z \rightarrow A$ be a projection, that is, a map such that $\pi|_{A}$ is the identity. We say that the projection $\pi$ is \textit{super-divergent} if there exists a constant $\theta > 0$ and a super-linear function $f:\mathbb{R}^{+} \to \mathbb{R}^{+}$ such that the following holds. For all $d >0$ and paths $p$  remaining outside of the $d$-neighbourhood of $A$, if $d_{A}(\pi(p_{-}),\pi(p_{+})) > \theta$ then $\ell(p) > f(d)$, where $p_{-}$ and $p_{+}$ are the endpoints of $p$ and $\ell(p)$ is its length. 
\end{definition}

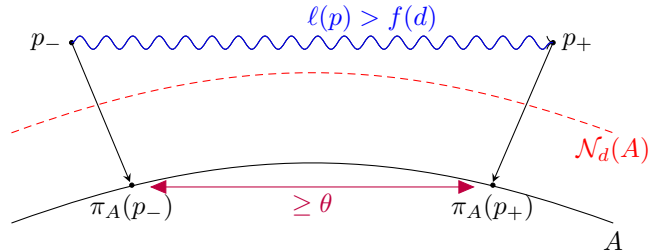
\begin{figure*}[h]
	\centering
	\begin{tikzpicture}[scale=0.8]
\draw[black] (0,0) to[out=20,in=160] node [pos=0.75] (mid1) {} 
        node [pos=0.55] (mid2) {} 
         node [pos=0.85] (mid3) {} (10,0);
\filldraw[black] (2,0.63) circle (1pt) node[anchor=north]{$\pi_A(p_{-})$};    
\filldraw[black] (8,0.63) circle (1pt) node[anchor=north]{$\pi_A(p_{+})$}; 

\draw[densely dashed][red] (0,1.5) to[out=20,in=160] node [pos=0.75] (mid1) {} 
        node [pos=0.55] (mid2) {} 
         node [pos=0.85] (mid3) {} (10,1.5);     

\filldraw[black] (1,3) circle (1pt) node[anchor=east]{$p_{-}$};
\filldraw[black] (9,3) circle (1pt) node[anchor=west]{$p_{+}$};

\draw[-stealth] (1,3) --(1.97,0.68);
\draw[-stealth] (9,3) --(8,0.68);

\draw [->,decorate,decoration=snake] (1,3) -- (9,3);
  \path [draw=blue,snake it] (1,3) -- (9,3);

\draw[blue] (6,3.4) node{$\ell(p) >f(d)$};

\draw[>=triangle 45, <->,purple] (2.3,0.6)-- (7.7,0.6) ;

\draw[purple] (5,0.3) node{$\geq \theta$};

\draw[red] (10,1.2) node{$\mathcal N_{d}(A)$};
\draw[black] (10,-0.3) node{$A$};

\end{tikzpicture}
\caption{The projection $\pi_A$ is super-divergent.}
\end{figure*}

\begin{definition}
\label{defn:X_proj}
Let $G$ act on a hyperbolic space $X$. If $A\subseteq G$, an $X$-\textit{projection} $\pi: G \rightarrow A$ is a map such that for all $h\in G$ we have that $\pi(h)x_0$ is a closest point in $Ax_0$ to $hx_0$, for a fixed basepoint $x_{0} \in X$.
\end{definition}

Recall that an element $g \in G$  \textit{acts loxodromically}  on $X$ if the map $\mathbb Z \rightarrow X$ given by $n \rightarrow g^nx_0$ is a quasi-isometric embedding for some (equivalently, any) $x_0 \in X$. We note that 
 if $g$ is loxodromic, then $\langle g \rangle x_0$ is quasi-convex in $X$. Hence we can apply Lemma \ref{lem:exo_hyperbolicspace} to $\langle g \rangle$, which we will do in Lemma \ref{lem:examples_of_superdivgelements} below.
\begin{definition}
Let $G$ act on a hyperbolic space $X$ and let $g \in G$. We say that $g$ is a \textit{super-divergent element}
if \begin{itemize}
\setlength\itemsep{1em}
    \item $g$ acts loxodromically on $X$, and
 
    \item an $X$-projection  $\pi_{g}=\pi : G \rightarrow \langle g \rangle$ is super divergent, where $G$ is endowed with a fixed word metric $d_G$. We shall call $\theta$ the \textit{super-divergent constant} for $g$.
\end{itemize}
\end{definition}

In the following lemmas we give examples of groups that have a super-divergent elements.

\begin{lemma}
\label{lem:examples_of_superdivgelements}
Infinite hyperbolic groups have a super-divergent element where the hyperbolic space acted on can be taken to be either 

\begin{enumerate}
    \item the Cayley graph with respect to a finite generating set, or (more generally)
    \item the coned-off Cayley graph with respect to a family of uniformly  quasi-convex infinite index subgroups.
\end{enumerate}
\end{lemma}

\begin{proof}

While technically the second case is more general, we prove the first one separately for clarity and since we will reduce the second case to the first one.

(1) Here, the $\delta$-hyperbolic space $X$ is the Cayley graph of $G$ with respect to a finite generating set $S$, for some $\delta \geq 0$. We know that as $G$ is hyperbolic, every infinite order element is loxodromic for the action of $G$ on its Cayley graph $X$. Let $g$ be such a loxodromic element. 
Let $R>0$ be the constant from Lemma \ref{lem:exo_hyperbolicspace} for the quasi-convex subgroup $\langle g\rangle$.  

Let  $d>R$ and $\alpha$ be a path staying outside of the $d$-neighbourhood of the axis $\langle g \rangle $ and such that the $X$-projection of $\alpha $ satisfies $d(\pi_{g}(\alpha_{-}), \pi_{g}(\alpha_{+})) \geq R $ where $\alpha_{-}$ and $\alpha_{+}$ are the endpoints of this path. Then by Lemma \ref{lem:exo_hyperbolicspace}, there is a point $m$ on the geodesic from $\alpha_{-}$ to $\alpha_{+}$ such that $d(m, \langle g \rangle) \leq R$. Hence $d(m, \alpha) \geq d-R$ and by \cite[Proposition III.H.1.6]{bridsonhaefliger} (taking exponentials on both sides), this leads to $\ell(\alpha) \geq 2^{(d-{R}-1)/ \delta}$ which is super linear in $d$ and hence $g$ is a super-divergent element. \\

(2) We will show that $G$ contains a super divergent element for the action on the coned-off Cayley graph with respect to one quasi-convex infinite-index subgroup, the general case follows from the fact that a family of uniformly quasi-convex subgroups will be finite.

 Let $X$ be the Cayley graph of $G$ with respect to a generating set $S$ and $\hat{X}$ be the coned-off Cayley graph over the quasi-convex subgroup $H$, which is hyperbolic by \cite[Proposition 5.6]{kapovichrafi}.  Consider a loxodromic element $g$ for the action on $\hat{X}$; such an element exists in view of \cite[Theorem 2.4]{mahermasaischleimer}, applied to the element given by \cite[Theorem 1]{minasyan}. We note that for a point $x \in G$ we will interchangeably say that $x \in X$ or $x \in \hat{X}$. When considering geodesics in either of these spaces we will be more precise as to where they are. We note that by definition of the coned-off Cayley graph we have $d_{\hat{X}} \leq d_X=d_G$.\\

\textbf{Claim}:  There exists a constant $M>0$ such that the following holds. For $x \in G $, let $\hat{\pi}_g(x)$ be coned-off projection ($\hat{X}$-projection) of $x$ and let $\pi_g(x)$ be the closest-point projection of $x$ on $\langle g \rangle$. Then for all $x\in G$, we have $d_X(\hat{\pi}_g(x),\pi_g(x)) \leq M$. 

\begin{proof}[Proof of claim]

We will show that there is a constant $M'>0$ such that $d_{\hat{X}}(\hat{\pi}_g(x),\pi_g(x)) \leq M'$. By the fact that $\langle g \rangle \xhookrightarrow{q.i.} \hat{X}$ is a quasi-isometric embedding, this is enough to prove the claim.

We first note that by \cite[Proposition 5.6]{kapovichrafi}, there exists a constant $N>0$ such that the following holds. Any geodesic $ [p_1,p_2]_{X}$ in $X$ between points $p_1$ and $p_2$ gets mapped $N$-Hausdorff close to a geodesic $ [p_1, p_2]_{\hat{X}}$ in $\hat{X}$ between $p_1$ and $p_2$. \\
Let $M'=2(N+R)$ where $R$ is from Lemma \ref{lem:exo_hyperbolicspace}, for the hyperbolic space $\hat{X}$ and quasi-convex subgroup $\langle g \rangle$. For a contradiction, assume that $d_{\hat{X}}(\hat{\pi}_g(x),\pi_g(x))> M'$. Hence $d_X(\hat{\pi}_g(x),\pi_g(x))> M'$ and by Lemma \ref{lem:exo_hyperbolicspace} there is a point $m \in [x, \hat{\pi}_g(x)]_X$ (a geodesic in $X$) such that $d_X(m,\pi_g(x)) \leq R$. Let $q \in [x, \hat{\pi}_g(x)]_{\hat{X}}$ (this is a geodesic in $\hat{X}$) be such that $d_{\hat{X}}(q,m) \leq N $, which exists by the fact that $[x, \hat{\pi}_g(x)]_X$ gets mapped $N$-Hausdorff close to $[x, \hat{\pi}_g(x)]_{\hat{X}}$, with respect to $d_{\hat{X}}$. \\

Hence $d_{\hat{X}}(q,\pi_g(x)) < R+N$ and so $d_{\hat{X}}(q,\hat{\pi}_g(x)) < R+N$ (otherwise $\hat{\pi}_g(x)$ would not be the $\hat{X}$-projection of $x$). This leads to $$
2(N+R)= M'< d_{\hat{X}}(\hat{\pi}_g(x),\pi_g(x)) \leq d_{\hat{X}}(q,\pi_g(x))+d_{\hat{X}}(q,\hat{\pi}_g(x)) \leq 2(N+R)
$$ a contradiction, proving the claim.
\end{proof}

Let $\alpha $ be a path in $G$ staying outside of the $d$-neighbourhood  of $\langle g \rangle$ and such that $d_{X}(\hat{\pi}_{g}(\alpha_{-}),\hat{\pi}_g( \alpha_{+})) > \theta_1 +2M$ where $\theta_1$ is the super-divergent constant from the bullet point (1) above and $M$ is from the claim. The claim then leads to \begin{align*}
  \begin{split}
d_X(\pi_g(\alpha_{-}),\pi_g(\alpha_{+}))) &\geq d_{X}(\hat{\pi}_{g}(\alpha_{-}),\hat{\pi}_g( \alpha_{+}))-d_X(\hat{\pi}_g(\alpha_{-}),\pi_g(\alpha_{-}))-d_X(\hat{\pi}_g(\alpha_{+}),\pi_g(\alpha_{+})) \\
&\geq \theta_1+2M-2M = \theta_1
  \end{split}  
\end{align*}
We can therefore apply bullet point (1) above to conclude. 
\end{proof}

\begin{lemma}
\label{lem:rh_superdiv}
Let $(G,\mathcal P)$ be a relatively hyperbolic group, where all $P\in\mathcal P$ have infinite index in $G$. Then $G$ contains a super-divergent element for the action on the coned-off Cayley graph $\Cay(G,S\cup \mathcal P)$, where $S$ is a finite generating set for $G$.
\end{lemma}

\begin{proof}
We denote by $\hat{G}$ the coned-off Cayley graph of $G$, with respect to $\mathcal P$. Further, denote by $d_G$ the word metric on $G$ with respect to the generating set $S$ and by $d_{\hat{G}}$ the metric in $\Cay(G,S\cup \mathcal P)$ with respect to $S \cup \mathcal P$. Clearly $d_{\hat{G}} \leq d_G$. Similarly to what we did above,  for a $x \in G$, we will interchangeably say that $x \in G$ or $x \in \hat{G}$. \\

By \cite[Theorem 3.26]{osinrelativelyhyperbolic}, there exists a constant $\nu>0$ such that the following holds. Let $\Delta = \xi_1\xi_2\xi_3$ be a geodesic triangle whose sides $\xi_i$ are geodesics in $\hat{G}$. Then for any vertex $a \in \xi_1$ there is a vertex $b \in \xi_2\cup \xi_3$ such that $d_G(a,b) \leq \nu$.\\

Let $g$ be an element acting loxodromically on $\hat{G}$ with axis $\langle g  \rangle$, which exists by \cite[Corollary 4.5]{osinelementary}. As $\langle g \rangle$ is quasi-isometrically embedded in $\hat{G}$ (say with quasi-isometric embedding constants $(a,b)$), the subgroup  $\langle g \rangle$ is $\iota$-quasi-convex in $\hat{G}$, for some $\iota>0$. Let $\alpha $ be a path in $G$, between points $\alpha_{-},\alpha_{+} \in G$,  staying outside of the $d$-neighbourhood (with respect to $d_{G}$) of $\langle g \rangle$ and such that $d_{G}(\pi_g(\alpha_{-}),\pi_g(\alpha_{+}) )\geq a(10 \nu +  10\iota +b)=: \theta$.
%, where $(a,b)$ are the quasi-isometric embedding constants for $\langle g \rangle$ in $\hat{G}$. \\

Let $\hat{\beta_1}$ be a geodesic, in $\hat{G}$, between $\pi_g(\alpha_{-})$ and $\pi_g(\alpha_{+})$. By the choice of $\theta$, we can find a point $m \in \hat{\beta_1}$ at distance at least $4\nu +3 \iota $ from both $\pi_g(\alpha_{-})$ and $\pi_g(\alpha_{+})$. Let $\hat{\beta_2}=[\alpha_{-},\alpha_{+}]_{\hat{G}}$ be a geodesic in $\hat{G}$, similarly $\hat{\beta_3}=[\alpha_{-},\pi_g(\alpha_{-})]_{\hat{G}}$ and $\hat{\beta_4}=[\alpha_{+},\pi_g(\alpha_{+})]_{\hat{G}}$. Then, by \cite[Theorem 3.26]{osinrelativelyhyperbolic} mentioned above and by a usual quadrangle argument, this leads to $d_{G}(m,\hat{\beta_2} \cup \hat{\beta_3} \cup \hat{\beta_4}) \leq 2\nu$. 

Now, by our choice of $m$ and the fact that $\pi_{g}(\alpha_{-}), \pi_g(\alpha_{+}) $ are closest point projections, we have $d_G(m, \hat{\beta_3} \cup \hat{\beta_4}) > 2\nu$. Hence there is a point $y \in \hat{\beta_2}$ such that $d_G(m,y) \leq 2 \nu $.\\

By  \cite[Lemma 8.8]{hruskarelativelyhyperbolic} and by the fact that $\langle g \rangle$ is also a quasi-geodesic in $G$, there is a constant $L_1>0$ (only depending on the quasi-geodesic constants of $\langle g \rangle$) and a point $m'\in \langle g \rangle$ such that $d_G(m,m') \leq L_1$. \\

Let $\beta_2$ be a geodesic, in $G$, from $\alpha_{-}$ to $\alpha_{+}$. Then by  \cite[Proposition 8.13]{hruskarelativelyhyperbolic}, there exists a constant $L_2$ such that for some transient point $z \in \beta_2$ we have that $d_G(y,z) \leq L_2$. We do not give the definition of a transient point and refer to \cite{hruskarelativelyhyperbolic} and \cite{sisto2013tracking} for definitions and properties. \\

Combining all of the above yields, $$
d_G(\langle g \rangle, z) \leq d_G(m',m) + d_G(m,y)+d_G(y,z) \leq L_1 + 2\nu +L_2 
$$

Hence $d_G(z,\alpha) \geq d-(L_1 + 2\nu +L_2 )$.

Using \cite[Lemma 4.3]{sisto2013tracking} and the fact that $z$ is a transient point,  we get that there exists a constant $R$ such that $$R\log_2(\ell(\alpha)+1) +R \geq d-(L_1 + 2\nu +L_2 )$$

This leads to \[
\ell(\alpha) \geq 2^{(d-(L_1 + 2\nu +L_2 )-R)/R}-1
\]
which is super-linear in $d$, giving the required function $f$.
\end{proof}

\subsection{A digression on divergence}
We now note a result, relating the existence of super-divergent elements and the  \textit{divergence} of the ambient group; this will not be needed in the rest of the paper.

We invite the reader to read \cite{DrutuMozesSapir}  for equivalent definitions of divergence in a group. 

We will only use the observation from \cite[Section 3]{divergencedrutu} that, given a bi-infinite quasi-geodesic $q$ in a metric space $X$, seen as a function $q : \mathbb{R} \to X$,  the \textit{divergence function of the quasi-geodesic} $Div^{q}_{\eta}$ is a lower bound for the divergence of $X$. For fixed constants $0 <\delta <1$ and $\eta \geq 0$, $Div^{q}_{\eta}(m)$ is defined as the infimum of the length  of the paths from $q(m)$ to $q(-m)$ avoiding the ball $B(q(0), \delta m-\eta)$. (We note that the choice of constants is irrelevant provided that they are in the range given in [Corollary 3.12, \cite{DrutuMozesSapir}]).

\begin{lemma}
\label{divergence}
If $G$ has a super divergent element $g$ for an action on a hyperbolic space $X$, then $G$ has at least super-quadratic divergence.
\end{lemma}

\begin{proof}
We fix a word metric $d_G$ on $G$.

Consider a super-divergent element $g$ and the $(a,b)$-quasi-geodesic $\langle g \rangle x_0 $. Fix constants $0 <\delta <1$ and $\eta \geq 0, \eta \neq b$. We will show that the function $Div^{q}_{\eta}(m)$ is super-quadratic in $m$. It is enough to consider this function for $m \geq m_0$ where $m_0 > \frac{2a^2(\eta-b)}{a\delta}=:M'$.  \\

\begin{claim} There is a constant $D>0$ such that the following holds. If $h\not\in B^G(1,\delta m-\eta)$ and $d_G(1,\pi_g(h)) \leq \delta m/2a -\eta $ then $h \not\in \mathcal{N}_{\delta m/D -\eta}(\langle g \rangle)$, all considered with respect to the metric $d_G$. 
\end{claim}
 
\begin{proof}[ Proof of claim]
Let $D > \frac{2a^2\delta m_0}{M' }$ where $M'$ is defined above. For a contradiction, assume that there is a point $h' \in \langle g \rangle$ such that $d_G(h,h')  \leq \delta m/D-\eta$. Hence $d_X(hx_0,h'x_0) \leq \delta m/D-\eta$. This leads to $d_G(1,h') \geq d_{G}(1,h)-d_G(h,h') \geq \delta m(D -1) /D $. \\

Hence $d_X(x_0, h'x_0) \geq \delta m(D -1) /aD-b$ as $\langle gx_0 \rangle$ is a $(a,b)$-quasi-geodesic in $X$. Therefore:

$$
d_X(hx_0, h'x_0)  \geq d_X(\pi_g(h)x_0, hx_0) \geq d_X(x_0, h'x_0) -d_X(hx_0, h'x_0)-d_X(x_0,\pi_g(h)x_0) 
$$

and so

\begin{align*}
 	\begin{split}
 2(\delta m/D-\eta) \geq 2d_X(hx_0, h'x_0) &\geq d_X(x_0, h'x_0) -d_X(x_0,\pi_g(h)x_0) \\ &\geq \delta m(D -1) /aD-b-(\delta m/2a -\eta )  
  	\end{split}
  	\end{align*}
 
which by the choice of $D$ is a contradiction. Hence there is no such $h' \in \langle g \rangle$ and $h \not\in \mathcal{N}_{\delta m/D -\eta}(\langle g \rangle)$ as claimed.
\end{proof}

% \begin{proof}[ Second Proof of claim]
% Let $D > \frac{3a+1}{1-a(b-4\eta/\delta m) }$ where $M'$ is defined above. For a contradiction, assume that there is a point $h' \in \langle g \rangle$ such that $d_G(h,h')  \leq \delta m/D-\eta$. Hence $d_X(hx_0,h'x_0) \leq \delta m/D-\eta$. This leads to $d_G(1,h') \geq d_{G}(1,h)-d_G(h,h') \geq \delta m(D -1) /D $. \\

% Hence $d_X(x_0, h'x_0) \geq \delta m(D -1) /aD-b$ as $\langle gx_0 \rangle$ is a $(a,b)$-quasi-geodesic in $X$. Therefore: \begin{align*}
%  	\begin{split}
%  2(\delta m/D-\eta) \geq 2d_X(hx_0, h'x_0)  &\geq d_X(\pi_g(h)x_0, hx_0) \geq d_X(x_0, h'x_0) -d_X(hx_0, h'x_0)-d_X(x_0,\pi_g(h)x_0) \\ &\geq \delta m(D -1) /aD-b-(\delta m/D-\eta)-(\delta m/2a -\eta )  \\
% \end{split}
% \end{align*}
 
% which by the choice of $D$ is a contradiction. Hence there is no such $h' \in \langle g \rangle$ and $h \not\in \mathcal{N}_{\delta m/D -\eta}(\langle g \rangle)$ as claimed.
% \end{proof}

 Let $p$ be a path in (the Cayley graph of) $G$ from $g^{-m}$ to $g^{m}$  staying outside of the ball $B(1,\delta m-\eta)$; for convenience we will actually consider a discrete path (that is, a sequence of vertices with consecutive ones connected by an edge). We will think of the path $p$ as being 'oriented' from $g^{-m}$ to $g^{m}$, this allows us to talk about 'first points' and 'last points'. \\

 Let $y_1$ be the first point on $p$ such that $d_G(1, \pi_g(y_1)) \leq \delta m/2a -\eta$. We then define the point $y_2$ to be the first point after $y_1$ such that $d_{\langle g \rangle}(y_1,y_2) \geq \theta$ where $\theta$ is the super-divergent constant for $g$. By minimality of $y_2$, and the fact that closest-point projections to quasi-convex sets are coarsely Lipschitz (Lemma \ref{lem:exo_hyperbolicspace}) this also means that $d_{\langle g \rangle}(y_1,y_2) \leq \theta +L$, for some constant $L$.\\
We recursively define the point $y_i$ as the first point, after $y_{i-1}$, such that $d_{\langle g \rangle}(y_{i-1},y_i) \geq \theta$. We again note that $d_{\langle g \rangle}(y_{i-1}, y_i) \leq \theta + L$. We stop this process as soon as $y_{s+1}$ projects on $\langle g \rangle$ outside of $B(1,\delta m/2a-\eta)$.\\ 
 
Say that there are $s$ of these points $(y_i)_{i=1}^s$. By the fact that $\langle g \rangle$ is a $(a,b)$-quasi geodesic, this leads to $s(\theta +L) > \frac{1}{a}(\delta m/2a-\eta) -b-2L$. By the claim above, each subpath $p_i$ of $p$ from $y_i$ to $y_{i+1}$  is outside of the $(\delta/D m-\eta)$-neighbourhood of $\langle g  \rangle$.\\

By super-divergence of $g$, we get that $$ \ell(p) > \sum_{i=1}^{s} \ell(p_i) \geq s f(\delta m/D -\eta) > \frac{1}{\theta+L}(\delta m/ 2a^2 -\eta /a -b-2L)f(\delta m/D -\eta) $$ which is super-quadratic in $m$ as $f$ is super linear in $m$. Hence $Div^{q}_{\eta}(m)$ is super-quadratic in $m$. Therefore the divergence of $G$ is at least super-quadratic.

\end{proof}

We note that since mapping class groups have quadratic divergence \cite{behrstock,Duchin_2009} (with finitely many exceptions), they cannot contain super-divergent elements.

% \begin{corollary} 

% Let $S$ be a surface of genus $g$ with $p$ punctures such that $3g+p >4$. Let $X$ be the curve complex $\mathcal{C}(S)$ of the surface. Then the mapping class group $Mod(S)$ does not have a super-divergent element for its natural action on the curve complex. 

% \end{corollary}

\subsection{WPD elements}
\label{subsec:WPD}

We recall the definition of a WPD element from \cite{bestvinafujiwara}.

\begin{definition}[WPD element]
Let $G$ be a group acting on a hyperbolic space $X$ and $g$ an element of $G$. We say that $g$ satisfies the \textit{weak proper discontinuity condition} (or that $g$ is a $WPD$ element) if for all $\kappa >0$ and $x_0 \in X$ there exists a $N \in \mathbb{N}$ such that \[
\# \{ h \in G \vert \quad d_X(x_0, hx_0)< \kappa, \quad d_X(g^{N}x_0, hg^{N}x_0) < \kappa \} < \infty
\]
\end{definition}

We fix a group $G$ acting on a hyperbolic space $X$ throughout this subsection.

We now show that super-divergent elements are WPD, and we list some known consequences of being WPD.
\begin{lemma}
\label{lem:super_div_WPD}
If $g \in G$ is a super-divergent element for the action on a hyperbolic space $X$ then it is a WPD element.
\end{lemma}

\begin{proof}
As $g$ acts loxodromically on $X$, the axis of $g$ is a $(a, b)$-quasi geodesic both in $X$ and in $G$, where at this point we fixed a word metric $d_G$ on $G$.

For a contradiction, assume that $g$ is not a WPD element. Hence there exists $\kappa>0$ such that for all $m \in \mathbb{N}$ there are infinitely many elements $h \in G$ satisfying:
\begin{equation}
\label{not_wpd_defn}
d_X(x_{0},hx_{0}) <\kappa \hspace{2mm} \text{and} \hspace{2mm} \hspace{1mm}d_X(g^{m}x_{0},hg^{m}x_{0}) < \kappa 
\end{equation}

We let $m$ be such that $m > a(\theta+b+2\kappa)$ where $\theta$ is the super-divergent constant for $g$. Let $d$ be such that $f(d) >a m +b$ where $f$ is the super-linear function from the definition of $g$ being super-divergent.

As we have infinitely many $h$ satisfying (\ref{not_wpd_defn}), we can find $h$ such that $d_G(1,h) > 10d+2(am+b)$. By definition, in $X$ we have that $\pi_g(h)x_0$ and $\pi_g(hg^m)x_0$ are the closest points on $\langle g  x_0 \rangle$ to $h$ and $hg^m$ respectively. Hence $d_X(hg^mx_0, \pi_g(hg^m)x_0) \leq d_X(hg^mx_0, g^mx_0) < \kappa$ and similarly $d_X(hx_0, \pi_g(h)x_0) < \kappa$. Therefore

\[
d_X\left(\pi_{g}(h)x_0,\pi_{g}(hg^m)x_0\right) > d_X(hx_0, hg^{m}x_0) -2\kappa >m/a-b-2\kappa > \theta
\]  by the choice of $m$. Hence $d_G(\pi_{g}(h),\pi_{g}(hg^m)) > \theta$. Further, by the choice of $h$, the geodesic from $h$ to $hg^{m}$ stays outside of the $d$-neighbourhood of $\langle g \rangle $. By super divergence of $g$, this leads to  \[am+b \geq d_G(h,hg^{m}) > f(d) > am+b\] a contradiction. Hence $g$ is a WPD element. 
\end{proof}

Recall that any loxodromic WPD element $g$ is contained in a unique maximal elementary subgroup of $G$, denoted $E(g)$ and called the \textit{elementary closure} of $g$, see  \cite[Theorem 1.4]{Osin_acylindrical}.

We will refer to the following lemma as the strong Behrstock inequality, and we will use it very often. The lemma follows from \cite[Theorem 4.1]{bromberg2017acylindrical}.

\begin{lemma}
\label{lem:Behrstock}
Let $g$ be a loxodromic WPD element. Then, for $\gamma=E(g)$, there is a $g$-equivariant map $\pi_{\gamma}:G\to \mathcal P(\gamma)$, where $\mathcal P(\gamma)$ is the set of all subsets of $\gamma$, and a constant $B$ with the following property. For all $x\in G$ and $h\gamma \neq h'\gamma$ we have
\[ d_X(\pi_{h\gamma}(x),\pi_{h\gamma}(h'\gamma)) >B \implies \pi_{h'\gamma}(x)=\pi_{h'\gamma}(h\gamma),
\]
where $\pi_{k\gamma}(z)=k\pi_{\gamma}(k^{-1}z)$. Moreover, for all $x\in G$ the Hausdorff distance between $\pi_\gamma(x)$ and an $X$-projection of $x$ to $\langle g\rangle$ is bounded by $B$.
\end{lemma}

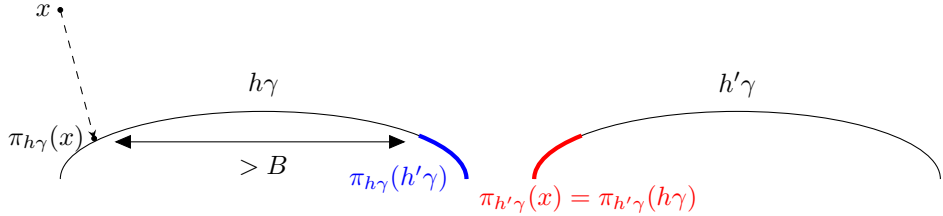
\begin{figure}[h]
\centering
\begin{tikzpicture}[scale=0.9]

\draw (0,0) arc(0:180:3cm and 1cm);

\draw (7,0) arc(0:180:3cm and 1cm);

\draw[ultra thick, blue] (0,0) arc(0:40:3cm and 1cm);
\draw[blue] (-1,0) node{$\pi_{h\gamma}(h'\gamma)$};

\draw[ultra thick, red] (1,0) arc(180:140:3cm and 1cm);

\draw[black] (-3,1.4) node{$h\gamma$};

\filldraw[black] (-6,2.5) circle  (1pt) node[anchor=east]{$x$};

\draw[-stealth, dashed] (-6,2.5) --(-5.5,0.6);

\filldraw[black] (-5.5,0.6) circle  (1pt) node[anchor=east]{$\pi_{h\gamma}(x)$}; 

\draw[red] (1.8,-0.3) node{$\pi_{h'\gamma}(x)=\pi_{h'\gamma}(h\gamma)$};

\draw[black] (4,1.4) node{$h'\gamma$};

\draw[>=triangle 45, <->,black] (-5.2,0.55)-- (-0.9,0.55);

\draw[black] (-3,0.2) node{$>B$};

\end{tikzpicture}
\caption{The strong Behrstock inequality.}
\end{figure}

\begin{notation}
\label{not:d_hgamma}
For a loxodromic WPD element $g$, consider $\gamma=E(g)$ and the maps $\pi_{h\gamma}$ as in Lemma \ref{lem:Behrstock}. For $h,x,y\in G$ we denote
$$d_{h\gamma}(x,y)=\diam (\pi_{h\gamma}(x)\cup\pi_{h\gamma}(y)).$$
\end{notation}

% We now know that the strong Behrstock inequality holds for $\gamma=E(g)$.

% We note that by Lemma \ref{Behrstock}, $\gamma=E(g)$ is quasi-convex and so the lemma holds for $\gamma$.

As in the statement of Lemma \ref{lem:Behrstock}, we will often denote $\gamma=E(g)$ when a loxodromic WPD element $g$ has been fixed.

We will often look at the set of cosets where two given elements have far away projections, as captured by the following definition.

 \begin{definition}
 \label{defn:H_T}
 Given a loxodromic WPD element $g$, for $x,y\in G$ and $T\geq 0$ we define the set \[ \mathcal{H}_T(x,y) :=\{ h\gamma \hspace{1mm} : \hspace{1mm} d_{h\gamma}(x,y) \geq T\}
\]
%In the case where $g$ is a super-divergent element, we let $\gamma=E(g)$ and
% For the purposes of this paper we can take
% \[T=100B+4\theta+100B_1+100R+100Q\]
% where $B$ is from the Behrstock inequality, $\theta$ is the super-divergent constant for $g$ and $B_1, R, Q$ are from Lemma \ref{Behrstock}; when the subscript $\mathcal{H}_T(o,w)$ is not specified we are considering this value of $T$. \\

  Throughout, we will consider the following (``distance formula") sum, for $x,y,z_1,z_2 \in G$, once a loxodromic WPD element $g$ has been fixed: 
 
 \[ \sum_{\mathcal{H}_T(x,y)}[z_1,z_2] := \sum_{\substack{h\gamma \in \mathcal{H}_T(x,y) \\ \pi_{h\gamma}(z_1)\neq \pi_{h\gamma}(z_2)}} d_{h\gamma}(z_1,z_2).
 \] 
 \end{definition}
 
 Figure \ref{fig:order} illustrates a set $\mathcal H_T(o,p)$, and Figure \ref{fig:dist_sum} illustrates a sum $\sum_{\mathcal{H}_T(x,y)}[z_1,z_2]$.

 \begin{remark}
 \label{rem:triangle_inequality}
 (Triangle inequality) Note that for all $x,y,z,o,p\in G$  we have
 $$\sum_{\mathcal{H}_T(o,p)} [x,z] \leq \sum_{\mathcal{H}_T(o,p)} [x,y]  + \sum_{\mathcal{H}_T(o,p)} [y,z],$$
 since projection distances satisfy the triangle inequality as well.
 \end{remark}

Let $o,p\in G$. In view of the fact that the projections on the $h\gamma$ satisfy the projection axioms of \cite{bestvina2014constructing}, by \cite[Theorem 3.3 (G)]{bestvina2014constructing}, we have a linear order on $\mathcal{H}(o,p)$:

\begin{lemma}(Consequence of \cite[Theorem 3.3 (G)]{bestvina2014constructing})
\label{linearorder}
Fix a loxodromic WPD element $g$. For any sufficiently large $T$, the set $\mathcal{H}_T(o,p) \cup \{o,p\}$ is totally ordered with least element $o$ and greatest element $p$. The order is given by $h\gamma \prec h'\gamma$ if one, and hence all, of the following equivalent conditions hold for $B$ as in Lemma \ref{lem:Behrstock}:\\

\begin{itemize}
    \item $d_{h\gamma}(o,h'\gamma) > B$.
    \item $\pi_{h'\gamma}(o)= \pi_{h'\gamma}(h\gamma)$.
    \item $d_{h'\gamma}(p,h\gamma) > B$.
    \item $\pi_{h\gamma}(p)= \pi_{h\gamma}(h'\gamma)$.
\end{itemize}
\end{lemma}

\begin{figure}[h]
\begin{tikzpicture}[scale=0.9]
\draw (0,0) arc(0:180:1cm and 0.5cm);
\draw (3,0) arc(0:180:1cm and 0.5cm);
\draw (6,0) arc(0:180:1cm and 0.5cm);
\draw (9,0) arc(0:180:1cm and 0.5cm);
\draw (12,0) arc(0:180:1cm and 0.5cm);

\draw[black] (0.5,0) node {$\cdots$};
\draw[black] (9.5,0) node {$\cdots$};

\draw[ultra thick, red] (0,0) arc(0:40:1cm and 0.5cm);
\draw[ultra thick, red] (3,0) arc(0:40:1cm and 0.5cm);
\draw[ultra thick, red] (6,0) arc(0:40:1cm and 0.5cm);
\draw[ultra thick, red] (9,0) arc(0:40:1cm and 0.5cm);
\draw[ultra thick, red] (12,0) arc(0:40:1cm and 0.5cm);

\draw[ultra thick, blue] (-2,0) arc(180:140:1cm and 0.5cm);
\draw[ultra thick, blue] (1,0) arc(180:140:1cm and 0.5cm);
\draw[ultra thick, blue] (4,0) arc(180:140:1cm and 0.5cm);
\draw[ultra thick, blue] (7,0) arc(180:140:1cm and 0.5cm);
\draw[ultra thick, blue] (10,0) arc(180:140:1cm and 0.5cm);

\draw[black] (5,-0.5) node {$h\gamma$};

\draw[blue] (4,-0.2) node {\tiny$\pi_{h\gamma}(h'\gamma)=\pi_{h\gamma}(o)$};

\filldraw[black] (-2,2) circle (1pt) node[anchor=east]{$o$}; 
\filldraw[black] (12,2) circle (1pt) node[anchor=east]{$p$}; 
\draw[dashed] (-2,1.9) --(-2,0);
\draw[dashed] (12,2) --(12,0);

\draw[red] (6,0.8) node {\tiny $\pi_{h\gamma}(h''\gamma)=\pi_{h\gamma}(p)$};

\draw [decorate,decoration={brace,amplitude=5pt,mirror,raise=4ex}]
  (-2,0) -- (3,0) node[midway,yshift=-3em]{$h'\gamma  \prec h\gamma$};

\draw [decorate,decoration={brace,amplitude=5pt,mirror,raise=4ex}]
  (7,0) -- (12,0) node[midway,yshift=-3em]{$h\gamma  \prec h''\gamma$};

\end{tikzpicture}

\caption{Linear order on $\mathcal H_T(o,p)$: all the cosets on the left of $h\gamma$ have the same projection on $h\gamma$, and similarly for all the cosets on the right of $h\gamma$.}\label{fig:order}
\end{figure}
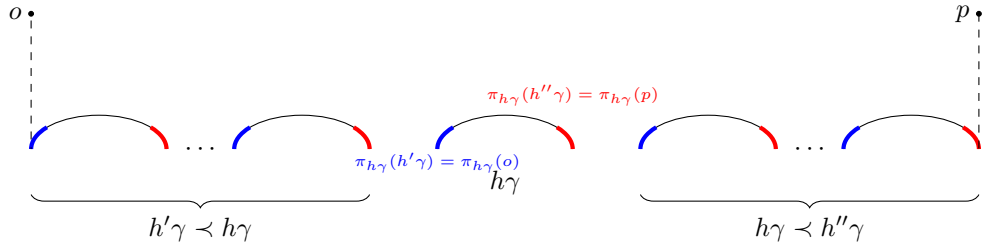

\begin{proof}
We refer to \cite[Theorem 3.3(G)]{bestvina2014constructing} for the linear order; the equivalence of the various conditions follows from the strong Behrstock inequality (Lemma \ref{lem:Behrstock}). 
\end{proof}

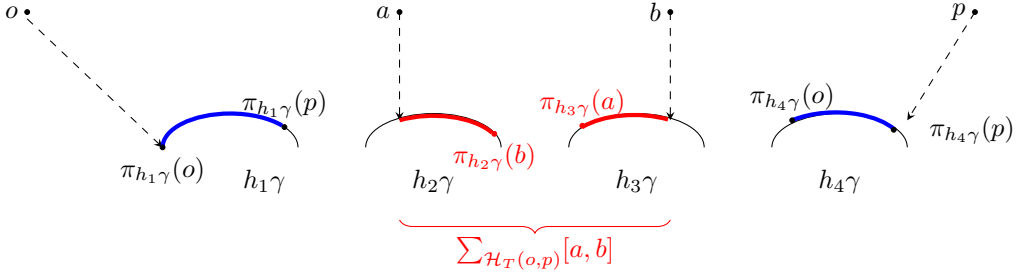
\begin{figure}[h]
\begin{tikzpicture}[scale=0.9]

\draw (2,0) arc(0:180:1cm and 0.5cm);
\draw (5,0) arc(0:180:1cm and 0.5cm);
\draw (8,0) arc(0:180:1cm and 0.5cm);
\draw (11,0) arc(0:180:1cm and 0.5cm);

\draw[black] (1.5,-0.5) node {$h_1\gamma$};
\draw[black] (4,-0.5) node {$h_2\gamma$};
\draw[black] (7,-0.5) node {$h_3\gamma$};
\draw[black] (10,-0.5) node {$h_4\gamma$};

\filldraw[black] (3.5,2) circle (1pt) node[anchor=east]{$a$}; 
\filldraw[black] (7.5,2) circle (1pt) node[anchor=east]{$b$}; 
\filldraw[black] (-2,2) circle (1pt) node[anchor=east]{$o$}; 
\filldraw[black] (12,2) circle (1pt) node[anchor=east]{$p$};

\draw[-stealth, dashed] (-2,1.9) --(0,0);
\draw[-stealth, dashed] (3.5,2) --(3.5,0.4);
\draw[-stealth, dashed] (7.5,2) --(7.5,0.4);
\draw[-stealth, dashed] (12,2) --(11,0.4);

\draw[ultra thick, blue] (0,0) arc(180:40:1cm and 0.5cm);
\filldraw[black] (1.8,0.3) circle (1pt) node[anchor=south]{$\pi_{h_1\gamma}(p)$};
\draw[ultra thick, red] (3.5,0.4) arc(120:30:1cm and 0.5cm);

\draw[ultra thick, red] (6.2,0.3) arc(140:60:1cm and 0.5cm);

\draw [red][decorate,decoration={brace,amplitude=5pt,mirror,raise=4ex}]
  (3.5,-0.4) -- (7.5,-0.4) node[midway,yshift=-3em, red]{$\sum_{\mathcal H_T(o,p)}[a,b]$};
\draw[ultra thick, blue] (9.3,0.4) arc(130:30:1cm and 0.5cm);
\filldraw[black] (9.3,0.4) circle (1pt) node[anchor=south]{$\pi_{h_4\gamma}(o)$};

\filldraw[red] (4.9,0.2) circle (1pt) node[anchor=north]{ $\pi_{h_2\gamma}(b)$};
\filldraw[red] (6.2,0.32) circle (1pt) node[anchor=south]{ $\pi_{h_3\gamma}(a)$};

\filldraw[black] (0,0) circle (1pt) node[anchor=north]{$\pi_{h_1\gamma}(o)$};
\filldraw[black] (10.8,0.26) circle (1pt) node[anchor=west]{\quad $\pi_{h_4\gamma}(p)$};
\end{tikzpicture}

\caption{In this example there are $4$ cosets in $\mathcal H_T(o,p)$ and the sum $\sum_{\mathcal H_T(o,p)}[a,b]$ is taken over the cosets $h_2 \gamma$ and $h_3 \gamma$. For all other cosets of $\mathcal H_T(o,p)$, the projections $\pi_{h\gamma}(a)$ and $\pi_{h\gamma}(b)$ coincide. }\label{fig:dist_sum}
\end{figure}

  The following lemma says that distance formula sums can be used to get lower bounds on distances in $X$.
   
  \begin{lemma}
 \label{distance}
 Fix a loxodromic WPD element $g$, and a basepoint $x_0\in X$. Then for all sufficiently large $T$ the following holds.
 
 	For all $a,b \in G$, we have 
 	\[
 	 d_{X}(ax_0,bx_0) \geq \frac{1}{2} \sum_{\mathcal{H}_T(a,b)} [a,b].
 	\]
 \end{lemma}
 
  \begin{proof}
The following claim is shown in the proof of \cite[Theorem 6.8]{dahmani2014hyperbolically} (showing that $E(g)$ is geometrically separated) via \cite[Theorem 4.42 (c)]{dahmani2014hyperbolically}. \\
  
 \begin{claim}
 
\cite[Theorem 6.8, Theorem 4.42 (c) ]{dahmani2014hyperbolically}: For all constants $R>0$, there exists a constant $B_1>0$ such that the following holds. For cosets $h\gamma\neq h'\gamma$ we have \[\diam\Big( \mathcal N_R(h\gamma x_0) \cap \mathcal N_R(h'\gamma x_0) \Big) < B_1.
  \]
  
  \par\medskip
  
   \end{claim}

  %where the superscripts "$X$" indicate that the neighborhoods and diameter are taken in $X$.\\

Let $\alpha$ be a geodesic in $X$ between $ax_0,bx_0$. By Lemma \ref{lem:exo_hyperbolicspace}, there is a constant $R$ such that for each $h_i \gamma \in \mathcal H_T(a,b)$ there  exist points $p_i, q_i \in \alpha$ such that $d_{X}(p_i, \pi_{h_i\gamma x_0}(ax_0)) \leq R$ and  $d_{X}(q_i, \pi_{h_i\gamma x_0}(bx_0)) \leq R$. Further, the geodesic between $p_i$ and $q_i$ lies within the $R$-neighbourhood of $(h_i\gamma) x_0$. For each $i$, let $\beta_i$ be the sub-geodesic of $[p_i,q_i]$ with endpoints at distance $B_1$ from $p_i$ and $q_i$ respectively, where $B_1$ is from the claim above. If we choose $T$ large enough, then we can ensure that this subgeodesic $\beta_i$ exists for all $h_i \gamma \in \mathcal H_T(a,b)$. We note that $\beta_i \subset [p_i,q_i] \subseteq \mathcal N_R(h_i \gamma x_0)$. \\

\begin{claim}
For $i \neq j$, we have $\beta_i \cap \beta_j = \emptyset$
\end{claim} 

\begin{proof}[Proof of claim 2]
For a contradiction, assume there exists $x \in  \beta_i \cap \beta_j$. Then $[p_i,q_i]$ and $[p_j,q_j]$ share a subgeodesic (containing $x$) of length $B_1$. Such subgeodesic is contained in both $\mathcal N_R(h_i \gamma x_0)$ and $\mathcal N_R(h_i \gamma x_0)$, a contradiction with Claim 1.
% Then for all $i \neq j$, we have $p_i,q_i \not\in [p_j,q_j]$ and vice versa. We show this for $p_i \not\in [p_j,q_j]$, the other cases are similar. If $p_i \in [p_j,q_j] $ then $p_i \in \mathcal N_R(h_j \gamma) \cap N_R(h_i \gamma) $ and also $x \in N_R(h_j \gamma) \cap N_R(h_i \gamma) $. Further, by construction of $\beta_i$ we have $d(x, p_i) \geq B_1$ and hence $\diam \Big(N_R(h_j \gamma) \cap N_R(h_i \gamma) \Big) \geq d(x, p_i) \geq B_1$, a contradiction with the claim above. Similarly for the other $3$ cases.  Hence $p_i,q_i \not\in [p_j,q_j]$ and vice versa. \\
% This means that $d(p_i,q_j)=d(p_i,x)+d(x,q_i)+d(q_i,p_j)+d(p_j,x)+d(x,q_j) < d(p_i,x)+d(x,q_j)$ leading to $x=q_i$ a contradiction. 
\end{proof}

Now, if $T$ is large enough, for all $h_i \gamma \in \mathcal H_T(a,b)$ we have
\[\ell(\beta_i) \geq d_X(p_i,q_i)-2B_1 \geq d_{h_i\gamma}(ax_0,bx_0) -2B_1-2R \geq d_{h_i\gamma}(ax_0,bx_0)/2.\] 
	
Hence 
	
	\[d_X(ax_0,bx_0)  \geq \sum_{h_i\gamma\in\mathcal{H}_T(a,b)} \ell(\beta_i) \geq  \frac{1}{2}\sum_{h_i\gamma\in\mathcal{H}_T(a,b)} [a,b] \]
	 as the $(\beta_i)_i$ are pairwise disjoint.
\end{proof}
 The following lemma will be used below to prove that the ``distance formula" is coarsely Lipschitz and it will also be used in the proof of lemma \ref{o(log)}. 
 \begin{lemma}
\label{lem:2cosetsmax}
Fix a loxodromic WPD element $g$. For any $T\geq 10 B$ which satisfies Lemma \ref{linearorder}, with $B$ as in Lemma \ref{lem:Behrstock}, the following holds.
Let $o,p,a\in G$. Then there at most $2$ cosets $h\gamma \in \mathcal{H}_T(o,p)$ such that $\pi_{h\gamma}(a)$ is distinct from both $\pi_{h\gamma}(o)$ and $\pi_{h\gamma}(p)$.
\end{lemma}

\begin{proof}
Assume this is not the case and let $h_1\gamma \prec h_2\gamma \prec h_3\gamma$ in $\mathcal{H}_T(o,p)$ be such that $\pi_{h_i\gamma}(a)\neq \pi_{h_i\gamma}(o)$ and $\pi_{h_i\gamma}(a)\neq\pi_{h_i\gamma}(p)$ for $i=1,2,3$.

Since $d_{h_2\gamma}(o,p)\geq T\geq 10B$, we have either $d_{h_2\gamma}(a,p)>B$ or $d_{h_2\gamma}(a,o)>B$. The two cases are symmetrical,  hence without loss of generality assume $d_{h_2\gamma}(a,p)>B$.

Since $\pi_{h_2\gamma}(p)=\pi_{h_2\gamma}(h_3\gamma)$ by Lemma \ref{linearorder} 
 this leads to $d_{h_2\gamma}(a,h_3\gamma)>B$ and so by the strong Behrstock inequality (Lemma \ref{lem:Behrstock}), we get $\pi_{h_3\gamma}(a)=\pi_{h_3\gamma}(h_2\gamma)$. Again by by Lemma \ref{linearorder}, we have $\pi_{h_3\gamma}(h_2\gamma)=\pi_{h_3\gamma}(o)$, a contradiction.
\end{proof}

We now show that the distance formula sum is coarsely Lipschitz.

\begin{lemma}
\label{lem:coarsely_lip_sum}
Fix a loxodromic WPD element $g$, and a basepoint $x_0\in X$. For any $T\geq 10 B$ which satisfies Lemma \ref{linearorder} and Lemma \ref{distance}, with $B$ as in Lemma \ref{lem:Behrstock}, there exists a constant $L$ such that the following holds. Let $o,p,a,b\in G$. Then
$$\sum_{\mathcal H_T(o,p)}[a,b]\leq L d_X(ax_0,bx_0)+L.$$
\end{lemma}
 
 \begin{proof}
 By Lemma \ref{lem:2cosetsmax}, there is a set $\mathcal A$ consisting of at most 4 cosets $h\gamma\in \mathcal H_T(o,p)$ such that one of $\pi_{h\gamma}(a)$ and $\pi_{h\gamma}(b)$ does not coincide with either $\pi_{h\gamma}(o)$ or $\pi_{h\gamma}(p)$. Note also that, for a given $h\gamma\in \mathcal H_T(o,p)$, if both $\pi_{h\gamma}(a)$ and $\pi_{h\gamma}(b)$ coincide with either $\pi_{h\gamma}(o)$ or $\pi_{h\gamma}(p)$, then either $\pi_{h\gamma}(a)=\pi_{h\gamma}(b)$ or $h\gamma\in \mathcal H_T(a,b)$. In view of this, we have
 $$\sum_{\mathcal H_T(o,p)}[a,b]\leq \sum_{h\gamma\in \mathcal A} d_{h\gamma}(a,b)+\sum_{\mathcal H_T(a,b)}[a,b].$$
 Since closest-point projections on quasi-convex sets in hyperbolic spaces are coarsely Lipschitz, and the first sum has at most 4 terms, we see that the first sum is bounded linearly in $d_X(ax_0,bx_0)$. The second one is also bounded linearly in $d_X(ax_0,bx_0)$ by Lemma \ref{distance}, and we are done.
 \end{proof}

\subsection{Logarithmic random projections}

We conclude this section with a lemma about projections of Markov chains (where in the statement we use the notation associated to the loxodromic WPD element $g$ established above). 

The lemma says, roughly, that it is very likely that the Markov chain creates logarithmically sized projections, this follows from Lemma \ref{markovword}.

\begin{lemma}
\label{etalogproj}
Let $(w_n^o)$ be a tame Markov chain on a group $G$ containing a loxodromic WPD element for the action on some hyperbolic space. Then there exist constants $ \eta, U  >0$  such that for each basepoint $p\in G$ and each $n\geq 1$ we have
	\[
	 \mathbb{P}\left[\exists m\leq n,\ h\in G: d_{h\gamma}(p,w^p_m)\geq \eta\log(n)\right]\geq 1- e^{-\sqrt{n}/U}.
	\]
\end{lemma}

\begin{proof}
Let $\eta,U$ be as in Lemma \ref{markovword} (but note that the $\eta$ that satisfies the corollary will be smaller). 

By Lemma \ref{markovword}, for any $n$ with probability at least  $1- e^{-\sqrt{n}/U}$ there exist $i,j\leq n$ such that $(w^o_{i})^{-1}w^o_{j}=g^{\lfloor\eta \log(n)/l_G(g)\rfloor}$.

Then $d_{w_i\gamma}(w^o_i,w^o_j)\geq d_{\gamma}(1,g^{\lfloor\eta \log(n)/l_G(g)\rfloor})\geq \eta'\log(n)$, for some sufficiently small $\eta'$, so that $d_{w_i\gamma}(o,w^o_m)\geq \eta'\log(n)/2$ for either $m=i$ or $m=j$, as required.
\end{proof}

% \section*{TEMP}

% Proposal:

% \begin{itemize}
%  \item 2.12 gets split in 2, super-divergent implies WPD, and WPD implies everything.
%  \item T doesn't get fixed in Section 2.
%  \item 2.14 and 2.15 say ``for all sufficiently large T''.
%  \item 3.1 says ``there exists a WPD $g$ such that for all sufficiently large $T$ there exists $C$ such that the following holds. In the notation of [the second part of 2.12], ...''
 
% \end{itemize}

\section{Our assumptions}
\label{sec:assump}

In this short section we spell out the two sets of assumptions under which we work for the rest of the paper, and point out examples of groups and spaces that satisfy them.

\begin{assumptions}
\label{assump:superdiv}
$G$ is a finitely generated group acting on a hyperbolic space $X$ and containing a super-divergent element.
\end{assumptions}

In particular, note that Assumption \ref{assump:superdiv} is satisfied by hyperbolic and relatively hyperbolic groups and corresponding spaces as in Lemma \ref{lem:examples_of_superdivgelements} and Lemma \ref{lem:rh_superdiv}.

%Let $G$ be a finitely generated group, with finite generating set $S$. We let $B_n(G,S)=\{ g \in G : d_S(1,g) \leq n\}$. Given two non-decreasing positive functions $f,g$, we say that they are \textit{equivalent} if there exists a constant $C$ such that for all $n \in \mathbb N_{>0}$: $$f(n/C) \leq g(n) \leq f(Cn)$$
%The \textit{growth rate} of $G$ is defined as the equivalence class of the function $\vert B_n(G,S) \vert$. We note that given two generating sets of $G$, the respective growth rate are equivalent. Hence the growth rate does not depend on the generating set and hence gives an invariant of the group. We say that a group $G$ has \textit{subexponential growth} if $$ \frac{\log\vert B_n(G,S) \vert}{n} \to 0  $$ as $n \to + \infty$, with respect to any (all) finite generating set $S$. Any group with subexponential growth is amenable. 

Before stating the next assumption, we need a definition.
\begin{definition}
\label{defn:subexp}
Let $G$ be a group with a fixed word metric $d_G$, and let $H<G$. We say that $H$ has \emph{subexponential growth in the ambient word metric} if the function $\psi:\mathbb{N} \to \mathbb{N}$ defined as $\psi(n)= \vert \{h \in H : d_G(1, h) \leq n \} \vert$ is subexponential in $n$.
\end{definition}

\begin{assumptions}
\label{assump:graph}
$G$ is the finitely generated fundamental group of a graph of groups where the edge groups (more precisely, their images in $G$) have subexponential growth in the ambient word metric, and there exists a loxodromic WPD element for the action on the Bass-Serre tree $X$.
\end{assumptions}

The assumption applies, for example, to many right-angled Artin group, as we now show.

\begin{lemma}
Any right-angled Artin group $A_\Gamma$ whose defining graph $\Gamma$ has a separating simplex and diameter at least 3 satisfies Assumption \ref{assump:graph}.
\end{lemma}  

\begin{proof}
The separating simplex $\Delta$ gives a splitting over an abelian subgroup $A_\Gamma=A_{\Gamma_1} *_\Delta A_{\Gamma_2}$ where the $\Gamma_i$ are subgraphs of $\Gamma$ intersecting at $\Delta$. Since $A_\Gamma$ is CAT(0), said subgroup is undistorted and hence has subexponential growth in the ambient word metric (as well as its own word metric).

Let $a, b$ be vertices of $\Gamma$ in two distinct connected components of the complement of the disconnecting simplex. Then the element $ab$ of $\Gamma$ is not contained in a vertex group of the splitting (for example this can be checked by passing to the abelianisation of $A_\Gamma$) and hence it is loxodromic for the action on the Bass-Serre tree.

Towards proving that $ab$ is WPD, we note that the vertices $A_{\Gamma_1}$ and $ab A_{\Gamma_1}$ of the Bass-Serre tree have trivial common stabiliser (recall that the vertices of the Bass-Serre tree are cosets of the vertex groups). This can be seen for example using the normal form from \cite[Proposition 3.2]{HermillerMeier}: The stabilisers of $A_{\Gamma_1}$ and $ab A_{\Gamma_1}$ are $A_{\Gamma_1}$ and $ab A_{\Gamma_1}b^{-1}a^{-1}$, so that the normal form for elements in the former will not contain the letter $b$, while those for the latter will.

This fact suffices to show that $ab$ is WPD since it implies that two vertices on the axis of $ab$ (the closest-point projections of the aforementioned vertices) have trivial common stabiliser, which means that \cite[Corollary 4.3]{minasyan2017acylindrical} applies and gives that $ab$ is WPD.
\end{proof}

\subsection{3-manifold groups}

In this subsection we show that each acylindrically hyperbolic 3-manifold group satisfies one of the assumptions above, for a suitable hyperbolic space $X$.

We recall the following well-known result on some obstructions to acylindrical hyperbolicity, see for example \cite{minasyan2017acylindrical}.

\begin{lemma}[ \cite{minasyan2017acylindrical}]
\label{obstructionacylindricalhyperbolicity}
Let $G$ be a group such that one of the following holds:

\begin{itemize}
    \item $G$ contains an infinite cyclic normal subgroup $Z$, or
    \item $G$ is virtually solvable.
\end{itemize}
 Then $G$ is not acylindrically hyperbolic.
\end{lemma}

By a result of Minasyan-Osin \cite{minasyan2017acylindrical} we know that most $3$-manifold groups are acylindrically hyperbolic; we will need a different version of their characterisation of when a $3$-manifold group is acylindrically hyperbolic. Using various results from the literature, including the geometrisation theorem, we can get the following result. 

\begin{proposition}
\label{acylindricallyhyperbolicmanifold}
	Let $M$ be a closed, connected, oriented $3$-manifold. Let $G:= \pi_1(M)$. Then $G$ is acylindrically hyperbolic if and only if $M$ satisfies one of the following:
	
	\begin{itemize}
		\item $M$ is not prime and not $\mathbb{RP}^3 \# \mathbb{RP}^3$,
		\item $M$ is a geometric manifold with Thurston geometry $\mathbb{H}^3$,
		\item $M$ is a non-geometric prime manifold.
	\end{itemize}
	Moreover, when $G$ is acylindrically hyperbolic then it either contains a super-divergent element for some action on a hyperbolic space, or it admits a graph of groups decomposition as in Assumption \ref{assump:graph}.
	\end{proposition}

\begin{proof}
We use the geometrisation theorem, see \cite{perelman2003ricci}, \cite{perelman2}, \cite{LottKleiner}, \cite{morganTian}, \cite{CaoZhu}. 
If $M= M_1 \# M_2\#\cdots \# M_n$ is the prime decomposition of $M$, then $G:=\pi_1(M)$ is hyperbolic relative to $\pi_1(M_i)$. In this case, $G$ is acylindrically hyperbolic unless it is virtually cyclic. This only happens when $M$ is the connected sum of two real projective $3$-spaces. Indeed, the only free product of non-trivial groups which is virtually cyclic is the free product of two copies of $\mathbb Z/2\mathbb Z$, and the only closed, connected, oriented $3$-manifold with this fundamental group is $\mathbb{RP}^3$.

We can therefore assume that $M$ is a prime manifold.
%(i.e. we cannot write $M= N_1 \# N_2$ where $N_i \neq S^3$). 

First, we consider the case where $M$ is geometric. By using Lemma \ref{obstructionacylindricalhyperbolicity},  we will show that the only acylindrically hyperbolic geometric $3$-manifolds are the ones with geometry $\mathbb{H}^3$.  We list the other possible geometries for a prime manifold as well as the reason why they are not acylindrically hyperbolic. We only note the possible fundamental groups, non acylindrical hyperbolicity will follow from Lemma \ref{obstructionacylindricalhyperbolicity}. We refer to \cite{scottmanifolds} for the details on the possible fundamental groups for manifolds with the following Thurston geometries: \\

\begin{itemize}
	\item $S^3$: finite fundamental group.
	\item $\mathbb{R}^3$: virtually abelian fundamental group.
	\item $\mathbb{H}^3$: non-elementary hyperbolic hence acylindrically hyperbolic.
	\item $S^2 \times \mathbb{R}$: virtually cyclic fundamental group.
	\item $\mathbb{H}^2 \times \mathbb{R}$: fundamental group contains an infinite normal cyclic subgroup. 
	\item $\widetilde{SL_2(\mathbb{R})}$: fundamental group contains an infinite normal cyclic subgroup.
	\item Nil: virtually nilpotent fundamental group. 
	\item Sol: virtually solvable fundamental group. 
\end{itemize}

We therefore only need to consider the case where $M$ is a non-geometric prime manifold. We have $2$ cases to consider: 
\begin{itemize}
	\item $M$ is a non-geometric graph manifold. In this case, $\pi_1(M)$ acts acylindrically on the Bass-Serre tree, see \cite[Lemma 2.4]{Wilton2008ProfinitePO}. Hence $G$ is acylindrically hyperbolic.
	\item $M$ contains a hyperbolic component.  In view of Dahmani's combination theorem \cite{Dahmani}, $\pi_1(M)$ is hyperbolic relative to infinite-index abelian and graph manifold groups and hence acylindrically hyperbolic.
\end{itemize}

Note that by the arguments above, if $M$ is not prime (and not  $\mathbb{RP}^3 \# \mathbb{RP}^3$) or if $M$ contains a hyperbolic component then it is relatively hyperbolic with infinite-index peripherals, and so $G$ contains a super-divergent element by Lemma  \ref{lem:rh_superdiv}.

If $M$ is prime with geometry $\mathbb{H}^3$, then $G$ is non-elementary hyperbolic and hence contains a super-divergent element by Lemma \ref{lem:examples_of_superdivgelements}. \\

Therefore, the only acylindrically hyperbolic $3$ manifold groups left to consider are  non-geometric graph manifolds. It might be worth noting that by a result of Gersten \cite{Gersten}, graph manifolds are exactly the closed $3$-manifolds with quadratic divergence, and it particular non-geometric graph manifolds do not contain super-divergent elements by Lemma \ref{divergence}. 

The fundamental group of a non-geometric graph manifold admits a graph of groups decomposition (the one coming from the geometric decomposition) where the edge groups are isomorphic to $\mathbb{Z}^2$. Moreover, the edge groups are undistorted. This can be seen, for example, from the fact that the universal cover of the graph manifold is quasi-isometric to that of a \emph{flip} graph manifold in a way that preserves the geometric decomposition \cite{KapovichLeeb}. Since flip graph manifolds are CAT(0), their abelian subgroups are undistorted, we see that edge groups are undistorted in fundamental groups of flip graph manifolds, and hence the same holds for any graph manifold.
\end{proof}

\section{Geometric arguments}
\label{geometricarguments}

Our proof of Theorem \ref{linear progress} (linear progress for Markov chains) is in two parts. The goal of the first part is to establish Proposition \ref{axiom} below. We do so separately for the two assumptions from Section \ref{sec:assump} and in both cases we make use of the geometric features of the groups under consideration. For the second part of the proof, which takes place in the next section and is probabilistic in nature, Proposition \ref{axiom} can be used as a black box. In particular, we note that whenever one establishes Proposition \ref{axiom} for some group $G$, then $G$ will also satisfy the conclusion of Theorem \ref{linear progress}.

The content of the proposition is roughly the following. Suppose that we start a Markov path at $o$ and at some stage we get to $p$. At that point we have created certain large projections on cosets of some $E(g)$, for $g$ a fixed loxodromic WPD, and the proposition says that, continuing the Markov path past $p$, it is unlikely that we undo much of these large projections.
% The proposition says, roughly, that it is exponentially unlikely that a Markov chains has a specified sequence of projections on axes of a certain loxodromic WPD. 

% 
% Now that we have introduced the definitions and a key observation of tame Markov chains, we will establish some geometric arguments that will allow us to show the following result which is crucial and is the main geometric result needed to prove Theorem \ref{linear progress}.

In Subsection \ref{subsec:WPD} we introduced various objects and notations related to a loxodromic WPD $g$ and the unique maximal elementary subgroup $E(g)=\gamma$ containing $g$. In particular we introduced certain projections $\pi_\gamma$ and projection distances $d_{h\gamma}$ (see Lemma \ref{lem:Behrstock} and Notation \ref{not:d_hgamma}), certain sets of cosets $\mathcal H_T(o,p)$ where projections are large, and sums of projection distances over such sets $\sum_{\mathcal H(x,y)}[z_1,z_2]$ (Definition \ref{defn:H_T}). Once a loxodromic WPD has been fixed, we will freely refer to these objects constructed from the given loxodromic WPD.

\begin{proposition}
\label{axiom}
Let $G$ be a group acting on a hyperbolic space $X$ satisfying one of Assumption \ref{assump:superdiv} or Assumption \ref{assump:graph}, and fix a basepoint $x_0\in X$. Then there exists a loxodromic WPD $g$ and a constant $T_0$ such that the following holds. Consider a tame Markov chain $(w^o_n)$ on $G$. For each $T\geq T_0$ there exists a constant $C$ such that for all $o,p\in G$,  $n \in \mathbb{N}$, and  $t>0$ we have
\[ \mathbb{P}\left[\exists r\leq n : \sum_{\mathcal{H}_T(o,p)} [p,w^p_r] \geq t \right] \leq Ce^{-t/C}.
%\[ \mathbb{P}\left(\sum_{\mathcal{H}_T(o,p)} [o,w^p_n] \leq \sum_{\mathcal{H}_T(o,p)} [o,p]-t \right) \leq Ce^{-t/C}.
\]
\end{proposition}

Each subsection will be about proving this Proposition for the different assumptions on $G$.
 %In the first subsection we will consider the case where $G$ contains a super-divergent element for an action on some hyperbolic space $X$.  In the second subsection, we prove this proposition for the fundamental group of a graph of groups where edge groups have sub-exponential growth in the ambient metric.

\subsection{Proof of Proposition \ref{axiom} for groups containing a super-divergent element}

We now fix some notation and constants. In this subsection, let $G$ be a group containing a super-divergent element $g$ for the action on a hyperbolic space $X$, and we fix the rest of the notation set above Proposition \ref{axiom}, for $g$ the given super-divergent element (which is loxodromic WPD by Lemma \ref{lem:super_div_WPD}), and $T$ chosen as described below. We also fix a word metric $d_G$ on $G$, with finite generating set $S$. Since the projections $\pi_{h\gamma}$ from Lemma \ref{lem:Behrstock} are bounded distance from $X$-projections onto $h\langle g \rangle$, the super-divergent property holds with $\pi=\pi_{h\gamma}$, with respect to the projection distances as in Notation \ref{not:d_hgamma}, meaning the following. There exist a constant $\theta > 0$ and a super-linear function $f:\mathbb{R}^{+} \to \mathbb{R}^{+}$ (possibly larger than the ones for $\langle g\rangle$) such that the following holds for all $h\in G$: for all $d >0$ and paths $p$  remaining outside of the $d$-neighbourhood of $h\gamma$, if $d_{h\gamma}(p_{-},p_{+}) > \theta$ then $\ell(p) > f(d)$. Furthermore, we fix the constant $B$ of Lemma \ref{lem:Behrstock} and we also fix any $T\geq 10(B+\theta)$ that satisfies Lemma \ref{linearorder} and such that Lemma \ref{distance} holds with $T$ replaced by $T-4B$. Finally, we set $\mathcal H(x,y)=\mathcal H_T(x,y)$.

The following lemma says that with high probability, a tame Markov chain moves away at linear speed from the fixed cosets we are projecting to.

\begin{lemma}
\label{20francs}
There exist constants $D,C_1\geq 1$ such that for all $o,p,q\in G$ and $n \in \mathbb{N}$ we have
\[\mathbb{P}\left[ w_n^q \in \mathcal{N}_{n/D}\left(\bigcup_{h\gamma\in \mathcal H(o,p)}h\gamma \right)  \right] \leq C_1  e^{-n/C_1}.\]
\end{lemma}

\begin{proof}
Let $K$ be a bound on the size of the jumps of $w_n^o$, that is, let $K$ be such that $d_G(w^p_n,w^p_{n+1})\leq K$ for all $p\in G$ and $n$ (see Remark \ref{rmk:bounded_jumps}). Up to increasing $K$ we can also assume that for all $a,b\in G$ we have $d_X(ax_0,bx_0)\leq K d_G(a,b)$.

We begin with the following claim, stating that for all points $q\in G$ the number of cosets from $\mathcal{H}(o,p)$ that intersect a ball of radius $Kn$ and centre $q$ is linear in $n$. \\

\begin{claim}
 For all $n$, we have $$\# \{ h\gamma \in \mathcal H(o,p) : h\gamma \cap B(q,Kn) \neq\emptyset \} \leq 4K^2n/(T-4B)+2 $$ where $B$ is from Lemma $\ref{lem:Behrstock}$ and $T$ is the threshold for $\mathcal{H}(o,p)$, see the discussion above the statement of the lemma for details. 
\end{claim}

\begin{proof}[Proof of claim]
Fix $n$ and let $h_1\gamma \prec h_2\gamma \prec \dots h_r\gamma $ be the cosets such that $h_i\gamma \in \mathcal{H}(o,p) $ and $h_i\gamma \cap B(q,Kn) \neq \emptyset$ (the linear order is from Lemma \ref{linearorder}). We therefore want to show that $r \leq 4K^2n/(T-4B)+2$.\\

If $r<2$ we are done, so let us assume $r\geq 2$ and pick $a\in h_1\gamma\cap  B(q,Kn)$ and $b\in h_r\gamma\cap  B(q,Kn)$. For each $1<i<r$ we have $\pi_{h_i\gamma}(a)\subseteq \pi_{h_i\gamma}(h_1\gamma)=\pi_{h_i\gamma}(o)$ (where we used the equivalent characterisation of the linear order), as well as $\pi_{h_i\gamma}(b)\subseteq \pi_{h_i\gamma}(p)$. In particular, we have $d_{h_i\gamma}(a,b)\geq T-4B$ (the "$-4B$" is due to the fact that projections have diameter at most $2B$ as a consequence of Lemma \ref{lem:Behrstock}).

% Let $x \in h_1\gamma $ such that $d_{h_1\gamma}(x, p) >B$ and $d(x,q) \leq Kn+B$. Similarly, let $y \in h_r\gamma $ such that $d_{h_r\gamma}(y, o) >B$ and $d(y,q) \leq Kn+B$. By the linear order on the cosets and the strong Behrstock inequality (Lemma \ref{lem:Behrstock}) we have the following : for all $r\geq i>1$ : $\pi_{h_i\gamma}(x)=\pi_{h_i\gamma}(h_1\gamma)=\pi_{h_i\gamma}(o)$. Similarly, we have for all $r>j \geq 1:$ 
% $\pi_{h_j\gamma}(y)=\pi_{h_j\gamma}(h_r\gamma)=\pi_{h_j\gamma}(p)$.\\

Hence, for all $1<i<r$ we have $d_{h_i\gamma}(o,p)=d_{h_i\gamma}(x, y)$. 

This leads to \[
2Kn\geq d_G(a,b) \geq d_X(a x_0, bx_0)/K\geq  \frac{1}{2K} \sum_{\mathcal{H}_{T-4B}(x,y)}[x,y] \geq (r-2)(T-4B)/(2K)
\]
where we used Lemma \ref{distance} (and our choice of $T$). This proves the claim. 
\end{proof}

Now, the inclusion of $E(g)$ into $G$ is a quasi-isometric embedding. Hence for each coset $h\gamma \in \mathcal H(o,p)$ there are at most $K'n$ elements of $h\gamma$ in the ball $B(q,Kn)$ for some constant $K'$. Hence, by the fact that our Markov chain is tame (and we use Definition \ref{defn:tame}-\ref{item:non-amen}), for $S$ the fixed finite generating set for $G$, we get that  
\begin{align*}
    \begin{split}
        \mathbb{P}\left[ w_n^q \in \mathcal{N}_{n/D}\left(\bigcup_{h\gamma\in \mathcal H(o,p)}h\gamma \right)  \right] &\leq  \# \{ h\gamma \in \mathcal H(o,p) : h\gamma \cap B(q,Kn) \neq\emptyset \} K'n\vert S\vert^{n/D} A \rho^n \\ &\leq (4K^2n/(T-4B)+2 ) K'n \vert S\vert^{n/D} A \rho^n  \\ &\leq Mn^2(\vert S\vert^{1/D} \rho)^n
    \end{split}
\end{align*}
for some $M>0$, where $A, \rho $ are from the `non-amenability' condition from the definition of tameness (Definition \ref{defn:tame}-\eqref{item:non-amen}). This decays exponentially if we choose $D$ large enough ($D$ such that $1/D < -\log_{\vert S\vert }(\rho)$ works).
\end{proof}

% Assume this is not the case and let $h_1\gamma \prec h_2\gamma \prec h_3\gamma $ be such that $\pi_{h_i\gamma}(a)\neq \pi_{h_i\gamma}(b)$ for $i=1,2,3$. 
% %We note that by Lemma \ref{distance}, the fact that $d_G(a,b) \leq 2B$ implies that $\sum_{
% Then we consider the following two cases : \\

% \underline{Case 1: } $d_{h_2\gamma}(a,o) \leq 3B$.\\

% Hence $d_{h_2\gamma}(a,p) \geq d_{h_2\gamma}(o,p)-d_{h_2\gamma}(a,o) \geq T-3B \geq 7B$ by the choice of $T$. In particular, this leads to $d_{h_2\gamma}(a,h_3\gamma)>B$ and so by the strong Behrstock inequality (Lemma \ref{lem:Behrstock}), we get $\pi_{h_3\gamma}(a)=\pi_{h_3\gamma}(h_2\gamma)$. On the other hand, by the linear order on $\mathcal{H}(o,p)$, we get $d_{h_2\gamma}(b, h_3\gamma) \geq d_{h_2\gamma}(a, h_3\gamma)-d_{h_2\gamma}(a,b)\geq 6B$ and again by the strong Behrstock inequality we get $\pi_{h_3\gamma}(b)=\pi_{h_3\gamma}(h_2\gamma)$. A contradiction as $\pi_{h_3\gamma}(a)\neq \pi_{h_3\gamma}(b)$. \\

% \underline{Case 2: } $d_{h_2\gamma}(a,o) > 3B$. \\

% In this case, $d_{h_2\gamma}(a,h_1\gamma)\geq d_{h_2\gamma}(a,o)-d_{h_2\gamma}(o,h_1\gamma) \geq 2B$ where we use the linear order on $\mathcal{H}(o,p)$. Similarly, $d_{h_2\gamma}(b,h_1\gamma) \geq d_{h_2\gamma}(a, h_1\gamma)-d_{h_2\gamma}(a,b) >2B-B=B$. Using the strong Behstrock inequality, again this leads to $\pi_{h_1\gamma}(a)=\pi_{h_1\gamma}(b)$, a contradiction. 

% Hence there at most $2$ cosets such that $\pi_{h\gamma}(a)=\pi_{h\gamma}(b)$.

% \end{proof}

In view of the previous lemma, it is of interest to study how projections of a path change assuming that the path moves away at linear speed from the cosets we are projecting to. We will apply the following lemma to a sample path of our Markov chain, but it holds for all discrete paths with $K$-bounded jumps for some $K$, that is, sequences of points with consecutive ones being distance at most $K$ apart.

\begin{lemma}
\label{o(log)}
For each $D,K>0$, there exists a function $\phi : \mathbb{R}^{+} \to \mathbb{R}^{+} $ with $\phi(n)=o(\log(n))$ such that the following holds. Let $o,p\in G$ and consider a discrete path $\alpha=(\alpha_n)_{n\geq 0}$ with $K$-bounded jumps and such that for all $n \in \mathbb{N}$ we have
$$d\left(\bigcup_{h\gamma \in \mathcal{H}(o,p)} h\gamma, \alpha_n\right) \geq n/D.$$
Then \[
\sum_{\mathcal{H}(o,p)} [\alpha_0,\alpha_n] \leq \phi(n).\]
\end{lemma}

\begin{proof}
Recall that $f$ is, roughly, the function controlling super-divergence of the fixed super-divergent element. We define the sequence $(\Psi_i)_{i\geq 1}$  recursively: \\

\begin{itemize}
    \item Let $\Psi_1$ be such that for all $\Psi\geq \Psi_1$ we have $\frac{1}{K}f\left(\frac{\Psi}{D}-K\right)\geq 2\Psi$ (which exists since $f$ is super-linear).

    \item $\Psi_{j+1}=\frac{1}{K}f\left(\frac{\sum_{i\leq j}\Psi_i}{D}-K\right)$.
\end{itemize}

Let \[u_{j}:= \sum_{i\leq j} \Psi_{i}.\] 

Note that by our choice of $\Psi_1$, we have $u_j\geq\Psi_j\geq 2^{j-1} \Psi_1$ for all $j$.

\par\medskip

\begin{claim}
For all $j\geq 1$ and $u_j< n\leq u_{j+1}$ we have \[  \sum_{  \mathcal{H}(o,p)} [\alpha_{u_j},\alpha_{n}] \leq 4\theta. \]
\end{claim} 

\begin{proof}[Proof of Claim 1]

\vspace{2mm}
Assume that the claim does not hold. Define $\alpha_{u,v}$ to be a concatenation of geodesics from $\alpha_i$ to $\alpha_{i+1}$ for $u\leq i<v$.

 We show below that there exists $h\gamma$ such that $d_{h\gamma}(\alpha_{u_j},\alpha_{n}) \geq \theta$. This is enough to prove the claim. Indeed, since $\alpha_{u_j,n} $ stays outside of the $(u_j/D-K)$-neighborhood of $h\gamma$,  by super-divergence of $g$, this leads to \[K( u_{j+1}-u_j) \geq \ell(\alpha_{u_j,n}) > f\left(\frac{u_j}{D}-K\right)= K\Psi_{j+1}. \]

a contradiction as $u_{j+1}-u_j= \Psi_{j+1}$. \\

Therefore to prove the claim it remains to show that there exists an $h\gamma$ such that $d_{h\gamma}(\alpha_{u_j},\alpha_{n}) \geq \theta$. Note that if for some $h\gamma\in\mathcal H(o,p)$ we have $\pi_{h\gamma}(\alpha_{u_j})=\pi_{h\gamma}(o)$ and $\pi_{h\gamma}(\alpha_{n})=\pi_{h\gamma}(p)$ (or vice versa), then we are done since $T\geq \theta$. Hence, assume that this is not the case. By Lemma \ref{lem:2cosetsmax}, there are at most 4 cosets $h\gamma\in \mathcal H(o,p)$ where one of $\pi_{h\gamma}(\alpha_{u_j})$ and $\pi_{h\gamma}(\alpha_{n})$ does not coincide with either $\pi_{h\gamma}(o)$ or $\pi_{h\gamma}(p)$. We have that $\sum_{  \mathcal{H}(o,p)} [\alpha_{u_j},\alpha_{n}]$ is in fact a sum over the aforementioned cosets, and since there are at most 4 of them, one of the cosets satisfies the required property.
% The second case to consider is when $\sum_{  \mathcal{H}(o,w)} [w,w_{u_{j+1}}]- \sum_{ \mathcal{H}(o,w)} [w,w_{u_{j}}] > 2\theta$ but for all cosets $h \gamma$ we have $d_{h\gamma}(w,w_{u_{j+1}})-d_{h\gamma}(w,w_{u_j}) < \theta$.
\end{proof}

In view of Lemma \ref{lem:coarsely_lip_sum}, there exists $L$ (independent of $\alpha$, $o$, $p$) such that for all $i$ we have $\sum_{\mathcal H(o,p)}[\alpha_i,\alpha_{i+1}]\leq L$. We note that this $L$ is actually bigger than the one from Lemma \ref{lem:coarsely_lip_sum} and depends on $K$.

In view of the claim we have, for $\phi(n):=\max \{j: u_j \leq n \}$,
\[
\sum_{\mathcal{H}(o,p)} [\alpha_0,\alpha_n] \leq\sum_{\mathcal{H}(o,p)} [\alpha_0,\alpha_{u_1}]+\dots \sum_{\mathcal{H}(o,p)} [\alpha_{u_{\phi(n)-1}},\alpha_{n}] \leq L\Psi_1+ 4\theta\phi(n).
\]

Therefore, we are done once we prove the following claim.
 
\begin{proof}
\[\phi(n)=o(\log(n))\]
\end{proof} 

\begin{proof}[Proof of Claim 2]
It suffices to show that for all $E>1$ we have $u_j\geq E^j$ for all sufficiently large $j$, which implies that $\phi(n)\leq \log_E(n)$ for all sufficiently large $n$.

%In turn, it suffices to show that for all $E$ we have $g(u_j)\geq 2E u_j$ for all sufficiently large $j$. 
We have $u_{j+1}=u_j+g(u_j)$ for $g(x):= \frac{1}{K}f(\frac{x}{D}-K)$, and recall that $u_j\geq 2^{j-1}\psi_1$. Since $g$ is super-linear, for all sufficiently large $j$ we have $g(u_j)\geq 2E u_j$, and hence for those same $j$ we have $u_{j+1}\geq 2E u_j$, and the claim follows.
\end{proof}

This proves the Lemma.
\end{proof}

The consequence of the lemma that we will need is the following Corollary, which says, roughly, that it is very unlikely that projections of a sample path of our Markov chain change faster than logarithmically in the number of steps.

\begin{corollary}
\label{undoproba}
For all $m, m' >0$, there exists a constant $C_2 >0$ such that the following holds. Let $o,p \in G$ be basepoints and $w^{o}_r$ a tame Markov chain. Then for all $t>0$ and $k \leq e^{mt/ \eta}$, where $\eta$ is as in Lemma \ref{etalogproj}, we have 
\[
\mathbb{P}\left[ \exists r \leq k : \sum_{\mathcal{H}(o,p)} [p,w^p_r] \geq t/m' \right] \leq C_2 e^{-t/C_2}.
\]
\end{corollary}

\begin{proof}
For all $t>0$, we let $i_t = \lfloor \frac{t}{(m'+1)LK} \rfloor $ where $L$ is from Lemma \ref{lem:coarsely_lip_sum} and $K$ is a bound on the size of the jumps of the Markov chain in $X$, meaning that $d_X(w_n^qx_0,w_{n+1}^qx_0)\leq K$ for all $q\in G$ and $n$ (see Remark \ref{rmk:bounded_jumps}). Let $A_{r,t}$ be the event ``for all $i_t \leq n \leq r : d\left(w^p_{n}, \bigcup_{h\gamma \in \mathcal{H}(o,p)} h\gamma\right) \geq n/D$", where $D$ is from Lemma \ref{20francs}. \\

\begin{claim}

There exists $t_1$ such that for $t \geq t_1$ and $r \leq e^{mt/\eta}$, the event $A_{r,t}$ implies that  $$\sum_{\mathcal{H}(o,p)} [p,w^p_r] < t/m'.$$ 
\end{claim}
\begin{proof}[Proof of Claim]
We note that the $K$-discrete path $(w^p_r)_{r\geq i_t}$ satisfies Lemma \ref{o(log)}. Hence there exists a function (independent of $o,p$) $\phi(n)=o(\log(n))$ such that if $A_{r,t}$ holds then  $\sum_{\mathcal{H}(o,p)} [w^p_{i_t},w^p_r] \leq \phi(r-i_t)$. Using the triangular inequality (Remark \ref{rem:triangle_inequality}), the coarsely Lipschitz property (Lemma \ref{lem:coarsely_lip_sum}) and the fact that $(w^p_n)_{n\geq 0}$ has $K$-bounded jumps, we get the following:

\begin{align*}
  \begin{split}
      \sum_{\mathcal{H}(o,p)} [p,w^p_r] &\leq \sum_{\mathcal{H}(o,p)} [p,w^p_{i_t}]+\sum_{\mathcal{H}(o,p)} [w^p_{i_t},w^p_r] \\ & \leq L d_X(px_0,w^p_{i_t}x_0)+L+ \phi(r-i_t) \\ & \leq LKi_t+L+\phi(r-i_t) \\
  \end{split}  
\end{align*}

Now by the fact that $\phi(r-i_t) =o(\log(r-i_t)) $ we can find $t_1$ such that for all $t \geq t_1$ : $$ t/m'-LKi_t-L > \phi(r-i_t)
$$

 Hence, for $t\geq t_1$ we get that $\sum_{\mathcal{H}(o,p)} [p,w^p_r] < t/m'$ as required. 
\end{proof}

By the claim, for $t\geq t_1$ we can have $\sum_{\mathcal{H}(o,p)} [p,w^p_r] \geq t/m'$ only in $A_{r,t}^c$. Also, we have
$$\mathbb P\left[\bigcup_r A_{r,t}^c\right] \leq \sum_{i\geq i_t} C_1 e^{-i}/C_1\leq C_2 e^{-i_t/C_2},$$
for $C_1$ as in Lemma \ref{20francs} and some $C_2$ depending on $C_1$.

Therefore,

$$\mathbb{P}\left[ \exists r \leq k : \sum_{\mathcal{H}(o,p)} [p,w^p_r] \geq t/m' \right] \leq \mathbb P\left[\bigcup_r A_{r,t}^c\right]\leq C_2 e^{-i_t/C_2}\leq C_2e^{- t /(m'+1)KLC_2}e^{1/C_2},$$
as required for $t\geq t_1$. To cover the case $t\leq t_1$ we can just increase the constant resulting from the previous computation.
\end{proof}

We are now ready to prove the main result of this section for groups containing a super-divergent element.

\begin{proof}[ Proof of Proposition \ref{axiom} for groups containing a super-divergent element]
Recall that we have to prove that there exists a constant $C$ such that for all $o,p\in G$,  $n \in \mathbb{N}$, and  $t>0$ we have
$$ \mathbb{P}\left[\exists r\leq n : \sum_{\mathcal{H}_T(o,p)} [p,w^p_r] \geq t \right] \leq Ce^{-t/C}.
%\[ \mathbb{P}\left(\sum_{\mathcal{H}_T(o,p)} [o,w^p_n] \leq \sum_{\mathcal{H}_T(o,p)} [o,p]-t \right) \leq Ce^{-t/C}.
$$

The idea of the proof is the following. As we run the Markov chain starting at $p$, it might start creating projections on the cosets in $\mathcal H(o,p)$, but it is unlikely that this happens fast by Corollary \ref{undoproba}. At the same time, it is likely that the Markov chain will create a new projection, by Lemma \ref{etalogproj}. Before creating more projections on the cosets in $\mathcal H(o,p)$, the Markov chain would have to undo said new projection, and this can be turned into a sort of recursive formula which will yield the desired bound. We now proceed to the actual proof.\\

Let $t\geq 10T$ and define $$p_{n,t}=\sup_{o,p \in G} \mathbb P\Big[\quad \exists r\leq n \quad : \sum_{\mathcal{H}(o,p)} [p,w^p_r] \geq t\Big] $$\\

We will show that there exists a constant $C>0$ such that $p_{n,t} \leq Ce^{-t/C}$, which is exactly the required inequality (except for the requirement "$t\geq 10 T$" but we can increase the constant to cover the other case).\\

The main claim is the following.\\

\begin{claim}
There exists $C'$ such that we have $p_{n,t}\leq C' e^{-t/C'}+p_{n,2t}$.
\end{claim} 

\begin{proof}[Proof of Claim]
Let $\epsilon>0$ be arbitrary.
%Let $\epsilon < e^{10T/C_2}-1$ where $C_2$ is from Corollary \ref{undoproba}, with $m=2$ and $m'=5$.
Choose $o,p$ such that $$p_{n,t}\leq (1+\epsilon) \mathbb P\Big[\quad \exists r\leq n \quad : \sum_{\mathcal{H}(o,p)} [p,w^p_r] \geq t\Big].$$
For $o,p \in G,  t>0$, let $\mathcal A_{o,p,t} \subseteq G$ be the set of all $q\in G$ such that:

\begin{enumerate}
    \item There exists $h\gamma$ such that $d_{h\gamma}(p,q)\geq 5t$.
    \item $$\sum_{ \mathcal H(o,p)} [p,q]\leq t/2.$$
\end{enumerate}

\begin{figure}[h]
\centering
\begin{tikzpicture}[scale=0.6]
\draw[black] (0,0) to[out=20,in=160] (2.5,0); 
\draw[black] (3,0) to[out=20,in=160] (5.5,0); 
\draw[black] (-3,0) to[out=20,in=160] (-0.5,0); 
\draw[black] (6,0) to[out=20,in=160] (8.5,0);
\draw[black] (9,0) to[out=20,in=160] (11.5,0);

\draw[red] (7,2) to[out=100,in=230] (7.5,7);  
\draw[red] (8,7) node{$h_q\gamma$};  
\filldraw[black] (5,7.5) circle (1pt) node[anchor=east]{$q$}; 
\draw[dashed] (5,7.5) --(6.9,6);
%\draw[>=triangle 45, <->, blue] (7,5.9)-- (7,2.6) ;
\draw[blue] (8.9,4.4) node {$\geq 5t$};

\filldraw[black] (6.9,6) circle (1pt) node[anchor=west]{$\pi_{h_q\gamma}(q)$}; 
\filldraw[black] (6.9,2.5) circle (1pt) node[anchor=west]{$\pi_{h_q\gamma}(p)$};

\filldraw[black] (-3,2.5) circle (1pt) node[anchor=east]{$o$};   
\filldraw[black] (12,2.5) circle (1pt) node[anchor=east]{$p$};
\filldraw[black] (11.6,0) circle (1pt) node[anchor=west]{$\pi_{h'\gamma}(p)$};
 
\draw[dashed] (-3,2.5) --(-2.5,0.2);
\draw[dashed] (12,2.5) --(11.6,0);
\draw[dashed] (5,7.5) --(7,0.2);

\filldraw[black] (7,0.2) circle (1pt) node[anchor=north]{$\pi_{h\gamma}(q)$};  

\draw [decorate,decoration={brace,amplitude=5pt,mirror,raise=4ex}]
  (7,0.2) -- (11.6,0.2) node[midway,yshift=-3em]{$\sum_{\mathcal H(o,p)}[p,q] \leq t/2$};
  
  \draw [decorate,decoration={brace,amplitude=5pt,raise=4ex}, blue]
  (6.9,5.8) -- (6.9,2.5) node[midway,yshift=-3em]{};

\end{tikzpicture}
\caption{Illustration of $q\in A_{o,p,t}$}
\end{figure}
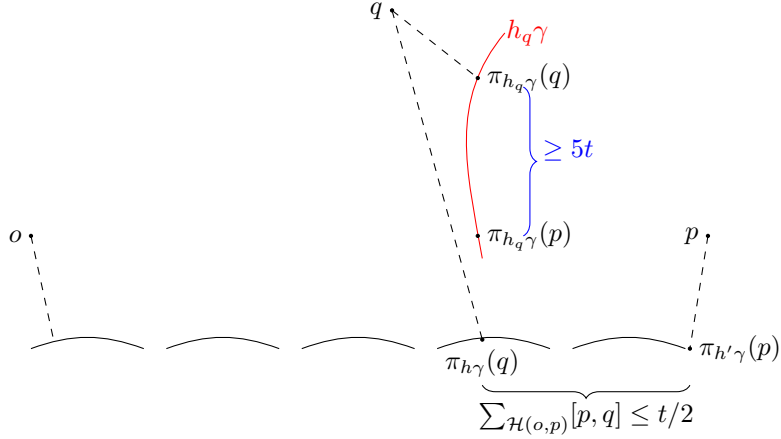

 We consider the following events, where $\eta$ is from Lemma \ref{etalogproj}.

\begin{enumerate}
    \item $\mathcal B^{t}_1$ : `` $\forall k \leq e^{5t/\eta}, \quad \forall h \in G : \quad d_{h\gamma}(p,w^p_{k}) <5t $".  This is the event that the Markov chain is not creating a new projection quickly. \\
    \item $\mathcal B^{t}_2$ : ``$\exists k <e^{5t/\eta} : \quad \sum_{\mathcal H(o,p)} [p,w^p_k] > t/2 $". This is the event that the Markov chain is undoing $t/2$-worth of projections on the original cosets $\mathcal H (o,p)$ quickly.\\
    \item Let  $\mathcal B^{t}_3(k,q)$ be the event ``$w^p_k=q$ and for all $i<k : w^p_i \not\in \mathcal A_{o,p,t}$ ".
    %This is the event that the Markov chain has created a new projection (quickly) and that it hasn't undone more than $t/2$-worth of old projections on the old cosets. 
\end{enumerate}

%We note that all possibilities are covered by \[\mathcal B^t_1 \cup \mathcal B_2^t \cup \bigsqcup_{\substack{k \leq e^{5t/\eta} \\ q\in \mathcal A_{o,p,t}}} \mathcal B^{t}_3(k,q) \]

We call $B_3^t=\bigsqcup_{\substack{k \leq e^{5t/\eta} \\ q\in \mathcal A_{o,p,t}}} \mathcal B^{t}_3(k,q)$. 

\par\medskip

\textbf{Subclaim: } Let $q\in \mathcal A_{o,p,t}$ and let $w \in G $ be such that $\sum_{\mathcal H(o,p)}[p,w]\geq t $. Then $\sum_{\mathcal H(o,q)}[q,w] \geq 2t$. 

\begin{proof}[Proof of Subclaim] 
We let $h_q \gamma$ be the coset satisfying the first item in the definition of $q$ belonging to the set $\mathcal A_{o,p,t}$.

We first note that $h_q \gamma \not\in \mathcal H(o,p)$. Indeed, if it was then $\sum_{\mathcal H(o,p)}[p,q] \geq 5t$, a contradiction with the second item in the definition of $\mathcal A_{o,p,t}$. Hence $d_{h_q\gamma}(o,p) <T $. To prove the claim, it is therefore enough to show that $d_{h_q\gamma}(q,w) \geq 2t$, as $h_q \gamma \in \mathcal H(o,q)$.

Let $k\gamma \in \mathcal H(o,p)$ be any coset. Then either $d_{k\gamma}(o,h_q\gamma)>B$ or $d_{k\gamma}(p,h_q\gamma)>B$ and so by the strong Behrstock inequality (Lemma \ref{lem:Behrstock}) we get that either $\pi_{h_q\gamma}(o)=\pi_{h_q\gamma}(k\gamma)$ or  $\pi_{h_q\gamma}(p)=\pi_{h_q\gamma}(k\gamma)$. \\

For a contradiction, assume that $d_{h_q\gamma}(q,w) < 2t$. Then
$$d_{h_q\gamma}(p,w) \geq d_{h_q\gamma}(p,q)-d_{h_q\gamma}(w,q) \geq 3t,$$
so that $d_{h_q\gamma}(o,w)\geq 3t-T$.

The cases $\pi_{h_q\gamma}(p)=\pi_{h_q\gamma}(k\gamma)$ or $\pi_{h_q\gamma}(o)=\pi_{h_q\gamma}(k\gamma)$ are similar so assume that $\pi_{h_q\gamma}(o)=\pi_{h_q\gamma}(k\gamma)$. Then $d_{h_q\gamma}(w,k\gamma)=d_{h_q\gamma}(w,o)\geq 3t-T>B$ and so by the strong Behrstock inequality (Lemma \ref{lem:Behrstock}), this leads to
$$\pi_{k\gamma}(w)=\pi_{k\gamma}(h_q\gamma).$$
We also have that $d_{h_q\gamma}(q,k\gamma) = d_{h_q\gamma}(q,o)>5t-T >B$. Again by the strong Behrstock inequality, this leads to
$$\pi_{k\gamma}(h_q\gamma)=\pi_{k\gamma}(q)$$
for all $k\gamma \in \mathcal H(o,p)$ and so $ \pi_{k\gamma}(q)=\pi_{k\gamma}(w)$ for all $k\gamma \in \mathcal H(o,p)$. Hence,

$$
t\leq \sum_{\mathcal H(o,p)} [p,w] \leq \sum_{\mathcal H(o,p)} [p,q] \leq t/2,
$$ a contradiction. Hence $d_{h_q\gamma}(q,w) \geq 2t$ and in particular $\sum_{\mathcal H(o,q)}[q,w] \geq 2t$, proving the claim.
\end{proof}

%We note that $\mathbb P\Big[ \exists   r \leq e^{5t/\eta} \quad : \sum_{\mathcal{H}(o,p)} [p,w^p_r] \geq t \quad \big\vert \quad \mathcal (B^{t}_1 \cup B^t_2)^C \Big] =0 $.  Hence we get the following chain of inequalities, which we explain below.

We note that the event ``$\exists r\leq n : \sum_{\mathcal{H}(o,p)} [p,w^p_r] \geq t$" is contained in the union of the events $\mathcal B^{t}_1$, $\mathcal B^{t}_2$, and the intersection of $B^t_3$ with the event ``$\exists e^{5t/\eta} \leq r\leq n : \sum_{\mathcal{H}(o,p)} [p,w^p_r] \geq t$". This leads to the following chain of inequalities, which we explain further below and where the constant $C_2$ is from Corollary \ref{undoproba} with $m=5$ and $m'=2$.

\begin{align*}
\begin{split}
\mathbb P\Big[\exists r\leq n : &\sum_{\mathcal{H}(o,p)} [p,w^p_r] \geq t\Big]   \\
& \leq \mathbb{P}[\mathcal B^{t}_1] + \mathbb{P}[\mathcal B^{t}_2]+ \sum_{\substack{k   \leq e^{5t/\eta} \\ q\in \mathcal A_{o,p,t} }} \mathbb P\left[\exists e^{5t/\eta} \leq r\leq n : \sum_{\mathcal{H}(o,p)} [p,w^p_r] \geq t  \quad   \Big\vert \quad   \mathcal B^{t}_3(q,k) \right] \mathbb P \Big[ B_3^t(k,q)\Big] \\
& \leq \mathbb{P}[\mathcal B^{t}_1] + \mathbb{P}[\mathcal B^{t}_2]+ \sum_{\substack{k   \leq e^{5t/\eta} \\ q\in \mathcal A_{o,p,t} }} \mathbb P\left[\exists e^{5t/\eta} \leq r\leq n : \sum_{\mathcal{H}(o,q)} [q,w^p_r] \geq 2t  \quad   \Big\vert \quad   \mathcal B^{t}_3(q,k) \right] \mathbb P \Big[ B_3^t(k,q)\Big] \\
&\leq \mathbb{P}[\mathcal B^{t}_1] + \mathbb{P}[\mathcal B^{t}_2]+ \sum_{\substack{k   \leq e^{5t/\eta} \\
q\in \mathcal A_{o,p,t} }}  \mathbb P\left[\exists r'\leq n-k : \sum_{\mathcal H(o,q)}[q,w_{r'}^q] \geq 2t \right]  \mathbb P \Big[w^p_k=q\Big] \\
 &\leq e^{-\sqrt{e^{5t/\eta}}/U} + C_2e^{-t/C_2}+p_{n,2t}
\end{split}
\end{align*}

The first inequality is a union bound plus conditioning. We then use the subclaim  for the second inequality. The third inequality follows from the usual Markov property (which says that we can just condition on "$w^p_k=q$" instead) and Lemma \ref{lem:Markov_property_set}. We then use Lemma \ref{etalogproj} to get the term ``$e^{-\sqrt{e^{5t/\eta}}/U}$" and Corollary \ref{undoproba} (with $m=5$ and $m'=2$) to get the term ``$C_2e^{-t/C_2}$". The last probability is from the definition of $p_{n,t}$ which is defined as the supremum over all basepoints $o,p$.

Hence, $$
p_{n,t} \leq (1+\epsilon)(e^{-\sqrt{e^{5t/\eta}}/U} + C_2e^{-t/C_2}+p_{n,2t}) \leq (1+\epsilon)C'e^{-t/C'}+(1+\epsilon)p_{n,2t}
$$ for some $C'>0$ depending on $C_2,\eta,U$. Since $\epsilon$ was arbitrary, this proves the claim.\\
\end{proof}

We note that for a fixed $n$, if $u>L(Kn+1)$ then $p_{n,u}=0$ where $K$ is from the bounded jumps condition (see Remark \ref{rmk:bounded_jumps}) and $L$ is the coarsely Lipschitz constant from Lemma \ref{lem:coarsely_lip_sum}. Let $s= \lceil \log_2(L(Kn+1)/t) \rceil$ so that $p_{n,2^st}=0$. 
Inductively, we get \begin{align*}
\begin{split}
p_{n,t} &\leq C'\sum_{i=1}^{s} e^{-2^{i}t/C'} \\
& \leq C'\sum_{i=1}^{\infty} e^{-it/C'} \\
%& = C_2 \frac{(1+\epsilon)e^{-t/C_2}-(1+\epsilon)^{s+1}e^{-(s+1)t/C_2}}{1-(1+\epsilon)e^{-t/C_2}} \\
& \leq Ce^{-t/C},
\end{split} 
\end{align*}
for some $C$ depending on $C'$, as we wanted. This proves Proposition \ref{axiom} for groups containing a super-divergent element. 
\end{proof}

We proved Proposition \ref{axiom} for groups containing a super divergent element. We will now show that this result also holds for some particular graphs of groups. 

\subsection{Proof of Proposition \ref{axiom} for graph of groups with subexponential growth edge groups}

In this subsection let $G = \pi_1( \mathcal{G})$ be the fundamental group of a graph of groups $\mathcal{G}$ where each edge group has subexponential growth, in a fixed ambient word metric $d_G$, and fix any loxodromic WPD $g$ for the action on the Bass-Serre tree $X$, for which we also fix a basepoint $x_0$. We also fix the rest of the notation set above Proposition \ref{axiom} for the given $g$, for some fixed $T\geq 10B$, for $B$ as in Lemma \ref{lem:Behrstock}, that satisfies both Lemma \ref{linearorder} and Lemma \ref{distance}. We also assume that $T\geq 10M$ for $M$ as in Lemma \ref{lem:X_e_proj} below (we will not use $T$ between now and then). We set $\mathcal H(x,y)=\mathcal H_T(x,y)$.

\subsubsection{Bass-Serre theory preliminaries}
\label{subsec:BS_prelims}

Since $\gamma=E(g)$ is virtually cyclic, with $g\in G$ acting loxodromically, it stabilises a line in the Bass-Serre tree. Similarly:

\begin{definition}
For every coset $h \gamma$, there is a bi-infinite line which is stabilised by $h\gamma h^{-1}$. We call this line the \textit{axis} of $h\gamma$ and denote it by $Axis(h\gamma)$. Note that $Axis(h\gamma)=h Axis(\gamma)$ for al $h$.
\end{definition}

\begin{lemma}
\label{lem:geodesic_in_axis}
There exists a constant $D$, depending only on the loxodromic WPD $g$ such that the following holds. If $a,b \in G$ are such that $d_{h\gamma}(a,b) >D$ then any geodesic in $X$ from $ax_0$ to $bx_0$ contains an edge in the axis of $h\gamma$, and moreover this edge can be chosen at distance at least $d_{h\gamma}(a,b)/10$ from $ax_0$.
\end{lemma}

\begin{proof}
Note that there exists $C$ such that for all $h\in G$ we have $d_{\text{Haus}}(h \gamma x_0, Axis(h\gamma)) \leq C$. In fact, $\gamma x_0$ and $Axis(\gamma)$ both lie within finite Hausdorff distance of an orbit of $\langle g\rangle$ (and proving the required bound for $h=1$ suffices).

Hence, if $D$ is sufficiently large and $d_{h\gamma}(a,b)\geq D$ then $ax_0$ and $bx_0$ have closest-point projections onto $Axis(h\gamma)$ that are at least $d_{h\gamma}(a,b)/2$ apart; this is because the projections used to define $d_{h\gamma}$ are within bounded distance of closest-point projections to $h\gamma x_0$, and in turn those lie within bounded distance of closest-point projections to $Axis(h\gamma)$.

Since $Axis(h\gamma)$ is a bi-infinite geodesic line in a tree, we have that the geodesic $[ax_0, bx_0]$ contains the subgeodesic of $Axis(h\gamma)$ connecting the closest-point projections of $ax_0$ and $bx_0$. Said subgeodesic contains an edge at distance $\geq d_{h\gamma}(a,b)/10$ from the closest-point projection $a'$ of $ax_0$ to $Axis(h\gamma)$, and this edge will also have distance $\geq d_{h\gamma}(a,b)/10$ from $ax_0$, as required.
\end{proof}

We now establish some notation related to Bass-Serre theory. In short, to each vertex and edge of the Bass-Serre tree we associate a suitable subspace of $G$.

We fix orbit representatives $v_1,\dots,v_a$ for the action of $G$ on the vertices of $X$, and similarly let $e_1,\dots,e_b$ be orbit representatives for the $G$-action on the edges. We let $G_{v_i}=\Stab(v_i)$ and $G_{e_i}=\Stab(e_i)$. Corresponding to each vertex $v$ of $X$ we have a coset of one of the $G_{v_i}$, which we denote by $X_v$ (namely, if $v=gv_i$ then $X_v=gG_{v_i}$). Similarly, we can associate to each edge $e$ a coset $X_e$ of some $G_{e_i}$.

We define $\Psi: \mathbb{R} \to \mathbb{R}$ to be the maximal of the all the growth functions $\psi_{e_i}$ of the $G_{e_i}$ with respect to the ambient word metric $d_G$, i.e. $\Psi(n):= \max \{ \psi_{e_i}(n) \} $ (see Definition \ref{defn:subexp}). This function is well-defined as $E(\Gamma)$ is finite, and is subexponential as each $\psi_e$ is subexponential.\\

% Now, for all $K$ there exists a constant $K'$ such that the following holds. In the Bass-Serre tree $X$, let $v,v'$ be two vertices with corresponding vertex space $X_v$ and $X_{v'}$ respectively. Let $\mu, \mu'$ be vertices of $\Gamma$ such that $\mu=\check{v}$ and $\mu'=\check{v'}$. Hence $X_v$ is isomorphic to the Cayley graph of $G_{\mu}$ with respect to $J_{\mu}$, details are omitted, see squids. For an edge $e$ in the Bass-Serre tree, we note that $X_e$ is isomorphic to the Cayley graph of $G_{\check{e}}$ with respect to $J_{\check{e}}$.\\

% Then any $K$-path in the Cayley graph, from $a$ to $b$ intersects $\mathcal N_K'(I)$ for all $I \in \mathcal I_{a,b}$, where $$\mathcal I_{a,b} = \{g\tau_{\alpha}(G_{\alpha}), gt_{\alpha}\tau_{\overline{\alpha}}(G_{\alpha}) \quad \vert \quad \alpha= \check{e} \quad \text{for } e \in  [v,v']  \}$$ 

 Now, for all $K$ there exists a constant $K'$ such that the following holds. Any $K$-path in the Cayley graph, from $a$ to $b$ intersects $\mathcal N_{K'}(X_e)$ for all $X_e \in \mathcal I_{a,b}$, where $$\mathcal I_{a,b} = \{X_e \quad \vert \quad  e \in  [v,v']  \}$$ 

where $v,v'$ are vertices of the Bass-Serre tree $X$ such that $a \in X_v$ and $b \in X_{v'}$, and $[v,v']$ is the (unique) geodesic between $v $ and $v'$ in $X$.\\

Finally, we note the following technical lemma.

\begin{lemma}
\label{lem:X_e_proj}
There exists $M>0$ such that the following holds. Let $e$ be an edge of $Axis(h\gamma)$ for some $h\in G$. Then for all $h'\gamma\neq h\gamma$ and all $a\in X_e$ we have $d_{h'\gamma}(x,h\gamma)\leq M$.
\end{lemma}

\begin{proof}
Note that there is a uniform bound on the distance between $ax_0$ and $e$, and in turn $e$ lies uniformly close to $h\gamma x_0$ as noted at the beginning of the proof of Lemma \ref{lem:geodesic_in_axis}. Hence, the claim follows from the fact that closest-point projections to quasi-convex sets in hyperbolic spaces are coarsely Lipschitz.
\end{proof}

\subsubsection{Back to the proof}
Recall that we have to prove that there exists a constant $C$ such that for all $o,p\in G$,  $n \in \mathbb{N}$, and  $t>0$ we have
$$\mathbb{P}\left[\exists r\leq n : \sum_{\mathcal{H}_T(o,p)} [p,w^p_r] \geq t \right] \leq Ce^{-t/C}.
%\[ \mathbb{P}\left(\sum_{\mathcal{H}_T(o,p)} [o,w^p_n] \leq \sum_{\mathcal{H}_T(o,p)} [o,p]-t \right) \leq Ce^{-t/C}.
$$

We note the following claim which will allow us to reduce the probability we are trying to bound to the problem of bounding the probability of two events. Let $o,p \in G$ be any two basepoints. \\

\begin{claim}

For any $w \in G$, if $\sum_{\mathcal H_T (o,p)}[p,w] \geq t$ then either

\begin{itemize}
    \item there exists a coset $h\gamma \in \mathcal H_T(o,p)$ such that $d_{h\gamma}(p,w) \geq t/10$, or
    \item $$\sum_{\substack{h\gamma \in \mathcal H_T(o,p) \\ \pi_{h\gamma}(o) =\pi_{h\gamma}(w)}}[p,w] \geq t/20.$$ 
\end{itemize}
\end{claim}
\begin{proof}[Proof of Claim]
We note that by Lemma \ref{lem:2cosetsmax} there are most 2 cosets such that $\pi_{h\gamma}(w) \neq \pi_{h\gamma}(o), \pi_{h\gamma}(p) $, say $h_1\gamma$ and $h_2\gamma$. Hence, if the first case does not hold we have

\begin{align*}
    \begin{split}
       \sum_{\mathcal H_T (o,p)}[p,w] &= \sum_{\substack{h\gamma \in \mathcal H_T(o,p) \\ \pi_{h\gamma}(o) =\pi_{h\gamma}(w)}}[p,w] + \sum_{\substack{h\gamma \in \mathcal H_T(o,p) \\ \pi_{h\gamma}(o) \neq \pi_{h\gamma}(w)}}[p,w] \\
        & =\sum_{\substack{h\gamma \in \mathcal H_T(o,p) \\ \pi_{h\gamma}(o) =\pi_{h\gamma}(w)}}[p,w]+d_{h_1\gamma}(p,w)+d_{h_2\gamma}(p,w) \\
& \leq \sum_{\substack{h\gamma \in \mathcal H_T(o,p) \\ \pi_{h\gamma}(o) =\pi_{h\gamma}(w)}}[p,w]+t/5
    \end{split}
\end{align*}

Hence if $\sum_{\mathcal H_T (o,p)}[p,w] \geq t$, the second case holds, proving the claim.
\end{proof}

We let $A_1 \subseteq G$ be the set of elements $w$ that satisfy the first bullet point of the claim and $A_2 \subseteq G$ the set of elements satisfying the second bullet point. Hence, by the claim if $\sum_{\mathcal H_T (o,p)}[p,w] \geq t$ then $w \in  A_1 \cup A_2$. Therefore, to prove the proposition it suffices to find a constant $C$ (independent of $p$ and $n$) such that we have
$$\mathbb P[\exists r\leq n : w^p_r\in A_1\cup A_2]\leq C e^{-n/C}.$$

The goal is now to show, roughly, that if the event from the equation above holds then the Markov path must have passed close to some edge space corresponding to an edge contained in some $Axis(h\gamma)$ for $h\gamma\in \mathcal H(o,p)$.

From now on we always assume that $t$ satisfies $t \geq 10 D$ where $D$ is from Lemma \ref{lem:geodesic_in_axis} (as usual, it suffices to consider $t$ large enough). Also, let $K$ be a bound on the size of the jumps of the Markov chain (see Remark \ref{rmk:bounded_jumps}), and let $K'$ be the corresponding constant from the end of Subsection \ref{subsec:BS_prelims}.

For all $h\gamma \in \mathcal H (o,p)$  let
$$
\mathcal I_{h\gamma}=\{X_e \quad \vert \quad  e \in Axis(h\gamma)\}.
$$

Further, we let 
$$
\mathcal {NI} =\bigcup_{h\gamma \in \mathcal H (o,p)}\bigcup_{X_e \in \mathcal I_{h\gamma}} \mathcal N_{K'}(X_e)\subseteq G,
$$

where we emphasise that the neighborhood is taken in $G$.

\par\medskip

\begin{claim}

 There exists $\epsilon>0$ (independent of $t$) such that if the event "$\exists r\leq n : w^p_r\in A_1\cup A_2$" holds, then the event "$\exists \epsilon t \leq k\leq n: w^p_k \in \mathcal {NI}$" also holds.
\end{claim}
\begin{proof}
If $w^p_r\in A_1$ for some $r$, then by Lemma \ref{lem:geodesic_in_axis} we have that a geodesic $[p,w_p^r]$ contains an edge $e$ of $Axis(h\gamma)$ for some $h\gamma\in \mathcal H(o,p)$. Moreover, said edge lies at distance bounded from below in terms of $t$ from $px_0$. By the defining property of $K'$ we have that $w^p_k \in X_e$ for some $k$, and further $k$ needs to be bounded from below in terms of $t$ because the Markov chain makes bounded jumps and the distance from $p$ to $X_e$ can be bounded linearly from below in terms of the corresponding distance in $X$.

The proof in the case where $w^p_r\in A_2$ for some $r$ is similar, but in this case the choice of $e$ is trickier. We consider a minimal element $h\gamma$ (with respect to the order from Lemma \ref{linearorder}) of $\{h'\gamma\in \mathcal H(o,p): \pi_{h'\gamma}(o)=\pi_{h'\gamma}(w^p_r)\}$. Again by Lemma \ref{lem:geodesic_in_axis} we have that a geodesic $[p,w_r^p]$ contains an edge $e$ of $Axis(h\gamma)$. As above, we have to prove that the distance from $p$ to $X_e$ is large. Actually, if $d_{h\gamma}(p,w^p_r)\geq t/40$ we can apply the same argument as above, so we will further assume that this is not the case. Hence

$$\sum_{\substack{h'\gamma \neq h\gamma \\ \pi_{h\gamma}(o) =\pi_{h\gamma}(w^p_r)}}[p,w^p_r] \geq t/40,$$

and in view of Lemma \ref{lem:X_e_proj} (and the fact that $T$ is much larger than the constant from that lemma) we also have a similar bound on the corresponding sum for any given element of $X_e$. Hence, we can conclude using Lemma \ref{distance}.
% and moreover we can choose the edge at distance bounded from below in terms of $d_{h\gamma}(p, w^p_r)$ (not $t$) from $px_0$. We now argue that $e$ in fact lies at distance bounded from below in terms of $t$ from $px_0$, and then one can repeat the same argument as above.

% Suppose that the edge $e$ contains a point $ax_0$ in the orbit of $x_0$. Then it is readily checked from the equivalent characterisations of the linear order that the projections of $a$ to all elements of $\{h'\gamma\in \mathcal H(o,p): \pi_{h'\gamma}(o)=\pi_{h'\gamma}(w^p_r)\}-\{h\gamma\}$ coincide with that of $o$. Given this, an application of Lemma \ref{distance} yields the required bound. If the edge $e$ does not contain a point in the orbit of $x_0$ we cannot apply Lemma \ref{distance}, but we can still choose a point in the orbit close to $e$ (recall that we chose $T$ much larger than the size of a fundamental domain for the $G$-action on $X$) and the same argument applies.
\end{proof}

Now, an argument similar to the Claim in the proof of Lemma \ref{20francs} yields that there exists $M$ such that for each $n\geq 1$ we have
$$\mathcal J_n=\bigcup_{h\gamma \in \mathcal H (o,p)} \{e\in I_{h\gamma}: X_e\cap B(p,Kn+K')\neq \emptyset\}\leq Mn.$$

Setting $\mathcal NI_n=\bigcup_{e\in \mathcal J_n} \mathcal N_{K'}(X_e\cap B(p,Kn+K'))$ we have that the events "$\exists \epsilon t \leq k\leq n: w^p_k \in \mathcal {NI}$" and "$\exists \epsilon t \leq k\leq n: w^p_k \in \mathcal {NI}_k$" coincide, because of the bounded jumps of the tame Markov chain. Noticing that $|X_e\cap B(p,Kn+K')|\leq \psi(M'n)$ for some $M'$ (where recall that $\psi$ is an upper bound for the growth functions of the edge groups), we can now estimate

\begin{align*}
 \begin{split}
 \mathbb P \left[ \exists r \leq n \quad : \sum_{\mathcal H (o,p)}[p,w^p_r] \geq t  \right] &\leq \mathbb P \Big[ \exists  \epsilon t \leq k\leq n: w^p_k \in \mathcal {NI}_n\Big]\\
 &\leq \sum_{i\geq \epsilon t} A\rho^{i} Mi \psi(M'i).
  \end{split}	
 \end{align*}
where $A,\rho$ are from the non-amenability condition on the Markov chain, Definition \ref{defn:tame}-\ref{item:non-amen}. Since $\psi$ is subexponential, the last term decays exponentially fast in $t$, as required. This concludes the proof of Proposition \ref{axiom} under Assumption \ref{assump:graph}.

\section{Probabilistic arguments}

\label{probabilisticarguments}

The main result of this section is that if Proposition \ref{axiom} holds, then we get that a tame Markov chain makes linear progress in the hyperbolic space being acted on.

\subsection{Proof of Theorem \ref{linear progress}}

We fix the setup of Theorem \ref{linear progress}, and in particular a group $G$ acting on a hyperbolic space $X$ and a tame Markov chain on $G$ starting at $o\in G$ (note that the constants below will be independent of $o$, as stated in the theorem).

We fix a $T \geq T_0$, for $T_0$ as in Proposition \ref{axiom}. We also require that $T$ satisfies both Lemma $\ref{linearorder}$ and Lemma $\ref{distance}$. As $T$ is fixed, for ease of notation, we drop the subscript throughout this proof, i.e. $\mathcal H (o,p) := \mathcal{H}_T (o,p)$. \\

We will show that there exist constant $L>0,C_0 \geq 1$ and $m \in \mathbb N_{>0}$ such that for all integers $j$ we have $$
\mathbb{P}\left[ \sum_{\mathcal{H}(o,w_{jm})}[o,w_{jm}] <  jm/C_0\right]\leq C_0e^{-jm/C_0},
$$

This is enough to prove Theorem \ref{linear progress} as we now argue. By Lemma \ref{distance}, we get \begin{equation}
\label{star}
\mathbb{P}\Big[ d_{X}(ox_0, w_{jm} x_0) <  jm/C'_0 \Big] \leq \mathbb{P}\left[ \sum_{\mathcal{H}(o,w_{jm})}[o,w_{jm}] < jm/C_0\right] \leq C_0e^{-jm/C_0}
\end{equation}

where $C'_0=C_0/2$, and we dropped the basepoint from the notation for simplicity.
Now, let $n$ be any positive integer, which we can assume to be large enough that $n/(C_0) \geq Km$, where $K$ is a bound on the size of the jumps of the Markov chain (Remark \ref{rmk:bounded_jumps}). Write $n=jm+r$ where $0 \leq r \leq m-1$, we note that $d_{X}(ox_0, w_nx_0) \geq d_{X}(ox_0, w_{jm}x_0)-d_{X}(w_{jm}x_0, w_nx_0)$. Hence 

\begin{equation*}
\begin{split}
\mathbb{P}\big[ d_X(ox_0,w_nx_0) <  n/(2C_0) \big]& \leq \mathbb{P}\big[ d_{X}(x_0, w_{jm}x_0)< n/C_0 \big]\\
& \leq C_0e^{-jm/C_0} \leq Ce^{-n/C},
\end{split}
\end{equation*}

for a suitable $C$, as required.

We now prove a similar result to the Claim in the proof of  \cite[Theorem 9.1]{MathieuSisto}.
%We find an upper bound on the expected value, which then allows us to get an upper bound on the desired probability, via Markov's inequality. 

\begin{proposition}
\label{prop:exp_lambda}

\label{expectation}
Let $G$ be a group satisfying the conclusion of Proposition \ref{axiom} with $T_0$ defined as in Proposition \ref{axiom}. Consider a tame Markov chain $(w^o_n)$ on $G$. Then for all $T \geq T_0$,  there exist constants $ \lambda, \epsilon_1 >0$ and $m \in \mathbb N$ such that the following holds. For all $o,p\in G$ we have
$$
\mathbb{E}\left[\exp\left(\lambda \left(\sum_{\mathcal{H}_T(o,p)}[o,p]-\sum_{\mathcal{H}_T(o,w^p_m)}[o,w^p_m]\right)\right)\right] \leq 1-\epsilon_1.
$$
\end{proposition}

We will denote by $\mathcal{H}(o,p)^C$ to be the complement of $\mathcal{H}(o,p)$ in $\mathcal{H}(o,w^p_m)$, i.e. $\mathcal{H}(o,w^p_m)=\mathcal{H}(o,p) \cup \mathcal{H}(o,p)^C$. \\

We begin with a lemma.

\begin{lemma}
\label{lem:small_complement}
There exist $\eta >0$ and a function $u:\mathbb{N}^{+} \to \mathbb{R}$, with $u(m) \to 0 $ as $m \to + \infty$,  such that $$
\mathbb{P} \Big[ \sum_{\mathcal{H}(o,p)^C} [o,w^p_m] \leq \eta \log(m) \Big] \leq u(m).
$$
\end{lemma}

\begin{proof}
Let $\eta, U> 0$ be as in Lemma \ref{etalogproj} and fix $m \geq e^{6T/\eta}$; the "$\eta$" satisfying the conclusion of the lemma will actually be $\eta/2$.

Let $\mathcal A$ be the set of all $q\in G$ such that there exists $h_q\gamma\notin \mathcal H(o,p)$ with $d_{h_q\gamma}(o,q)\geq \eta \log(m)$; for each such $q$ we fix a corresponding $h_q$. For $k \in \mathbb{N}, q \in G$, we consider the events $A_{k,q}= ``q \in \mathcal A, w^p_k=q, w^p_i\notin \mathcal A  \quad\forall i<k $".

Note that if both events $A_{k,q}$ and ``$d_{h_q\gamma}(w^p_k,w^p_m)< \eta \log(m)/2$" hold then $\sum_{\mathcal{H}(o,p)^C} [o,w^p_m] > \eta \log(m)/2$. Indeed, the sum includes the term $d_{h_q\gamma}(o,w^p_m)\geq \eta \log(m)/2$.

\par\medskip

\begin{claim}

There exists a function $g$ with with $g(m) \to 0 $ as $m \to + \infty$ such that for all $m$ we have
$$\mathbb P[\exists k\leq m : w^p_k\in \mathcal A]\geq 1- g(m).$$
\end{claim}
\begin{proof}
If we did not have the requirement $h_q\gamma\notin \mathcal H(o,p)$ then this would follow directly from Lemma \ref{etalogproj}. However, it suffices to combine said lemma with the fact that Proposition \ref{axiom} (together with the fact that projections are coarsely Lipschitz) implies that
$$\mathbb P\Big[\exists k\leq m , h\gamma\in \mathcal H(o,p): d_{h\gamma}( w^p_k, o)\geq \eta \log(m)/2\Big]$$
is bounded by some function of $m$ that goes to $0$ as $m$ tends to infinity.
\end{proof}

Note that for each $k\leq m$ and $q\in \mathcal A$ we have

$$\mathbb P[d_{h_q\gamma}(w^p_k,w^p_m)< \eta \log(m)/2| A_{k,q}]\ =\ \mathbb P[d_{h_q\gamma}(q,w^q_{m-k})< \eta \log(m)/2]\ \geq \ 1- Ce^{-\eta\log(m)/(2C)},$$

where the equality follows from the Markov property (Lemma \ref{lem:Markov_property_set}), while the inequality (and corresponding constant) comes from Proposition \ref{axiom}. The last term can be rewritten as $1-h(m)$ for some function going to $0$ as $m\to +\infty$.

Since the events $A_{k,q}$ are disjoint we can now estimate

\begin{align*}
    \mathbb{P} \Big[ \sum_{\mathcal{H}(o,p)^C} [o,w^p_m] > &\eta \log(m)/2 \Big] \\
    &\geq \sum_{k\leq m, q\in \mathcal A} \mathbb P[d_{h_q\gamma}(w^p_k,w^p_m) < \eta \log(m)/2\ |\ A_{k,q}]\ \mathbb P[A_{k,q}]\\
    & \geq (1-h(m)) \sum_{k\leq m, q\in \mathcal A}\mathbb P[A_{k,q}]\\
    & \geq (1-h(m))(1-g(m)),
\end{align*}

where we used both the Claim and the estimate on conditional probabilities explained above. This proves the lemma.
\end{proof}

\begin{proof}[Proof of Proposition \ref{prop:exp_lambda}]
Using the Cauchy-Schwarz inequality and Remark \ref{rem:triangle_inequality} (the triangular inequality), we get
\begin{align*}
    \begin{split}
       \mathbb{E}\left[\exp\Big(\lambda (\sum_{\mathcal{H}(o,p)}[o,p]-\sum_{\mathcal{H}(o,w^p_m)}[o,w^p_m])\Big)\right] &=\mathbb{E}\left[\exp\left(\lambda \left(\sum_{\mathcal{H}(o,p)}[o,p]-\sum_{\mathcal{H}(o,p)}[o,w^p_m]\right)\right)\exp\left(-\lambda \sum_{\mathcal{H}(o,p)^C} [o,w^p_m]\right)\right]  \\
       &\leq \mathbb{E}\left[\exp\left(2\lambda \left(\sum_{\mathcal{H}(o,p)}[o,p]-\sum_{\mathcal{H}(o,p)}[o,w^p_m]\right)\right)\right]^{\frac{1}{2}}\mathbb{E}\left[\exp\left(-2\lambda \sum_{\mathcal{H}(o,p)^C}[o,w^p_m]\right)\right]^{\frac{1}{2}} \\
       & \leq \mathbb E \left[\exp \left(2\lambda \sum_{\mathcal{H}(o,p)}[p,w^p_m] \right) \right]^{\frac{1}{2}}\mathbb{E}\left[\exp\left(-2\lambda \sum_{\mathcal{H}(o,p)^C}[o,w^p_m]\right)\right]^{\frac{1}{2}}.
    \end{split}
\end{align*}

Now, by Proposition \ref{axiom} there exists a constant $C>0$ such that for all $t>0$: $$
\mathbb{P}\left[ \sum_{\mathcal{H}(o,p)}[p,w^p_m] \geq t\right]\leq Ce^{-t/C}.
$$

which leads to $$
\mathbb{P}\left[\exp \left(2\lambda \sum_{\mathcal{H}(o,p)}[p,w^p_m] \right)\geq s\right]=\mathbb{P}\left[ \sum_{\mathcal{H}(o,p)}[p,w^p_m] \geq \frac{\log(s)}{2\lambda}\right] \leq Cs^{-1/(2\lambda C)}
$$

We recall that for a non-negative random variable $Y$, the expected value is given by:

$$\mathbb{E}(Y) = \int_{0}^{+\infty} \mathbb{P}(Y \geq s)  ds$$

If we choose $\lambda < \frac{1}{2C}$ this leads to 
\begin{equation*}
    \begin{split}
 \mathbb E \left[\exp \left(2\lambda \sum_{\mathcal{H}(o,p)}[p,w^p_m] \right) \right] &=  \int_{0}^{+\infty} \mathbb{P}\left[\exp (2\lambda \sum_{\mathcal{H}(o,p)}[p,w^p_m] )\geq s\right]ds \\
&\leq  \int_{0}^{1} \mathbb{P}\left[\exp (2\lambda \sum_{\mathcal{H}(o,p)}[p,w^p_m] )\geq s\right] ds + C\int_{1}^{+\infty} s^{-1/(2\lambda C)} ds \\
&\leq 1+\frac{2C^2\lambda}{1-2\lambda C} \\
&=\frac{1-2\lambda C (1-C)}{1-2\lambda C} \\
 \end{split}
\end{equation*}

Now, letting $B$ be the event $``\sum_{\mathcal{H}(o,p)^{C}}[o,w^p_m] \leq \eta \log(m)" $ where $\eta $ is from Lemma \ref{lem:small_complement},  we get:

\begin{align*}
    \begin{split}
    \mathbb{E}\left[\exp(-2\lambda (\sum_{\mathcal{H}(o,p)^C}[o,w^p_m])\right] &= \mathbb{E}\left[\exp(-2\lambda (\sum_{\mathcal{H}(o,p)^C}[o,w^p_m]))\mathbb{1}_{B}\right] + \mathbb{E}\left[\exp(-2\lambda (\sum_{\mathcal{H}(o,p)^C}[o,w^p_m]))\mathbb{1}_{B^{C}}\right] \\
    & \leq \mathbb{P}[B] + m^{-2\lambda \eta} \\
& \leq f(m)+m^{-2\lambda \eta} \\
    \end{split}
\end{align*}

where we go from the second to the third by Lemma \ref{lem:small_complement}, with $f(m) \to 0$ as $m \to +\infty$.

This all leads to $$
\mathbb{E}\left[\exp\left(\lambda (\sum_{\mathcal{H}(o,p)}[o,w]-\sum_{\mathcal{H}(o,w^p_m)}[o,w^p_m])\right)\right] \leq \Big(\frac{1-2C\lambda(1-C)}{1-2\lambda C}\Big)^{1/2}(f(m)+m^{-2\lambda \eta})^{1/2}
$$

Recall that we have fixed some $0<\lambda <\frac{1}{2 C}$. We can therefore increase $m$ to get  $$
\mathbb{E}\left[\exp\left(\lambda (\sum_{\mathcal{H}(o,p)}[o,w]-\sum_{\mathcal{H}(o,w^p_m)}[o,w^p_m])\right)\right] \leq 1-\epsilon_1$$
for some $\epsilon_1>0$ as required, proving the Proposition.
\end{proof}

We now prove Equation \eqref{star}. Fix $\epsilon_1, \lambda$ and $m$ from Proposition \ref{expectation}. For any positive integer $j$, we have: $$
\sum_{\mathcal{H}(o,w_{m+jm})}[o,w_{m+jm}]=\sum_{\mathcal{H}(o,w_{m+jm})}[o,w_{jm+m}]+\sum_{\mathcal{H}(o,w_{jm})}[o,w_{jm}]-\sum_{\mathcal{H}(o,w_{jm})}[o,w_{jm}]
$$
and, by Proposition \ref{prop:exp_lambda}, for all $g\in G$ we have 

$$
\mathbb{E}\left[\exp(-\lambda\big( \sum_{\mathcal{H}(o,w_{m+jm})}[o,w_{jm+m}] -\sum_{\mathcal{H}(o,w_{jm})}[o,w_{jm}] ) \hspace{2mm} \Big\vert \hspace{2mm} w_{jm}=g\right] \leq 1-\epsilon_1.
$$
Summing over all $g$, this leads to
\begin{equation*}
\begin{split}
\mathbb{E}&\Big[\exp(-\lambda \sum_{\mathcal{H}(o,w_{m+jm})}[o,w_{jm+m}])\Big] \\
&\leq \sum_{g\in G}\mathbb{E}\Big[e^{-\lambda( \sum_{\mathcal{H}(o,w_{m+jm})}[o,w_{jm+m}] -\sum_{\mathcal{H}(o,w_{jm})}[o,w_{jm}])} \ e^{-\lambda \sum_{\mathcal{H}(o,g)}[o,g]}\hspace{2mm} \Big\vert \hspace{2mm} w_{jm}=g \Big]\mathbb{P}[w_{jm}=g] \\
&\leq (1-\epsilon_1)\sum_{g\in G} e^{-\lambda \sum_{\mathcal{H}(o,g)}[o,g]}\mathbb{P}[w_{jm}=g] \\
%&=\sum_{g\in G}\mathbb{E}\Big(\exp(-\lambda\big( \sum_{\mathcal{H}(o,w_{m+jm})}[o,w_{jm+m}] -\sum_{\mathcal{H}(o,w_{jm})}[o,w_{jm}]) \big)\ \exp(-\lambda \sum_{\mathcal{H}(o,g)}[o,g]\hspace{2mm} \Big\vert \hspace{2mm} w_{jm}=g \Big) \leq \\
&\leq (1-\epsilon_1)\mathbb{E}\Big[\exp(-\lambda(\sum_{\mathcal{H}(o,w_{jm})}[o,w_{jm}])\Big].
\end{split}
\end{equation*}

Inductively, we get for all $j$ :$$
\mathbb{E}\left[\exp(-\lambda(\sum_{\mathcal{H}(o,w_{jm})}[o,w_{jm}]))\right] \leq (1-\epsilon_1)^{j}
$$

Hence, by Markov's inequality, for any $L>0$ we have

\begin{equation*}
    \begin{split}
        \mathbb{P}\left[ \sum_{\mathcal{H}(o,w_{jm})}[o,w_{jm}] < L jm\right] & = \mathbb{P}\Big[\exp(-\lambda(\sum_{\mathcal{H}(o,w_{jm})}[o,w_{jm}])) > \exp(-\lambda L jm)\Big] \\
        & \leq \frac{\mathbb{E}\left[\exp(-\lambda(\sum_{\mathcal{H}(o,w_{jm})}[o,w_{jm}]))\right]}{e^{-\lambda L jm}}\\
        &\leq (1-\epsilon)^{j}e^{\lambda L jm} \\
          \end{split}
\end{equation*}

if we choose $L$ small enough, we can find a $C_0 \geq 1$ such that $$
\mathbb{P}\Big[ \sum_{\mathcal{H}(o,w_{jm})}[o,w_{jm}] <  jm/C_0\Big] \leq C_0e^{-jm/C_0}
$$

as required. By the discussion at the start of this section, this is enough to prove Theorem \ref{linear progress}.
\qed

\section{Applications of linear progress}
\label{applications}

In Theorem \ref{linear progress} we have established that, under suitable conditions and with probability tending to 1 exponentially fast, tame Markov chains are at a `linear' distance from the starting point in a hyperbolic space being acted on. We now encapsulate  this notion of `linear progress' in a definition and then state some of the applications one can derive from a tame Markov chain making linear progress. These will be similar to the applications given in \cite{maher2015random} and \cite{MathieuSisto}. Specifically, we will study deviations from quasi-geodesics in Subsection \ref{subsec:deviation_qgeod}, a central limit theorem in Subsection \ref{subsec:clt} and the translation lengths of Markov elements in Subsection \ref{subsec:transl_length}. In all cases, and in contrast with the proof of linear progress, we can follow known arguments for random walks with small modifications.

\begin{definition}
Consider a group $G$ acting on a metric space $X$, with a basepoint $x_0$.
We say that a Markov chain $(w_n^{o})_n$ on $G$ makes \textit{linear progress in $X$ with exponential decay} if there exists a constant $ C>0$ such that for all basepoints $o \in G$ we have $$
\mathbb{P}\Big[ d_{X}(ox_0,w^{o}_nx_0) \geq n/C \Big] \geq 1-Ce^{-n/C} $$
\end{definition}

\subsection{Deviations from quasi-geodesics for tame Markov chains}
\label{subsec:deviation_qgeod}

We start by showing that, in a suitable sense, tame Markov chains tend not to deviate too far from quasi-geodesics. The arguments are almost identical to \cite[Section 10]{MathieuSisto}, in particular we show a similar result to Theorem 10.7 in the case of tame Markov chains. \\

We will state the deviation inequality in an \textit{acylindrically intermediate space} for $(G,X)$. We omit this definition here but refer the reader to  \cite[Definition 10.1]{MathieuSisto}. Instead, we just note that the following spaces $Y$ are acylindrically intermediate for the specified group $G$ and metric space $X$ (see  \cite[Proposition 10.3]{MathieuSisto} for these and more examples). 

\begin{example}

\begin{itemize}
    \item If $G$ is a finitely generated group acting acylindrically on the geodesic hyperbolic space $X$, then $Y=X$ and $Y= \Cay(G,S)$, for a finite generating set $S$, are both acylindrically intermediate spaces for $(G,X)$. \\
    \item If $G$ is relatively hyperbolic, $X$ is its coned-off Cayley graph and $Y$ is its Bowditch space, then $Y$ is acylindrically intermediate for $(G,X)$. 
\end{itemize}
\end{example}

\begin{theorem}
\label{thmdeviationquasigeo}
Let $G$ be a finitely generated group acting non-elementarily on a hyperbolic space $X$ and let $Y$ be acylindrically intermediate for $(G,X)$ with basepoint $y_0$. Consider a tame Markov chain $(w^{o}_n)_n$ on $G$ making linear progress with exponential decay in $X$. Then for every $D$, there exists a constant $D'$ such that for all $l$, $k <n$, and $p\in G$ we have 
$$
\mathbb{P}\Big[ \sup_{\alpha \in QG_{D}(p,w^p_n)} d_{Y}(w^p_k y_0, \alpha) \geq l \Big] \leq D'e^{-l/D'}
$$

where $QG_{D}(a,b)$ is the set of all $(D,D)$- quasi-geodesics from $ay_0$ to $by_0$ (with respect to $d_{Y}$).
\end{theorem}

\begin{proof}
% The proof is basically identical to the one from Mathieu-Sisto \cite{MathieuSisto} in section 10. We include the statements here for completeness.
We drop the basepoints from the notation for Markov chains as well as from expressions such as $d_X(ax_0,bx_0)$, as they will not be very important in this proof. We fix a bound $K$ on the size of the jumps of the Markov chain in $Y$ (see Remark \ref{rmk:bounded_jumps}), that is $d_{Y}(w_{n},w_{n+1})\leq K$ for all $n$ (and all basepoints).

Similarly to \cite[Section 10]{MathieuSisto}, we say that a discrete path $(w_i)_{i\leq k}$ in $G$ is \textit{tight around} $w_k$ \textit{at scale $l$} if for any $k_1 \leq k \leq k_2$ with $k_2-k \geq l$ we have:

\begin{enumerate}
    \item $d_{X}(w_{k_1}, w_{k_2}) \geq \frac{k_2-k_1}{C} $,
    \item $d_{Y}(w_{k'},w_{k'+1})\leq K$ for every $k'$.
%     \item $l_{Y}((w_{k_i})_{k_1\leq k \leq k_2} \leq C_0 (k_2-k_1)$,
%     \item $d_{Y}(w_k',w_{k'+1}) \leq \max\{l,\frac{\vert k-k'\vert}{ 100 C_1}\}$ for all $k'$.
\end{enumerate}

where the constant $C>0$ is from the definition of linear progress with exponential decay.

We note that in \cite[Section 10]{MathieuSisto}, the definition of tightness has three conditions, the first one being the same as above, while the other two are implied by our second condition. That is, our definition is more restrictive, but it is simpler and sufficient for our purposes. 

\begin{lemma}(cfr. \cite[Lemma 10.11]{MathieuSisto})
\label{probatightscale}

There exists a constant $F>0$ so that for all $n$ and  $k \leq n $ and all $l \geq 1$ :

$$
\mathbb{P}\Big[ \hspace{2mm} (w_i)_{i \leq n} \text{ is tight around } w_k \text{\ at scale l} \hspace{2mm} \Big] \geq 1-Fe^{-l/ F}.
$$
\end{lemma}

\begin{proof}
Note that the fact that our Markov chain has bounded jumps yields point 2 of the definition, with probability 1. We now estimate the probability that point 1 above does not hold.

We claim that for all $k_1\leq k_2$ we have $\mathbb{P}\Big[d_{X}(w_{k_1},w_{k_2}) < (k_2-k_1) / C \Big]\leq Ce^{-(k_2-k_1)/C}$, where $C$ is the constant of linear progress with exponential decay. To see this, we need Lemma \ref{lem:Markov_property_set}, which allows us to perform a "change of basepoint" computation:

$$\mathbb{P}\Big[d_{X}(w_{k_1},w_{k_2}) < (k_2-k_1) / C \Big]$$
$$ \leq \sum_{g\in G}\mathbb{P}\Big[d_{X}(w_{k_1},w_{k_2}) < (k_2-k_1) / C \ | \ w_{k_1}=g\Big]\  \mathbb P[w_{k_1}=g]$$
$$=\sum_{g\in G}\mathbb{P}\Big[d_{X}(g,w^g_{k_2-k_1}) < (k_2-k_1) / C\Big] \mathbb P[w_{k_1}=g],$$
after which we can apply linear progress.

For fixed $k, k_1$ we have

$$\mathbb{P}\Big[ \exists k_2 \geq k : k_2-k \geq l, d_{X}(w_{k_1},w_{k_2}) < (k_2-k_1) / C \Big] \leq \sum_{k_2-k_1=j \geq \max\{l,k-k_1\}} e^{-j/C}\leq F_1 e^{-l/C_1}.
$$
for some constant $F_1>0$; this is where we use linear progress with exponential decay, in the form explained above. We also note that we are implicitly using the Markov property from lemma \ref{lem:Markov_property_set} and the fact that $k_1$ has been fixed. 

We now sum over all possible $k_1$:

$$\mathbb{P}\big[ \exists k_1\leq k\leq k_2  : k_2-k \geq l, d_{X}(w_{k_1},w_{k_2}) < (k_2-k_1) / C \big] $$
$$\leq \sum_{k_1\leq k, k-k_1\leq l}F_1 e^{-l/C_1} + \sum_{ k-k_1=j> l} F_1 e^{-l/C}$$
$$\leq F_1le^{-l/C_1}+F_2e^{-l/F_2},$$

for some constant $F_2>0$, and we are done since this decays exponentially fast in $l$.
\end{proof}

In view of Lemma \ref{probatightscale}, the theorem follows from \cite[Lemma 10.12]{MathieuSisto}, which says that if $(w_i)_{0 \leq i \leq n}$ is tight around $w_k$ at scale $l$ then there exists a constant $F_3>0$ such that $d_{Y}(w_k, \alpha(w_0, w_n)) \leq F_3 l$ (note that the lemma holds for the notion of tightness as in \cite{MathieuSisto}, hence, a fortiori, for the notion of tightness we use here).
\end{proof}
Now that we know that if a tame Markov chain on $G$ makes linear progress in the hyperbolic space $X$ on which it acts, we get ``deviation from quasi-geodesics" results in an acylindrically intermediate space, we can get the following corollary, in the case when $G$ acts acylindrically on the hyperbolic space $X$.

We note that if $G$ acts acylindrically on $X$ and any tame Markov chain makes linear progress with exponential decay in $X$ then the action of $G$ on $X$ is non-elementary in view of the classification of acylindrical actions on hyperbolic spaces \cite[Theorem 1.1]{Osin_acylindrical}.

From here on, given two elements $a,b$ in a group with a fixed word metric, we denote by $[a,b]$ any choice of geodesic in the corresponding Cayley graph connecting them.

\begin{corollary}
\label{devgeod}
Let $G$ be a finitely generated group acting acylindrically (and non-elementarily) on a hyperbolic space $X$, such that any tame Markov chain in $G$ makes linear progress with exponential decay in $X$. Let $H$ be a group quasi-isometric to $G$, with a fixed word metric $d_H$. Consider a random walk $(Z_n)_n$ on $H$ with driving measure whose support is finite and generates $H$ as a semigroup. Then there exists a constant $R$ such that for all $k \leq n$ and $l>0$ we have $$
\mathbb{P}\big[ d_{H}(Z_k, [1,Z_n]) \geq l \big] \leq R e^{-l / R}.
$$
\end{corollary}

\begin{proof}
Since $H$ and $G$ are quasi-isometric, and $G$ is non-amenable then by \cite[Theorem 2]{whyte}] there exists a bijective $R'$-quasi-isometry $f:H\to G$. For $g,h\in G$ we denote by $\alpha(g,h)$ the quasi-geodesic $f([f^{-1}(g),f^{-1}(h)])$ in $G$ between $g$ and $h$.

Let $(w^o_n)$ be the push-forward of $(Z_n)$ to $G$ (see Definition \ref{defn:push_forward}), which is a tame Markov chain by Lemma \ref{pushforwardistame}. We let the basepoint of the Markov chain be $o=f(1)$. Applying Lemma \ref{lem:push_variable} (with $A$ the set of $n$-tuples where the $0$-th coordinate is the identity, and the $k$-th coordinate lies further away than $l$ from the chosen geodesic connecting the identity to the $n$-th coordinate) we get

$$\mathbb{P}\big[ d_{H}(Z_k, [1,Z_n]) \geq l \big] \leq \mathbb{P}\big[ d_{G}(w^o_k, \alpha(o,w^o_k)) \geq l/R'-R' \big].$$

Theorem \ref{thmdeviationquasigeo} then gives the required bound.
\end{proof}

\subsection{Central limit theorem}
\label{subsec:clt}

We recall that, given points $x,y,z$ in a metric space $Z$, the Gromov product $(x,y)_z$ is defined as $\left(d(x,z)+d(z,y)-d(x,y)\right)/2$. If $z$ lies within $d$ of some geodesic between $x$ and $y$, then $(x,y)_z\leq d$.

% The second moment deviation inequality holds for a random walk in $H$, this is a corollary of the deviation inequality from geodesics. We know that, as $X$ is hyperbolic, 
% \[
% d(x,[y,z])-\delta \leq (y,z)_x \leq d(x,[y,z])
% \]
% 
% The second inequality holds in any metric space $X$, not necessarily $\delta$-hyperbolic. 

\begin{lemma}\label{lem:second_moment_dev}
Let $H$ be a group with a given word metric $d_H$ and consider a random walk $(Z_n)$ on $H$ with driving measure whose support is finite and generates $H$ as a semigroup. Then there exists a constant $R$ such that the following holds. For all $k \leq n$ and $l>0$:$$
\mathbb{P}\Big[ d_{H}(Z_k, [1,Z_n]) \geq l \Big] \leq R e^{-l / R}.
$$
Then there exists a constant $R_2>0$, depending only on $R$, such that 
$$
\mathbb{E}\big[ (1,Z_n)_{Z_k}^2 \big] \leq R_2,
$$
where the Gromov product is measured in the metric $d_H$.
\end{lemma}

\begin{proof}
 We have \begin{align*}
 \begin{split}
 \mathbb{E}\big[ ((1,Z_n)_{Z_k})^2 \big]	& = \int_{0}^{+\infty} \mathbb{P}\big[ ((1,Z_n)_{Z_k})^2 \geq t \big] \quad dt \\
 & \leq \int_{0}^{+\infty} \mathbb{P}\big[  d_H(Z_k, [1,Z_n]) \geq \sqrt{t} \big] \quad dt\\
 &\leq R \int_{0}^{+\infty}  e^{-\sqrt{t}/R} dt \\
 & \leq R_2
 \end{split}
 \end{align*}
 
 as the integral converges and is equal to a constant only depending on $R$.
\end{proof}

% The following theorem is from \cite{MathieuSisto} (Theorem 4.2): 
% 
% 
% 
% \begin{theorem}[Theorem 4.2 in \cite{MathieuSisto}]
% Assume $Z=(Z_n)_n$ has finite second moment and satisfies the second moment deviation inequality with respect to $\mu$. Then the law of $\frac{Z_n-ln}{\sqrt{n}}$ under $\mathbb{P}$ weakly converges to a Gaussian law with zero mean and variance $\sigma^{2}(\mu, Z)$
% \end{theorem}
% Hence, we have shown that the C.L.T. holds for a random walk in $H$: 

The lemma above allows us to apply the machinery of \cite{MathieuSisto} to obtain the following central limit theorem.

\begin{theorem}

Let $G$ be a finitely generated group acting acylindrically (and non elementarily) on a hyperbolic space $X$, where any tame Markov chain makes linear progress in $X$ with exponential decay. Let $H$ be a group quasi-isometric to $G$ and let $(Z_n)_n$ be a random walk on $H$ with driving measure whose support is finite and generates $H$ as a semigroup. Then $(Z_n)_n$ satisfies the Central Limit Theorem. i.e. there exist constants $l>0$ and $\sigma>0$ such that $$
\frac{d_H(1,Z_n)-ln}{\sigma^2\sqrt{n}} \to \mathcal{N}(0,1),
$$
where the convergence is a convergence in law.
\end{theorem}

\begin{proof}
By Corollary \ref{devgeod} and Lemma \ref{lem:second_moment_dev}, there exists a constant $R_2>0$ such that $$
\mathbb{E}\big[ (1,Z_n)_{Z_k}^2 \big] \leq R_2$$ (which, in the language of \cite{MathieuSisto}, says that the random walk under consideration satisfies the second moment deviation inequality). \\
% We know that, by Corollary \ref{devgeod}, there exists a constant $R>0 $ such that the following holds. For all $k \leq n$ and $l>0$:  \[
% \mathbb{P}\Big[ d_{H}(Z_k, [1,Z_n]) \geq l \Big] \leq R e^{-l / R}.
% \] 
% Hence, by Lemma \ref{lem:second_moment_dev}, there exists a constant $R_2>0$, depending only on $R$, such that \[
% \mathbb{E}\big[ (1,Z_n)_{Z_k}^2 \big] \leq R_2\] (which, in the language of \cite{MathieuSisto}, says that the random walk under consideration satisfies the second moment deviation inequality). \\

Hence, by  \cite[Theorem 4.2]{MathieuSisto}, this yields that the random walk $(Z_n)_n$ satisfies the Central Limit Theorem.
\end{proof}

\subsection{Translation lengths of Markov elements}
\label{subsec:transl_length}

We recall that, given a group $G$ acting on a metric space $X$, the \textit{translation length} of an element $g\in G$ is defined as$$\tau(g) = \liminf_{n \to + \infty} \frac{d_X(x_0,g^nx_0)}{n}
$$
which is well defined and independent of the point $x_0 \in X$. We recall that an element $g$ is \emph{loxodromic} if and only if $\tau(g) >0$. 
\\

We will establish the following result for a Markov chain, similar to a result of Maher-Tiozzo for random walks \cite[Theorem 1.4]{maher2015random}.  In particular, the result says that loxodromic elements are generic with respect to tame Markov chains.

\begin{theorem}
\label{translationlength}
Let $G$ be a group acting acylindrically on a hyperbolic space $X$. Consider a tame Markov chain $(w^{o}_n)$ on $G$, making linear progress in $X$ with exponential decay. Then there exist constants $L_1, C_2> 0$ such that

$$
\mathbb{P}\Big[ \tau(w_n^{o}) > L_1 n \Big] \geq 1- C_{2}e^{-n/C_{2}}$$
\end{theorem}

We will need the following lemma.

\begin{lemma}
\label{small_gromov_2}
There exist constants $M,C>0$ such that the following holds. For all $h \in G$ and $k\geq 1$:

$$
\mathbb P\big[ (x_0,(w^h_k)^2x_0)_{w^h_kx_0} \geq Mk\big] \leq Ce^{-k/C}
$$
\end{lemma}

In turn, in order to prove Lemma 
\ref{small_gromov_2} we will use the following (deterministic) lemma from \cite{MathieuSisto}. The notation $B^G$ and $\diam^X$ indicate balls in $G$ and diameter in the metric of $X$, respectively, for emphasis. We note that the lemma requires that the action of $G$ on the hyperbolic space $X$ is acylindrical.

\begin{lemma}\cite[Lemma 9.5]{MathieuSisto}
\label{diam}
There exist constants $C', D>0$ with the following property. For each $u,v \in G$ and $k \geq 0$, the set $B(u,v,k)$ of elements $s \in B^{G}(id,k)$ with $\diam^{X} \Big( [x_0,ux_0] \cap \mathcal N_{2 \delta}([usx_0,usvx_0])\Big) \geq D$ has cardinality at most $C'k^2$.
\end{lemma}

\begin{figure}[h]
\centering
\begin{tikzpicture}[scale=0.8]

\filldraw[black] (0,0) circle (1pt) node[anchor=east]{$x_0$};  

\filldraw[black] (1.5,7) circle (1pt) node[anchor=east]{$hx_0$};  
\filldraw[black] (3,7) circle (1pt) node[anchor=south]{$w^h_kx_0$};  
\filldraw[black] (8,3) circle (1pt) node[anchor=south]{$w^h_khx_0$}; 
\filldraw[black] (9,1.5) circle (1pt) node[anchor=east]{$(w^h_k)^2x_0$}; 

\draw[black] (0,0) to[out=60,in=300] (1.5,7);   
\draw[black] (0,0) to[out=40,in=280] node[pos=0.70] (mid1) {} node[pos=0.90] (mid2) {}  (3,7);

\draw[black] (3,7) to[out=290,in=170] (8,3);

\path [draw=blue,snake it] (1.5,7) -- (3,7);

\path [draw=blue,snake it] (8,3) -- (9,1.5);

\draw[blue] (9.3,2.5) node {$\leq kK $};

\draw[blue] (2,8) node {$\leq kK $};

\end{tikzpicture}
\caption{Points and geodesics relevant to the proof of Lemma \ref{small_gromov_2}.}
\end{figure}
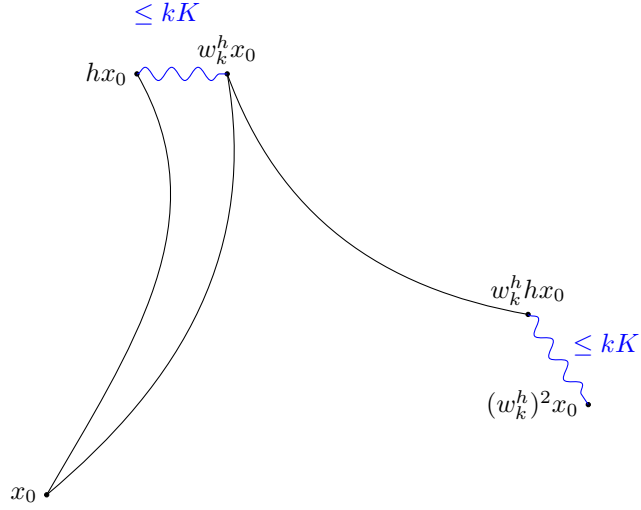

\begin{proof}[Proof of Lemma \ref{small_gromov_2}]
We let $K$ be a bound on the size of the jumps of the Markov chain in $X$ (see Remark \ref{rmk:bounded_jumps}), that is, $d_X(w^p_n,w^p_{n+1})\leq K$ for all $p$ and $n$.

First we note that if $(x_0,(w^h_k)^2x_0)_{w^h_kx_0} \geq Mk$ then $(x_0,w^h_k hx_0)_{w^h_kx_0}\geq Mk-Kk$, since $d_X(w^h_k hx_0,(w^h_k)^2x_0)\leq Kk$. In this case, two geodesics $[w^h_kx_0, x_0]$ and $[w^h_kx_0,w^h_k hx_0 ]$ will have initial subgeodesics of length $Mk-Kk-E$ that stay within $\delta$ of each other, where $\delta$ is a hyperbolicity constant for $X$ and $E$ only depends on $\delta$. Since $d_X(hx_0, w^h_kx_0)\leq Kk$, we further have that geodesics $[hx_0,x_0]$ and $[w^h_kx_0,w^h_k hx_0 ]$ have subgeodesics of length $Mk-2Kk-2E$ that stay within $2\delta$ of each other. Now, if $M$ was chosen large enough so that $Mk-2Kk-2E\geq D$, then by Lemma \ref{diam} we have that $h^{-1}w^h_k\in B(h,h, Kk)$ and $\# B(h,h,Kk) \leq C'(Kk)^2$. 

Hence, the non-amenability assumption from tameness of the Markov chain (\ref{defn:tame}-\eqref{item:non-amen}) gives the following, where $\rho <1$:
\begin{align*}
\begin{split}
	\mathbb P\Big[ (x_0,(w^h_k)^2x_0)_{w^h_kx_0} \geq Mk\Big] &\leq \# B(h,h,Kk) A \rho^k\\
	& \leq C'(Kk)^2A\rho^k \\
	& \leq Ce^{-k/C}
	\end{split}
\end{align*}
for some constant $C >0$, proving the lemma.
\end{proof}

\begin{lemma}
\label{smallgromov}
For all $\epsilon >0$, there exists a constant $C_{\epsilon}$ such that the following holds. For all $o \in G$ and $n$:

$$ \mathbb{P}\big[ (x_0,(w^{o}_n)^2  x_0)_{w^{o}_nx_0} \geq \epsilon n \big] \leq C_{\epsilon}e^{-n/ C_{\epsilon}}
$$
\end{lemma}

\begin{proof}
Fix $\epsilon >0$ and $n$. Let $k = \lfloor  \frac{\epsilon n}{M}\rfloor$ where $M $ is from Lemma \ref{small_gromov_2}.  We have

\begin{align*}
    \begin{split}
        \mathbb{P}\Big[ (x_0,(w^{o}_n)^2  x_0)_{w^{o}_nx_0} \geq \epsilon n \Big] &= \sum_{h \in G}\mathbb P \Big[ (x_0, (w^o_n)^2x_0)_{w^o_n x_0} \geq \epsilon n  \quad \Big\vert \quad w^{o}_{n-k}=h\Big]\mathbb P \Big[ w^{o}_{n-k}=h\Big] \\
        & \leq \sum_{h \in G}\mathbb P \Big[ (x_0, (w^h_k)^2x_0)_{w^h_k x_0} \geq \epsilon n \Big] \mathbb P \Big[ w^{o}_{n-k}=h\Big] \\
        & \leq \sum_{h \in G} \mathbb P \Big[  (x_0, (w^h_k)^2x_0)_{w^h_k x_0} \geq Mk \Big] \mathbb P \Big[ w^{o}_{n-k}=h\Big]\\
        & \leq Ce^{-k/C} \\
        & \leq C''e^{-n/C''}
    \end{split}
\end{align*}
where we use the Markov property from lemma \ref{lem:Markov_property_set} to go from the first to the second line. We then use Lemma \ref{small_gromov_2} and the fact that $ \frac{\epsilon n}{M} \geq k \geq \frac{\epsilon n}{M}-1$.
\end{proof}

We need another result, from \cite{maher2015random}.

\begin{lemma}\cite[Proposition 5.8]{maher2015random}
\label{lowerbound_translationlength}
There exists a constant $C_{1} > 0$, depending only on $\delta $ such that the following holds. For any isometry $g$ of the $\delta$-hyperbolic space $X$ the translation length of $g$ satisfies 
$$
\tau(g) \geq d_X(x_0,gx_0)-2(gx_0,g^{-1}x_0)_{x_0} - S \delta$$
for some constant $S$.
\end{lemma}

We now have everything in order to prove that with high probability, the translation length of a Markov element is at least linear in $n$, and in particular that loxodromic elements are generic.

\begin{proof}[Proof of Theorem \ref{translationlength}]
Let $L_1 < C$, where $C$ is the constant of linear progress with exponential decay. Fix the constant $S$ from Lemma \ref{lowerbound_translationlength}. Also, let $C_\epsilon$ be the constant from Lemma \ref{smallgromov} for $\epsilon = (C-L_1)/3>0$. The for all $n$ large enough (it suffices to consider this case) we have

 \begin{align*}
\begin{split}
\mathbb{P}\Big[ \tau(w_n) \leq L_1 n \Big] &\leq \mathbb{P}\big[ d_X(x_0,w_nx_0)-2(w_nx_0,w_n^{-1}x_0)_{x_0} -S\delta \leq L_1 n \big] \\
& \leq \mathbb{P}\big[ d_X(x_0,w_nx_0) <Ln \big]+ \mathbb{P}\big[ ((w^{o}_n)^2x_0,x_0)_{w_nx_0} \geq (L - L_1) n/2-S\delta /2\big] \\
&\leq Ce^{-n/C}+C_{\epsilon} e^{-n/C_{\epsilon}}, \\
\end{split}
\end{align*}

where we uses linear progress and Lemma \ref{smallgromov} in the last step. The latter term decays exponentially in $n$ so we are done.
\end{proof}

\bibliography{main}
\bibliographystyle{alpha}
\end{document}